\documentclass{article}
\usepackage[utf8]{inputenc}
\usepackage{float}[floatsintext]
\usepackage{blkarray}
\usepackage{multirow}

\usepackage{cancel}

\usepackage{placeins}

\usepackage[dvipsnames,table]{xcolor}

\usepackage[margin=1.56in]{geometry}

\usepackage{amsmath,amsfonts,amsthm,amssymb,esint,mathtools}
\usepackage{mathrsfs}

\newcommand\cd{\cdot} 
\newcommand\bp{\mathbf{p}} \newcommand\bs{\mathbf{s}}
\newcommand\bx{\mathbf{x}} \newcommand\by{\mathbf{y}}
\newcommand\bz{\mathbf{z}} \newcommand\bdp{\ensuremath{\mathbf{dp}}}
\newcommand\bds{\ensuremath{\mathbf{ds}}} \newcommand\bw{\mathbf{w}}

 \newcommand\oc{\mathcal{O}}
\newcommand\cls{\mathcal{S}} 
\newcommand\clh{\mathcal{H}} \newcommand\clsn{\mathcal{S}_n}
\newcommand\one{\ensuremath{-}}
\newcommand\blu[1]{\textcolor{blue}{#1}}

\newcommand\sfact[1]{\ensuremath{!_{#1}}}

\newcommand\disp{\displaystyle}

 \newcommand\ZZ{\mathbb Z}
 
\newcommand\calP{\mathcal{P}}

\newtheorem{theorem}{Theorem}[section]
\newtheorem{lemma}[theorem]{Lemma}
\newtheorem{corollary}[theorem]{Corollary}
\newtheorem{porism}[theorem]{Porism}
\newtheorem{proposition}[theorem]{Proposition}
\newtheorem{conjecture}[theorem]{Conjecture}

\theoremstyle{definition} \newtheorem{definition}[theorem]{Definition}

\newtheorem{example}[theorem]{Example}

\theoremstyle{remark} \newtheorem{remark}[theorem]{Remark}

\newtheorem{notation}[theorem]{Notation}

\usepackage{url}

\usepackage{latexsym}

\usepackage{mathtools}

\usepackage{tikz}
\usetikzlibrary{intersections,patterns,arrows,decorations.pathreplacing, calc}
\providecolor{myyellow}{rgb}{1.0, 0.70, 0.0}
\providecommand{\rankstep}{0.32cm}               
\providecommand{\plabelfont}{\ttfamily\scriptsize} 
\tikzset{
  dot/.style={circle, fill=blue, draw=none, inner sep=0pt, minimum size=6pt},
  rdot/.style={dot, fill=red},
  bigdot/.style={circle, fill=blue, draw=none, inner sep=0pt, minimum size=7pt},
  gdot/.style={bigdot, fill=myyellow},
  odot/.style={circle, fill=red, draw=none, inner sep=0pt, minimum size=6pt,
    append after command={\pgfextra{\draw[myyellow, line width=2pt]
      (\tikzlastnode.center) circle (5.5pt);}}},
  edge/.style={blue, very thick},
  plabel/.style={draw, rectangle, inner sep=4pt, font=\plabelfont,
    anchor=north, yshift=-8pt},
}

\numberwithin{equation}{section} \numberwithin{figure}{section}

\theoremstyle{remark}

\def\bino#1,#2,{\binom{#1}{#2}}

\newcommand\rd[1]{{\large\textcolor{red}{#1}}}
\newcommand\rdd[1]{{\textcolor{red}{#1}}}

\newcommand\OEIS{{\small\sc OEIS}} \newcommand\oeis{{\small\sc oeis}}

\newcommand\mcPP{{\mathcal{P}}_\pi}
\newcommand\mcPT{{\mathcal{P}}_\tau}
\newcommand\pp[1]{\ensuremath{\mcPP^#1}}
\newcommand\ppm{\ensuremath{\mcPP^m}}
\newcommand\ptm{\ensuremath{\mcPT^m}}

\newcommand\ppmj{\ensuremath{\mcPP^m(j)}}

\newcommand\ptmj{\ensuremath{\mcPT^m(j)}}
\newcommand\ptnj{\ensuremath{\mcPT^n(j)}}

\newcommand\vanish[1]{}

\newcommand\sgp{shifted grid poset} \newcommand\sgps{shifted grid
posets}

\newcommand\shift{{\mathsf{shift}}}

\newcommand\cla{\mathcal{A}} \newcommand\clp{\mathcal{P}}
\newcommand\clc{\mathcal{C}} \newcommand\clo{\mathcal{O}}
\newcommand\cll{\mathcal{L}} \newcommand\clw{\mathcal{W}}

\newcommand\defends{{\hfill$\square$}}

\usepackage{enumitem}

\usepackage[hidelinks]{hyperref}

\title{Permutations, Pattern Covers, Posets and Positivity: Part I: Patterns and
  Posets}

\author{%
  Natasha
  Blitvi\'c\thanks{\href{mailto:n.blitvic@qmul.ac.uk}{n.blitvic@qmul.ac.uk}.
    Partially supported by Leverhulme Trust Research Project Grant RPG-2020-103 and the Engineering and Physical Sciences Research Council (EPSRC) grant EP/V048902.}\\
  {\small School of Mathematical Sciences, Queen Mary University of London} \and
  Mohammed Slim
  Kammoun\thanks{\href{mailto:medslimkammoun@gmail.com}{medslimkammoun@gmail.com}.
    Partially supported by Leverhulme Trust Research Project Grant RPG-2020-103,
    CNRS-INSMI PEPS JCJC projects LAMPE and CAPRICES, LabEx CIMI
    (ANR-11-LABX-0040), and ANR project LOUCCOUM (ANR-24-CE40-7809).}\\
  {\small Université de Poitiers, CNRS, LMA} \and
  Einar Steingr\'imsson\thanks{\href{mailto:einar@alum.mit.edu}{einar@alum.mit.edu}}\\
  {\small Department of Mathematics and Statistics, University of Strathclyde}%
} \date{}

\begin{document}
\maketitle
\thispagestyle{empty}

\begin{abstract} We study permutations covered by occurrences of a
consecutive permutation pattern, with successive occurrences
overlapping in a fixed number of positions.  Several special cases of
this construction have been studied in the theory of consecutive
patterns and via the cluster method. Here we study the sequences
obtained by repeating a fixed overlap.  We elicit many structural
properties of such pattern covers, mainly in terms of posets whose
linear extensions are in bijection with the permutations in question.
The enumeration of these permutations for different patterns gives
rise to a great variety of sequences, many classical, but most of them
not found in the \OEIS.  We conjecture that, with minor modifications,
all of these sequences are moments of probability measures on the real
line.

  We develop a general enumerative and order-theoretic framework for
these objects. For overlap one, we give a generating function whose
coefficients enumerate permutations covered by an arbitrary fixed
pattern with any prescribed number of occurrences. For every overlap
at most half the length of the pattern, we derive a recursive formula
which, for fixed pattern and overlap, yields a polynomial-time
enumeration algorithm.  We then associate to each pattern, overlap,
and number of occurrences a naturally defined poset whose linear
extensions are in bijection with the corresponding pattern-covered
permutations.

The regularity of this poset construction makes possible a structural
analysis that in many cases leads to enumerative formulas for large
classes of patterns.  We prove uniform restrictions on the local
structure of the Hasse diagrams, describe their branches (pendant
chains) in terms of the corresponding patterns, and characterize when
the poset is a tree (seen as an undirected graph) or chain.

  We also construct patterns whose Hasse diagrams have a simple
grid-like structure. Some of these results yield explicit families of
formulas, explain unexpected equinumerations, and exhibit classes of
new enumerative phenomena in consecutive permutation patterns.

  In a forthcoming companion paper~\cite{BKS-patt-prob}, we
investigate the sequences arising here from a probabilistic
perspective, viewing them both as point processes and as moment
sequences. We also obtain integral and matrix representations for
important overlap $1$ and $2$ classes of pattern-covering sequences.
\end{abstract}

\tableofcontents

\section{Introduction}

Let $\cls_\ell$ be the set of permutations of the set
$\{1,2,\ldots,\ell\}$. Given permutations $\pi\in\cls_K$ and
$\sigma\in\cls_n$, an \emph{occurrence of the pattern} $\pi$ in
$\sigma$ is obtained by selecting entries of $\sigma$ whose relative
order is the same as that of $\pi$. If the selected entries are
required to occupy consecutive positions, then $\pi$ is called a
\emph{consecutive pattern}. The systematic study of permutation
patterns has its origins in work of Knuth and of Simion and Schmidt
\cite{Knuth1997,simion-schmidt}. For some background on consecutive
patterns, see \cite{genpatt-survey}.

\enlargethispage{1ex}
Two occurrences of a consecutive pattern containing some of the same entries of
a permutation impose order conditions on these entries, beyond what each
occurrence must satisfy. Thus the enumeration is not determined merely by the
length and structure of the pattern or by the number of occurrences: it depends
on how two copies of that particular pattern can fit together, at the beginning
of one and the end of the other. Cluster methods and related generating-function
techniques have provided the principal general framework for studying these
interactions \cite{Dotsenko-Khoroshkin-2013, Elizalde-Noy-2012,
  GouldenJackson1979, goulden-jackson-comb-enum, Liese-Remmel-2010}. The role of
overlap also underlies refined notions of consecutive \emph{Wilf equivalence}
(two patterns having the same enumeration) and the study of patterns two of
whose occurrences can overlap in at most one entry of a permutation
\cite{Duane-Remmel-2011,Dwyer-Elizalde-2018}.

In this paper, we require successive occurrences of a consecutive pattern in a
permutation~$\sigma$ to overlap in exactly $m$ entries for some given $m$, and
study the structure of such permutations and their enumeration.  Concretely, fix
a pattern $\pi\in\cls_K$, an overlap \mbox{$m<K$}, and $j\geq 1$, the prescribed
number of occurrences. We count permutations $\sigma\in\clsn$ containing
occurrences of $\pi$ beginning at
$$
1, 1+(K-m), 1+2(K-m), \ldots, 1+(j-1)(K-m),
$$
so that every two successive prescribed occurrences overlap in exactly
$m$ positions, and every entry of $\sigma$ belongs to some occurrence.
These conditions of course constrain the size of the permutation
$\sigma$ to be $n=K+(j-1)(K-m)$.  We call such corresponding
permutations \emph{pattern packings} and denote their count by
$N_\pi^m(j)$. (Note that we disregard potential occurrences outside
the prescribed ones.)

The individual numbers $N_\pi^m(j)$ are instances of the refined
cluster numbers studied by Dwyer and Elizalde
\cite{Dwyer-Elizalde-2018}, where occurrences may be prescribed at an
arbitrary set of starting positions.  Here we instead systematically
study the constant-overlap specialization in which the starting
positions form an arithmetic progression and the number of prescribed
occurrences varies.

Because the same overlap is repeated at every step, a packing with
$j+1$ occurrences is obtained from one with $j$ occurrences by
adjoining another copy of the same local configuration.  Although
adding a new copy in this way typically allows for rearrangement of
the previous entries, this highlights a recursive structure that can
in many cases be exploited for understanding of the global structure
and for enumeration of the permutations in question, even though it
does not make the enumeration uniformly simple. Even for short
patterns, the sequences
$$
\bigl(N_\pi^m(j)\bigr)_{j\geq 1}
$$
display strikingly different behavior. Some are given by Catalan or
Fuss--Catalan numbers, some by products of factorials or binomial
coefficients, and most appear not to belong to previously studied
families. Small changes in the pattern may alter the sequence
substantially, while quite different patterns may unexpectedly produce
the same sequence. The problem is therefore to understand how much of
this variation can be explained by general mechanisms, rather than by
a separate argument for each pattern.

The covering of permutations by occurrences of a consecutive pattern
partly generalizes the notion of the \emph{descent set}, which is the
set of starting positions of descents in a permutation, that is,
positions $i$ such that $\sigma(i)>\sigma(i+1)$, which is simply an
occurrence of the consecutive pattern 21.  In a seminal
paper~\cite{gessel-viennot-bino-dets} Gessel and Viennot gave a
formula, in terms of determinants, for the number of permutations in
$\clsn$ with any given descent set (a special case of the celebrated
Lindström-Gessel-Viennot Lemma). A general formula for the number of
permutations with any set of starting points for a given consecutive
pattern seems out of reach, given how hard the general problem is for
a single pattern occurring at regular intervals.  The only case we
present here that goes beyond regular intervals and overlaps involves
patterns whose occurrences are allowed to overlap in either one or two
places, see Proposition~\ref{prop-vary-olap}.

The first part of the paper develops such mechanisms on the
enumerative side. We establish symmetries that preserve the packing
numbers and identify further transformations specific to a fixed
overlap. For overlap $m=1$, where the enumeration depends only on the
first and last entry of a pattern, we express $N_\pi^1(j)$ as a
coefficient of a single multivariate formal power series.

For general overlap satisfying $m\leq K/2$, we derive a recurrence by analyzing
the effect of adding a new occurrence and how that allows for rearrangement of
the previous entries. For each such pair $(\pi,m)$ the recurrence gives an
algorithm for computing $N_\pi^m(j)$ as $j$ varies.  This algorithm is very
effective for $m=1$, but less so for greater $m$.

The second part of the paper develops an order-theoretic approach to
the same problem. Each occurrence of $\pi$ imposes a chain of
inequalities among the entries occupying its positions. Superimposing
the chains from the prescribed occurrences produces a finite poset,
which we call the \emph{pattern-packing poset} and denote by
$\ppm(j)$. Its linear extensions are naturally in bijection with the
corresponding pattern packings, so that
$$
N_\pi^m(j) = \left|\cll\bigl(\ppm(j)\bigr)\right|,
$$
where $\cll(\ppmj)$ is the set of linear extensions of $\ppmj$.

Special cases of these posets, encoding the inequalities imposed by
prescribed overlapping occurrences, arise in cluster constructions and
in the framework of
\cite{Dotsenko-Khoroshkin-2013,Dwyer-Elizalde-2018,Elizalde-Noy-2012}. Here
we treat the constant-overlap posets as a sequence indexed by $j$ and
exploit the repeated structure arising from the regularly spaced
starting positions of occurrences.

To make that structure visible, we introduce a \emph{stack}
representation in which each prescribed occurrence contributes a
translated copy of the same chain, and overlapping positions are
identified. Although the defining chains are individually simple,
their transitive closure can produce complicated structures. In
particular, repeated local constraints may create undirected cycles,
branching, interleaving, or grid-like regions in the Hasse
diagram. This is another manifestation of the central difficulty: the
local rule is specified by a single permutation, but its repeated
application can generate a much less transparent global order.

We prove several restrictions valid for every pattern-packing poset.
Each element is covered by at most two elements and covers at most two
elements, and the poset decomposes into a \emph{core} ---which is a
lattice bounded by elements that belong to every chain of maximal
length --- and (possibly) chain components attached to the core by
single edges. We then characterize several distinguished families
depending on the structure of their Hasse diagrams and the
corresponding undirected graph, which in many cases leads to
enumerative formulas. We give necessary and sufficient conditions for
the Hasse diagram to be a tree, and, when $m\leq K/2$, express these
conditions directly in terms of the prefix and suffix of the pattern.
We similarly characterize the case where the poset is a chain, and
construct a class of patterns whose packing posets form \emph{shifted
grids}, posets obtained from rectangular subsets of the integer
lattice $\ZZ^2$ by shifting $h$ steps to the right the unit rectangles
at height $h$. These classifications identify finite features of the
pattern that control the structure of the entire family
$$
\ppm(1),\ppm(2),\ppm(3),\ldots.
$$

The poset viewpoint also clarifies many equinumerations observed among
the enumerating sequences. Some arise from elementary symmetries of
the patterns, some from isomorphisms of the associated posets, and
some from two posets differing in ways that leave the number of linear
extensions invariant for trivial reasons. However, we also find pairs
of patterns with identical packing sequences despite having
nonisomorphic posets with visibly different structures, including
examples in which one Hasse diagram is a tree and the other contains
cycles, and even such pairs of patterns of different lengths and
overlaps. Such examples indicate that the constant-overlap setting has
a substantial equivalence theory that is not yet fully understood.

Wilf equivalences have been much studied for pattern avoidance.  In
our setting it seems to be much richer, since such equinumerations are
not restricted to patterns of the same length and the same overlap.
We present several such cases that are easily proved, but those only
seem to scratch the surface.

Some of the results presented here were presented in our extended
abstract~\cite{BKS-FPSAC-2023} for the conference {\sc{fpsac}} in
2023.

In a forthcoming companion paper~\cite{BKS-patt-prob}, we study the sequences
arising here as point processes and moment sequences. Using probabilistic
methods, we derive integral and matrix representations for important overlap-$1$
and overlap-$2$ families. We further conjecture, and for many families are able
to prove, that the pattern-packing \emph{uniformly} gives rise to moments of
probability measures on the real line. This not only identifies natural new
connections between permutation patterns and probability theory, but also opens
up more general approaches in a field of combinatorics where structure remains
pattern-dependent.

The paper is organized as follows. Section~\ref{sec-defs} formally
defines the relevant objects and establishes the notation. In
Section~\ref{sec-symmetries} we establish symmetries of the packing
numbers and discuss their relation to consecutive Wilf
equivalence. Section~\ref{sec-gf} contains the coefficient-extraction
formula for overlap one and the recurrence for overlaps at most half
the pattern length.  Section~\ref{sec-posets} introduces the
pattern-packing posets and develops their structural theory, including
the core-and-branch decomposition, the tree and chain criteria, and
the shifted-grid construction.  In Section~\ref{sec-vary-olap} we
mention one case of permutations covered by irregularly spaced
occurrences of a pattern, with overlaps either 1 or 2.  Finally,
Section~\ref{sec-open} presents open problems concerning further Wilf
equivalences, extremal packing numbers, and other structural
properties of pattern-packing posets than those we treat explicitly.

\section{First Definitions and Notation}\label{sec-defs}

As is habitual in combinatorics, we let $[n]:=\{1,\ldots,n\}$ and
denote by $\cls_n$ the set of all bijections $[n]\to[n]$. The set
$\cls_n$ forms a group (\emph{the symmetric group}) under functional
composition~$\circ$, with the identity permutation as the identity
element.

The study of permutation patterns is predominantly combinatorial, with
the group structure rarely playing a role. Results in this area are
typically formulated in the language of words, where $\clsn$ is the
set of permutations of $\{1,2,\ldots,n\}$, that is, the elements of
the symmetric group:

\begin{definition}
  [Permutations as words (the one-line notation)]\label{def-one-line} We
  interpret elements of $\clsn$ as words of length $n$ in the alphabet $[n]$
  formed of distinct letters. We write $\sigma=\sigma_1\ldots\sigma_n$ and refer
  to this as the \emph{one-line notation}. For $\sigma$ a word of length $n$ and
  $\sigma'$ a word of length $n'$, we denote by $\sigma\sigma'$ their
  \emph{concatenation}, that is, the word of length $n+n'$ whose first $n$
  letters are a copy of $\sigma$ and the remaining $n'$ letters are a copy of
  $\sigma'$.
\end{definition}
When talking about patterns and permutations in one-line notation we will
interchangeably talk about their \emph{letters} and \emph{elements}.

When necessary, we mix the two interpretations of permutations. For
example, $\sigma^{-1}_i$ denotes the $i$th letter, in the one-line
notation, of the functional inverse of $\sigma$.

\begin{definition}\label{def-order-iso} Given words $w$ and $w'$ of
  length $\ell$ on the alphabet $[n]$, we say that $w$ and $w'$ are \emph{order
    isomorphic}, denoted $w\sim w'$, if for every $i,j\in [\ell]$, $w_i<w_j$ if
  and only if $w_i'<w_j'$.

\label{def-order-isomorphism}
\end{definition} For example, 4936 is order isomorphic to 2413.

We now consider ways in which two permutations can be compared to one
another by inclusion.

\begin{definition}[Permutation Patterns]\label{def-perm-pattern} Fix
$K$ and $\pi\in\cls_K$. For any $\sigma\in \clsn $, where $n\geq K$,
we say that a tuple of indices $(i_1,\ldots,i_K)$, with $1\leq
i_1<i_2<\cdots<i_K\leq n$, forms an \emph{occurrence} of the
\emph{classical permutation pattern $\pi$} if the word formed by
restricting $\sigma$ to indices $i_1,\ldots,i_K$ is isomorphic to
$\pi$, that is, $\sigma_{i_1}\ldots\sigma_{i_K}\sim\pi$.

In contrast, for $(i_1,\ldots,i_K)$ to form an occurrence in $\sigma$
of the \emph{consecutive permutation pattern} $\pi$, the indices must
be consecutive, that is, $i_K=i_{K-1}+1=\cdots=i_1+K-1$, so that the
occurrence is a contiguous subword of $\sigma$.

The indices $i_1,\ldots,i_K$ are referred to as the \emph{positions}
or \emph{places} of the occurrence, with $i_1$ being the
\emph{starting position}.  \defends
\end{definition}

\begin{remark}\label{rem-vinc-patts} There is a more general inclusion
rule that interpolates between the classical and consecutive
patterns. These are the \emph{vincular patterns}~\cite{genpatt-survey}
(previously called \emph{generalized patterns}), which we don't
consider here.
\end{remark}

When a permutation $\sigma$ contains no occurrences of a pattern
$\pi$, we say that $\sigma$ \emph{avoids $\pi$} (or \emph{is an
avoider of $\pi$}). A central question in the study of permutation
patterns is this: Given a pattern $\pi$, how many permutations
$\sigma\in\cls_n$ avoid $\pi$ (for any given $n$)?  Here, instead, we
count permutations with occurrences of a given consecutive pattern
appearing at regular, overlapping intervals.

\subsection{Main objects}

Henceforth, 'pattern' will refer to a consecutive permutation
pattern. Furthermore, as the size of the pattern will recur in the
indexation and enumeration of various quantities, we assign it a
capital letter and make it an explicit definition:

\begin{definition}\label{def-k-patt} A \emph{$K$-pattern} is a
consecutive permutation pattern $\pi\in \cls_K$.
\end{definition}

The following is the main definition of this paper.
\begin{definition}[Pattern packing]\label{def-packing} Fix $K\ge2$,
$j\ge1$, $\pi$ a $K$-pattern, and $1\le m<K$. We say that a
permutation $\sigma\in\clsn$ \emph{$m$-packs $j$ occurrences of $\pi$}
or that $\sigma$ \emph{$(m,j)$-packs $\pi$} if $n=K+(j-1)(K-m)$ and
$\sigma$ contains occurrences of $\pi$ starting in positions
$\{1,1+d,1+2d,\ldots,1+(j-1)d\}$, where $d=K-m$. In other words,
\begin{equation}\label{def-stack-1}
\sigma_{1+(i-1)d}\ldots\sigma_{K+(i-1)d}\;\sim\;\pi,
\end{equation} where $\sim$ is order isomorphism and $i$ runs from 1
to $j$.  We also say that $\pi$ \emph{covers $\sigma$ with overlaps
$m$}, or that \emph{$\sigma$ is $m$-covered by $\pi$}, when the number
of occurrences need not be specified.
 \end{definition}
 
\begin{remark}\label{remark-autocorrelation}
  Definition~\ref{def-packing} does not preclude additional occurrences, outside
  of the regularly occurring ones.  For example, the permutation 192837465
  1-packs 2 occurrences of 15243 since it contains occurrences of 15243 starting
  in positions 1 and 5. But it also contains an occurrence of 15243 starting in
  position 3, which we do not prohibit. Whether it is possible to have
  additional occurrences of a pattern depends on the length of the pattern
  $\pi$, the length of the overlaps and the \emph{autocorrelation} of $\pi$,
  that is, the maximum possible overlap of two distinct occurrences of $\pi$ in
  a permutation.  In order for a $K$-pattern~$\pi$ to possibly occur in more
  places in a permutation than those prescribed by overlaps $m\geq 1$, $\pi$
  must have autocorrelation at least $(K+m)/2$. In the remainder of the paper,
  an 'occurrence' will always refer to one of the \emph{regular} occurrences,
  that is, those starting in positions $\{1,1+d,1+2d,\ldots,1+(j-1)d\}$, where
  $d=K-m$.
\end{remark}

\begin{remark}\label{rem-allow-gaps} The Main Definition focuses on
the case where the regular occurrences overlap (that is, $1\le m<K$).
However, allowing regular gaps between occurrences (which we can
interpret as letting $m<0$), or having occurrences adjacent without
overlap ($m=0$), simply adds trivially computed multiplicative factors
to any enumerative formula.
\end{remark}

\subsection{Symmetries and Wilf equivalences}

In pattern \emph{avoidance}, two patterns $\pi$ and $\pi'$ are said to
be \emph{Wilf equivalent} if, for all $n$, the subset of $\clsn$ that
avoids $\pi$ has the same size as that avoiding $\pi'$.  Some such
equivalences are proved simply by applying the same symmetry operation
on both a pattern and the permutations avoiding it, such as reversing
both, or taking complements, which involves replacing the letter $i$
in a permutation/pattern of length $n$ by $n+1-i$.

Both of the above mentioned operations also preserve enumeration in
the case of the pattern packing studied here, but there are a few
others that apply to all patterns given a fixed overlap.  For example,
if the overlap $m$ is at most $K/2$ for a $K$-pattern then
pattern-packing enumeration is invariant under the action of switching
the $m$-suffix and $m$-prefix (Lemma~\ref{lem-pref-suff-wilf}), as
well as permuting the $m$-prefix and $m$-suffix of a $K$-pattern in
the same way (Lemma~\ref{lem-overlap-pref-suff}).  In addition to
these, and some further general symmetries listed in
Section~\ref{sec-symmetries}, there seem to be myriad cases of
pattern/overlap combinations with the same enumeration without falling
under any of these general symmetries.  Such equinumerations occur
both when we fix pattern length and overlap length and for pairs of
patterns of different lengths and with different overlaps, whereas for
avoidance of single patterns, Wilf equivalence is only possible for
patterns of the same length.

In Section~\ref{sec-open} we mention several conjectured Wilf
equivalences that do not fall under any of the general symmetries we
exhibit in Section~\ref{sec-symmetries}.

\subsection{Pattern-packing posets and enumerating their linear
extensions}

In Section~\ref{sec-posets}, we show that the inequalities satisfied
by the elements of a permutation $m$-packing a pattern $\pi$ give rise
to a poset, which we term a \emph{pattern-packing poset} (not to be
confused with the posets in \cite{Elizalde-posets} and earlier papers
cited there). These posets have previously been described for certain
classes of patterns by Dotsenko and
Khoroshkin~\cite{Dotsenko-Khoroshkin-2013} and by Elizalde and
Noy~\cite{Elizalde-Noy-2012}.  While the structure of such a poset
depends on the underlying pattern, the pattern-packing posets enjoy
some general properties. For example, an element in such a poset can
cover (and be covered by) at most two elements. Furthermore, due to
the regularity of the underlying construction
(Definition~\ref{def-packing}), these posets themselves exhibit
regularity.  As a result, pattern-packing posets can be classified
according to the structure of their Hasse diagrams. In particular, we
characterize patterns that give rise to trees
(Theorem~\ref{thm-posets-trees}), and construct patterns realizing
multilevel grids (Theorem~\ref{thm-shifted-grid}).

For a given pattern $\pi$, the number of permutations $m$-packing $j$
occurrences of $\pi$ equals the number of linear extensions of a poset
defined in terms of $\pi$, $m$ and $j$. The regularity of the
pattern-packing posets enables us in some cases to explicitly count
their linear extensions, developing methods that we present in
Section~\ref{sec-posets}. The enumerative results take the form of:
\begin{itemize}
\item A polynomial-time algorithm, for any overlap $m\le
K/2$. (Theorem~\ref{thm-main-formula}.)
\item A generating function, whenever
$m=1$. (Theorem~\ref{thm-gf-1-packing}.)
\item Explicit enumeration formulas for specific families.
\item Various integral formulas and matrix product formulas for
general classes of patterns (to appear in our forthcoming companion
paper~\cite{BKS-patt-prob} to this one).

\end{itemize}

In Figures~\ref{fig-posets-zoo-1} and~\ref{fig-posets-zoo-2}, we show
examples of posets for a variety of patterns and overlaps, and in
Table~\ref{table-contain} we list the enumeration of all pattern-overlap combinations
for patterns of length at most 5, up to the trivial equivalences given
in Section~\ref{sec-symmetries}.  All of these where a formula is
known are special cases of more general results, to appear in
Section~\ref{sec-posets}.


\begin{figure}[p]
\centering
\vspace*{-13ex}
\def\rankstep{0.36cm}
\def\plabelfont{\ttfamily\scriptsize}
\noindent%
\makebox[0.2000\linewidth][c]{
\begin{tikzpicture}[x=\rankstep, y=\rankstep]
  \node[rdot] (4) at ( -0.02, 17.12) {};
  \node[rdot] (5) at ( -0.02, 16.00) {};
  \node[odot] (9) at ( -0.02, 15.00) {};
  \node[gdot] (10) at (  1.30, 14.00) {};
  \node[rdot] (3) at ( -1.48, 13.00) {};
  \node[gdot] (14) at (  1.51, 13.00) {};
  \node[odot] (8) at ( -1.48, 12.00) {};
  \node[dot] (15) at (  1.51, 12.00) {};
  \node[dot] (19) at (  1.51, 11.00) {};
  \node[gdot] (13) at ( -1.48, 10.00) {};
  \node[dot] (20) at (  1.51, 10.00) {};
  \node[dot] (24) at (  1.30,  9.00) {};
  \node[dot] (18) at ( -0.02,  8.00) {};
  \node[dot] (23) at ( -0.02,  7.00) {};
  \node[dot] (22) at ( -0.02,  6.00) {};
  \node[dot] (17) at ( -0.02,  5.00) {};
  \node[gdot] (12) at ( -1.51,  4.00) {};
  \node[dot] (21) at (  1.47,  4.00) {};
  \node[odot] (7) at ( -1.51,  3.00) {};
  \node[dot] (16) at (  1.47,  3.00) {};
  \node[rdot] (2) at ( -1.51,  2.00) {};
  \node[gdot] (11) at (  1.47,  2.00) {};
  \node[odot] (6) at ( -0.02,  1.00) {};
  \node[rdot] (1) at ( -0.02,  0.00) {};

  \draw[edge] (6) -- (1);
  \draw[edge] (7) -- (2);
  \draw[edge] (9) -- (3);
  \draw[edge] (4) -- (5);
  \draw[edge] (2) -- (6);
  \draw[edge] (11) -- (6);
  \draw[edge] (12) -- (7);
  \draw[edge] (3) -- (8);
  \draw[edge] (14) -- (8);
  \draw[edge] (5) -- (9);
  \draw[edge] (9) -- (10);
  \draw[edge] (7) -- (11);
  \draw[edge] (16) -- (11);
  \draw[edge] (17) -- (12);
  \draw[edge] (8) -- (13);
  \draw[edge] (19) -- (13);
  \draw[edge] (10) -- (14);
  \draw[edge] (14) -- (15);
  \draw[edge] (12) -- (16);
  \draw[edge] (21) -- (16);
  \draw[edge] (22) -- (17);
  \draw[edge] (13) -- (18);
  \draw[edge] (24) -- (18);
  \draw[edge] (15) -- (19);
  \draw[edge] (19) -- (20);
  \draw[edge] (17) -- (21);
  \draw[edge] (23) -- (22);
  \draw[edge] (18) -- (23);
  \draw[edge] (20) -- (24);

  \node[plabel] at (0,0) {136982457-4};
\end{tikzpicture}%
}%
\makebox[0.2000\linewidth][c]{
\begin{tikzpicture}[x=\rankstep, y=\rankstep]
  \node[dot] (30) at (  0.01, 23.12) {};
  \node[dot] (28) at (  0.01, 22.00) {};
  \node[dot] (27) at (  0.01, 21.00) {};
  \node[dot] (26) at (  0.01, 20.00) {};
  \node[dot] (29) at (  0.01, 19.00) {};
  \node[dot] (23) at (  0.01, 18.00) {};
  \node[dot] (21) at ( -2.01, 17.00) {};
  \node[dot] (25) at (  2.01, 17.00) {};
  \node[dot] (20) at ( -2.01, 16.00) {};
  \node[dot] (19) at ( -2.01, 15.00) {};
  \node[dot] (24) at (  2.01, 15.00) {};
  \node[dot] (22) at (  0.01, 14.00) {};
  \node[gdot] (16) at (  0.01, 13.00) {};
  \node[gdot] (14) at ( -2.01, 12.00) {};
  \node[dot] (18) at (  2.01, 12.00) {};
  \node[gdot] (13) at ( -2.01, 11.00) {};
  \node[gdot] (12) at ( -2.01, 10.00) {};
  \node[dot] (17) at (  2.01, 10.00) {};
  \node[gdot] (15) at (  0.01,  9.00) {};
  \node[odot] (9) at (  0.01,  8.00) {};
  \node[rdot] (7) at ( -2.01,  7.00) {};
  \node[gdot] (11) at (  2.01,  7.00) {};
  \node[rdot] (6) at ( -2.01,  6.00) {};
  \node[rdot] (5) at ( -2.01,  5.00) {};
  \node[gdot] (10) at (  2.01,  5.00) {};
  \node[odot] (8) at (  0.01,  4.00) {};
  \node[rdot] (2) at (  0.01,  3.00) {};
  \node[rdot] (4) at (  0.01,  2.00) {};
  \node[rdot] (3) at (  0.01,  1.00) {};
  \node[rdot] (1) at (  0.01,  0.00) {};

  \draw[edge] (3) -- (1);
  \draw[edge] (8) -- (2);
  \draw[edge] (4) -- (3);
  \draw[edge] (2) -- (4);
  \draw[edge] (6) -- (5);
  \draw[edge] (7) -- (6);
  \draw[edge] (9) -- (7);
  \draw[edge] (5) -- (8);
  \draw[edge] (10) -- (8);
  \draw[edge] (15) -- (9);
  \draw[edge] (11) -- (10);
  \draw[edge] (9) -- (11);
  \draw[edge] (13) -- (12);
  \draw[edge] (14) -- (13);
  \draw[edge] (16) -- (14);
  \draw[edge] (12) -- (15);
  \draw[edge] (17) -- (15);
  \draw[edge] (22) -- (16);
  \draw[edge] (18) -- (17);
  \draw[edge] (16) -- (18);
  \draw[edge] (20) -- (19);
  \draw[edge] (21) -- (20);
  \draw[edge] (23) -- (21);
  \draw[edge] (19) -- (22);
  \draw[edge] (24) -- (22);
  \draw[edge] (29) -- (23);
  \draw[edge] (25) -- (24);
  \draw[edge] (23) -- (25);
  \draw[edge] (27) -- (26);
  \draw[edge] (28) -- (27);
  \draw[edge] (30) -- (28);
  \draw[edge] (26) -- (29);

  \node[plabel] at (0,0) {142367859-2};
\end{tikzpicture}%
}%
\makebox[0.2000\linewidth][c]{
\begin{tikzpicture}[x=\rankstep, y=\rankstep]
  \node[dot] (35) at ( -0.10, 28.12) {};
  \node[dot] (32) at ( -0.10, 27.00) {};
  \node[dot] (34) at ( -0.10, 26.00) {};
  \node[dot] (31) at ( -0.10, 25.00) {};
  \node[dot] (30) at ( -0.10, 24.00) {};
  \node[dot] (27) at ( -0.10, 23.00) {};
  \node[dot] (33) at (  1.25, 22.00) {};
  \node[dot] (24) at ( -1.58, 21.00) {};
  \node[dot] (29) at (  1.46, 21.00) {};
  \node[dot] (26) at ( -0.06, 20.00) {};
  \node[dot] (23) at ( -1.21, 19.00) {};
  \node[dot] (22) at ( -1.54, 18.00) {};
  \node[dot] (28) at (  1.50, 18.00) {};
  \node[gdot] (19) at ( -1.54, 17.00) {};
  \node[dot] (25) at (  1.50, 16.00) {};
  \node[gdot] (16) at ( -1.54, 15.00) {};
  \node[dot] (21) at (  1.50, 15.00) {};
  \node[gdot] (18) at ( -0.01, 14.00) {};
  \node[gdot] (15) at ( -1.17, 13.00) {};
  \node[gdot] (14) at ( -1.50, 12.00) {};
  \node[dot] (20) at (  1.54, 12.00) {};
  \node[odot] (11) at ( -1.50, 11.00) {};
  \node[gdot] (17) at (  1.54, 10.00) {};
  \node[rdot] (8) at ( -1.50,  9.00) {};
  \node[gdot] (13) at (  1.54,  9.00) {};
  \node[odot] (10) at (  0.03,  8.00) {};
  \node[rdot] (7) at ( -1.13,  7.00) {};
  \node[rdot] (6) at ( -1.46,  6.00) {};
  \node[gdot] (12) at (  1.58,  6.00) {};
  \node[rdot] (3) at ( -1.13,  5.00) {};
  \node[odot] (9) at (  0.03,  4.00) {};
  \node[rdot] (5) at (  0.03,  3.00) {};
  \node[rdot] (2) at (  0.03,  2.00) {};
  \node[rdot] (4) at (  0.03,  1.00) {};
  \node[rdot] (1) at (  0.03,  0.00) {};

  \draw[edge] (4) -- (1);
  \draw[edge] (5) -- (2);
  \draw[edge] (6) -- (3);
  \draw[edge] (2) -- (4);
  \draw[edge] (9) -- (5);
  \draw[edge] (7) -- (6);
  \draw[edge] (10) -- (7);
  \draw[edge] (11) -- (8);
  \draw[edge] (3) -- (9);
  \draw[edge] (12) -- (9);
  \draw[edge] (8) -- (10);
  \draw[edge] (13) -- (10);
  \draw[edge] (14) -- (11);
  \draw[edge] (10) -- (12);
  \draw[edge] (17) -- (13);
  \draw[edge] (15) -- (14);
  \draw[edge] (18) -- (15);
  \draw[edge] (19) -- (16);
  \draw[edge] (11) -- (17);
  \draw[edge] (20) -- (17);
  \draw[edge] (16) -- (18);
  \draw[edge] (21) -- (18);
  \draw[edge] (22) -- (19);
  \draw[edge] (18) -- (20);
  \draw[edge] (25) -- (21);
  \draw[edge] (23) -- (22);
  \draw[edge] (26) -- (23);
  \draw[edge] (27) -- (24);
  \draw[edge] (19) -- (25);
  \draw[edge] (28) -- (25);
  \draw[edge] (24) -- (26);
  \draw[edge] (29) -- (26);
  \draw[edge] (30) -- (27);
  \draw[edge] (26) -- (28);
  \draw[edge] (33) -- (29);
  \draw[edge] (31) -- (30);
  \draw[edge] (34) -- (31);
  \draw[edge] (35) -- (32);
  \draw[edge] (27) -- (33);
  \draw[edge] (32) -- (34);

  \node[plabel] at (0,0) {1\,3\,6\,2\,4\,7\,8\,10\,5\,9\,11-3};
\end{tikzpicture}%
}%
\makebox[0.2000\linewidth][c]{
\begin{tikzpicture}[x=\rankstep, y=\rankstep]
  \node[dot] (25) at (  0.01, 22.12) {};
  \node[dot] (24) at (  0.01, 21.00) {};
  \node[dot] (20) at (  0.01, 20.00) {};
  \node[dot] (23) at (  0.01, 19.00) {};
  \node[dot] (26) at (  0.01, 18.00) {};
  \node[dot] (19) at (  0.01, 17.00) {};
  \node[gdot] (15) at ( -1.22, 16.00) {};
  \node[dot] (22) at (  1.22, 16.00) {};
  \node[dot] (18) at (  0.01, 15.00) {};
  \node[dot] (21) at (  0.01, 14.00) {};
  \node[gdot] (14) at (  0.01, 13.00) {};
  \node[odot] (10) at ( -1.22, 12.00) {};
  \node[dot] (17) at (  1.22, 12.00) {};
  \node[gdot] (13) at (  0.01, 11.00) {};
  \node[gdot] (16) at (  0.01, 10.00) {};
  \node[odot] (9) at (  0.01,  9.00) {};
  \node[rdot] (5) at ( -1.22,  8.00) {};
  \node[gdot] (12) at (  1.22,  8.00) {};
  \node[odot] (8) at (  0.01,  7.00) {};
  \node[odot] (11) at (  0.01,  6.00) {};
  \node[rdot] (4) at (  0.01,  5.00) {};
  \node[odot] (7) at (  0.01,  4.00) {};
  \node[rdot] (3) at (  0.01,  3.00) {};
  \node[odot] (6) at (  0.01,  2.00) {};
  \node[rdot] (2) at (  0.01,  1.00) {};
  \node[rdot] (1) at (  0.01,  0.00) {};

  \draw[edge] (2) -- (1);
  \draw[edge] (6) -- (2);
  \draw[edge] (7) -- (3);
  \draw[edge] (11) -- (4);
  \draw[edge] (9) -- (5);
  \draw[edge] (3) -- (6);
  \draw[edge] (4) -- (7);
  \draw[edge] (5) -- (8);
  \draw[edge] (12) -- (8);
  \draw[edge] (16) -- (9);
  \draw[edge] (14) -- (10);
  \draw[edge] (8) -- (11);
  \draw[edge] (9) -- (12);
  \draw[edge] (10) -- (13);
  \draw[edge] (17) -- (13);
  \draw[edge] (21) -- (14);
  \draw[edge] (19) -- (15);
  \draw[edge] (13) -- (16);
  \draw[edge] (14) -- (17);
  \draw[edge] (15) -- (18);
  \draw[edge] (22) -- (18);
  \draw[edge] (26) -- (19);
  \draw[edge] (24) -- (20);
  \draw[edge] (18) -- (21);
  \draw[edge] (19) -- (22);
  \draw[edge] (20) -- (23);
  \draw[edge] (25) -- (24);
  \draw[edge] (23) -- (26);

  \node[plabel] at (0,0) {1\,2\,4\,6\,9\,3\,5\,8\,10\,11\,7-6};
\end{tikzpicture}%
}%
\makebox[0.2000\linewidth][c]{
\begin{tikzpicture}[x=\rankstep, y=\rankstep]
  \node[dot] (25) at ( -0.04, 19.12) {};
  \node[dot] (20) at ( -0.04, 18.00) {};
  \node[dot] (24) at (  0.84, 17.00) {};
  \node[gdot] (15) at ( -1.01, 16.00) {};
  \node[dot] (26) at (  0.98, 16.00) {};
  \node[dot] (19) at (  0.98, 15.00) {};
  \node[odot] (10) at ( -1.01, 14.00) {};
  \node[dot] (21) at (  0.98, 14.00) {};
  \node[gdot] (14) at ( -0.23, 13.00) {};
  \node[dot] (23) at (  1.76, 13.00) {};
  \node[rdot] (5) at ( -1.76, 12.00) {};
  \node[gdot] (16) at (  0.23, 12.00) {};
  \node[odot] (9) at ( -1.76, 11.00) {};
  \node[dot] (22) at (  0.23, 11.00) {};
  \node[dot] (18) at (  0.10, 10.00) {};
  \node[odot] (11) at ( -0.79,  9.00) {};
  \node[dot] (17) at (  0.10,  8.00) {};
  \node[rdot] (4) at ( -1.76,  7.00) {};
  \node[gdot] (13) at (  0.23,  7.00) {};
  \node[odot] (6) at ( -0.76,  6.00) {};
  \node[gdot] (12) at ( -0.76,  5.00) {};
  \node[odot] (8) at ( -0.76,  4.00) {};
  \node[rdot] (1) at ( -0.76,  3.00) {};
  \node[odot] (7) at ( -0.76,  2.00) {};
  \node[rdot] (3) at ( -0.76,  1.00) {};
  \node[rdot] (2) at ( -0.76,  0.00) {};

  \draw[edge] (8) -- (1);
  \draw[edge] (3) -- (2);
  \draw[edge] (7) -- (3);
  \draw[edge] (11) -- (4);
  \draw[edge] (10) -- (5);
  \draw[edge] (4) -- (6);
  \draw[edge] (13) -- (6);
  \draw[edge] (1) -- (7);
  \draw[edge] (12) -- (8);
  \draw[edge] (5) -- (9);
  \draw[edge] (16) -- (9);
  \draw[edge] (15) -- (10);
  \draw[edge] (9) -- (11);
  \draw[edge] (18) -- (11);
  \draw[edge] (6) -- (12);
  \draw[edge] (17) -- (13);
  \draw[edge] (10) -- (14);
  \draw[edge] (21) -- (14);
  \draw[edge] (20) -- (15);
  \draw[edge] (14) -- (16);
  \draw[edge] (23) -- (16);
  \draw[edge] (11) -- (17);
  \draw[edge] (22) -- (18);
  \draw[edge] (15) -- (19);
  \draw[edge] (26) -- (19);
  \draw[edge] (25) -- (20);
  \draw[edge] (19) -- (21);
  \draw[edge] (16) -- (22);
  \draw[edge] (21) -- (23);
  \draw[edge] (20) -- (24);
  \draw[edge] (24) -- (26);

  \node[plabel] at (0,0) {4\,1\,2\,7\,10\,6\,3\,5\,9\,11\,8-6};
\end{tikzpicture}%
}%
\par\vspace*{1.5ex}

\noindent%
\makebox[0.2000\linewidth][c]{
\begin{tikzpicture}[x=\rankstep, y=\rankstep]
  \node[dot] (20) at ( -0.00, 15.12) {};
  \node[dot] (16) at ( -0.00, 14.00) {};
  \node[gdot] (12) at ( -0.90, 13.00) {};
  \node[dot] (22) at (  0.90, 13.00) {};
  \node[dot] (18) at (  0.90, 12.00) {};
  \node[odot] (8) at ( -0.90, 11.00) {};
  \node[dot] (21) at (  0.90, 11.00) {};
  \node[gdot] (14) at (  0.90, 10.00) {};
  \node[rdot] (4) at ( -0.90,  9.00) {};
  \node[dot] (17) at (  0.90,  9.00) {};
  \node[odot] (10) at ( -0.90,  8.00) {};
  \node[dot] (19) at (  0.90,  8.00) {};
  \node[gdot] (13) at ( -0.00,  7.00) {};
  \node[odot] (6) at ( -0.90,  6.00) {};
  \node[dot] (15) at (  0.90,  6.00) {};
  \node[odot] (9) at ( -0.00,  5.00) {};
  \node[rdot] (2) at ( -0.90,  4.00) {};
  \node[gdot] (11) at (  0.90,  4.00) {};
  \node[odot] (5) at ( -0.00,  3.00) {};
  \node[odot] (7) at ( -0.00,  2.00) {};
  \node[rdot] (1) at ( -0.00,  1.00) {};
  \node[rdot] (3) at ( -0.00,  0.00) {};

  \draw[edge] (7) -- (1);
  \draw[edge] (9) -- (2);
  \draw[edge] (1) -- (3);
  \draw[edge] (8) -- (4);
  \draw[edge] (2) -- (5);
  \draw[edge] (11) -- (5);
  \draw[edge] (13) -- (6);
  \draw[edge] (5) -- (7);
  \draw[edge] (12) -- (8);
  \draw[edge] (6) -- (9);
  \draw[edge] (15) -- (9);
  \draw[edge] (4) -- (10);
  \draw[edge] (17) -- (10);
  \draw[edge] (9) -- (11);
  \draw[edge] (16) -- (12);
  \draw[edge] (10) -- (13);
  \draw[edge] (19) -- (13);
  \draw[edge] (8) -- (14);
  \draw[edge] (21) -- (14);
  \draw[edge] (13) -- (15);
  \draw[edge] (20) -- (16);
  \draw[edge] (14) -- (17);
  \draw[edge] (12) -- (18);
  \draw[edge] (22) -- (18);
  \draw[edge] (17) -- (19);
  \draw[edge] (18) -- (21);
  \draw[edge] (16) -- (22);

  \node[plabel] at (0,0) {2\,5\,1\,9\,4\,7\,3\,10\,6\,8-6};
\end{tikzpicture}%
}%
\makebox[0.2000\linewidth][c]{
\begin{tikzpicture}[x=\rankstep, y=\rankstep]
  \node[rdot] (8) at ( -2.25, 25.12) {};
  \node[rdot] (3) at ( -2.25, 24.00) {};
  \node[rdot] (7) at ( -2.25, 23.00) {};
  \node[gdot] (16) at (  0.00, 23.00) {};
  \node[odot] (11) at ( -1.12, 22.00) {};
  \node[gdot] (15) at ( -1.12, 21.00) {};
  \node[dot] (24) at (  1.12, 21.00) {};
  \node[gdot] (19) at (  0.00, 20.00) {};
  \node[dot] (23) at (  0.00, 19.00) {};
  \node[dot] (32) at (  2.25, 19.00) {};
  \node[dot] (27) at (  1.12, 18.00) {};
  \node[dot] (31) at (  1.12, 17.00) {};
  \node[dot] (35) at (  1.12, 16.00) {};
  \node[dot] (34) at (  1.12, 15.00) {};
  \node[dot] (30) at (  1.12, 14.00) {};
  \node[dot] (33) at (  1.12, 13.00) {};
  \node[dot] (26) at (  1.12, 12.00) {};
  \node[dot] (22) at (  0.00, 11.00) {};
  \node[dot] (29) at (  2.25, 11.00) {};
  \node[dot] (25) at (  1.12, 10.00) {};
  \node[gdot] (18) at (  0.00,  9.00) {};
  \node[dot] (28) at (  2.25,  9.00) {};
  \node[gdot] (14) at ( -1.12,  8.00) {};
  \node[dot] (21) at (  1.12,  8.00) {};
  \node[gdot] (17) at (  0.00,  7.00) {};
  \node[odot] (10) at ( -1.12,  6.00) {};
  \node[dot] (20) at (  1.12,  6.00) {};
  \node[rdot] (6) at ( -2.25,  5.00) {};
  \node[gdot] (13) at (  0.00,  5.00) {};
  \node[odot] (9) at ( -1.12,  4.00) {};
  \node[rdot] (2) at ( -2.25,  3.00) {};
  \node[gdot] (12) at (  0.00,  3.00) {};
  \node[rdot] (5) at ( -2.25,  2.00) {};
  \node[rdot] (1) at ( -2.25,  1.00) {};
  \node[rdot] (4) at ( -2.25,  0.00) {};

  \draw[edge] (5) -- (1);
  \draw[edge] (9) -- (2);
  \draw[edge] (8) -- (3);
  \draw[edge] (1) -- (4);
  \draw[edge] (2) -- (5);
  \draw[edge] (10) -- (6);
  \draw[edge] (3) -- (7);
  \draw[edge] (6) -- (9);
  \draw[edge] (13) -- (9);
  \draw[edge] (17) -- (10);
  \draw[edge] (7) -- (11);
  \draw[edge] (16) -- (11);
  \draw[edge] (9) -- (12);
  \draw[edge] (10) -- (13);
  \draw[edge] (18) -- (14);
  \draw[edge] (11) -- (15);
  \draw[edge] (14) -- (17);
  \draw[edge] (21) -- (17);
  \draw[edge] (25) -- (18);
  \draw[edge] (15) -- (19);
  \draw[edge] (24) -- (19);
  \draw[edge] (17) -- (20);
  \draw[edge] (18) -- (21);
  \draw[edge] (26) -- (22);
  \draw[edge] (19) -- (23);
  \draw[edge] (22) -- (25);
  \draw[edge] (29) -- (25);
  \draw[edge] (33) -- (26);
  \draw[edge] (23) -- (27);
  \draw[edge] (32) -- (27);
  \draw[edge] (25) -- (28);
  \draw[edge] (26) -- (29);
  \draw[edge] (34) -- (30);
  \draw[edge] (27) -- (31);
  \draw[edge] (30) -- (33);
  \draw[edge] (35) -- (34);
  \draw[edge] (31) -- (35);

  \node[plabel] at (0,0) {2\,4\,10\,1\,3\,6\,9\,11\,5\,7\,8-3};
\end{tikzpicture}%
}%
\makebox[0.2000\linewidth][c]{
\begin{tikzpicture}[x=\rankstep, y=\rankstep]
  \node[rdot] (8) at ( -2.69, 19.12) {};
  \node[rdot] (7) at ( -2.69, 18.00) {};
  \node[gdot] (16) at ( -1.29, 18.00) {};
  \node[rdot] (3) at ( -2.69, 17.00) {};
  \node[gdot] (15) at ( -1.29, 17.00) {};
  \node[dot] (24) at (  0.11, 17.00) {};
  \node[odot] (11) at ( -1.29, 16.00) {};
  \node[dot] (23) at (  0.11, 16.00) {};
  \node[dot] (32) at (  1.53, 16.00) {};
  \node[gdot] (19) at (  0.11, 15.00) {};
  \node[dot] (31) at (  1.53, 15.00) {};
  \node[dot] (27) at (  0.82, 14.00) {};
  \node[dot] (35) at (  0.82, 13.00) {};
  \node[dot] (34) at (  0.82, 12.00) {};
  \node[dot] (30) at (  0.82, 11.00) {};
  \node[dot] (26) at (  0.82, 10.00) {};
  \node[dot] (22) at (  0.29,  9.00) {};
  \node[gdot] (18) at (  0.14,  8.00) {};
  \node[dot] (33) at (  1.54,  8.00) {};
  \node[gdot] (14) at (  0.14,  7.00) {};
  \node[dot] (25) at (  1.54,  7.00) {};
  \node[odot] (10) at (  0.14,  6.00) {};
  \node[dot] (29) at (  2.60,  6.00) {};
  \node[rdot] (6) at ( -0.11,  5.00) {};
  \node[gdot] (17) at (  1.29,  5.00) {};
  \node[dot] (28) at (  2.69,  5.00) {};
  \node[rdot] (2) at ( -0.11,  4.00) {};
  \node[dot] (21) at (  2.03,  4.00) {};
  \node[odot] (9) at (  0.25,  3.00) {};
  \node[dot] (20) at (  2.03,  3.00) {};
  \node[rdot] (1) at ( -0.44,  2.00) {};
  \node[gdot] (13) at (  0.96,  2.00) {};
  \node[rdot] (5) at ( -0.44,  1.00) {};
  \node[gdot] (12) at (  0.96,  1.00) {};
  \node[rdot] (4) at ( -0.44,  0.00) {};

  \draw[edge] (9) -- (1);
  \draw[edge] (6) -- (2);
  \draw[edge] (7) -- (3);
  \draw[edge] (5) -- (4);
  \draw[edge] (1) -- (5);
  \draw[edge] (10) -- (6);
  \draw[edge] (8) -- (7);
  \draw[edge] (2) -- (9);
  \draw[edge] (17) -- (9);
  \draw[edge] (14) -- (10);
  \draw[edge] (3) -- (11);
  \draw[edge] (15) -- (11);
  \draw[edge] (13) -- (12);
  \draw[edge] (9) -- (13);
  \draw[edge] (18) -- (14);
  \draw[edge] (16) -- (15);
  \draw[edge] (10) -- (17);
  \draw[edge] (25) -- (17);
  \draw[edge] (22) -- (18);
  \draw[edge] (11) -- (19);
  \draw[edge] (23) -- (19);
  \draw[edge] (21) -- (20);
  \draw[edge] (17) -- (21);
  \draw[edge] (26) -- (22);
  \draw[edge] (24) -- (23);
  \draw[edge] (18) -- (25);
  \draw[edge] (33) -- (25);
  \draw[edge] (30) -- (26);
  \draw[edge] (19) -- (27);
  \draw[edge] (31) -- (27);
  \draw[edge] (29) -- (28);
  \draw[edge] (25) -- (29);
  \draw[edge] (34) -- (30);
  \draw[edge] (32) -- (31);
  \draw[edge] (26) -- (33);
  \draw[edge] (35) -- (34);
  \draw[edge] (27) -- (35);

  \node[plabel] at (0,0) {3\,5\,9\,1\,2\,6\,10\,11\,4\,7\,8-3};
\end{tikzpicture}%
}%
\makebox[0.2000\linewidth][c]{
\begin{tikzpicture}[x=\rankstep, y=\rankstep]
  \node[dot] (25) at ( -0.33, 16.12) {};
  \node[dot] (20) at ( -0.33, 15.00) {};
  \node[gdot] (15) at ( -1.14, 14.00) {};
  \node[dot] (24) at (  0.47, 14.00) {};
  \node[dot] (19) at (  0.26, 13.00) {};
  \node[odot] (10) at ( -1.14, 12.00) {};
  \node[dot] (23) at (  1.50, 12.00) {};
  \node[rdot] (5) at ( -1.61, 11.00) {};
  \node[gdot] (14) at (  0.00, 11.00) {};
  \node[dot] (26) at (  1.61, 11.00) {};
  \node[odot] (9) at ( -0.81, 10.00) {};
  \node[dot] (18) at (  1.03, 10.00) {};
  \node[dot] (21) at (  1.03,  9.00) {};
  \node[rdot] (4) at ( -1.06,  8.00) {};
  \node[gdot] (13) at (  0.56,  8.00) {};
  \node[gdot] (16) at (  0.56,  7.00) {};
  \node[odot] (8) at ( -1.06,  6.00) {};
  \node[odot] (11) at ( -1.06,  5.00) {};
  \node[dot] (22) at (  0.56,  5.00) {};
  \node[rdot] (3) at ( -1.06,  4.00) {};
  \node[dot] (17) at (  0.56,  4.00) {};
  \node[odot] (6) at ( -1.06,  3.00) {};
  \node[rdot] (1) at ( -1.06,  2.00) {};
  \node[gdot] (12) at (  0.56,  2.00) {};
  \node[odot] (7) at ( -0.25,  1.00) {};
  \node[rdot] (2) at ( -0.25,  0.00) {};

  \draw[edge] (6) -- (1);
  \draw[edge] (7) -- (2);
  \draw[edge] (11) -- (3);
  \draw[edge] (9) -- (4);
  \draw[edge] (10) -- (5);
  \draw[edge] (3) -- (6);
  \draw[edge] (1) -- (7);
  \draw[edge] (12) -- (7);
  \draw[edge] (4) -- (8);
  \draw[edge] (16) -- (8);
  \draw[edge] (5) -- (9);
  \draw[edge] (14) -- (9);
  \draw[edge] (15) -- (10);
  \draw[edge] (8) -- (11);
  \draw[edge] (6) -- (12);
  \draw[edge] (17) -- (12);
  \draw[edge] (9) -- (13);
  \draw[edge] (21) -- (13);
  \draw[edge] (10) -- (14);
  \draw[edge] (19) -- (14);
  \draw[edge] (20) -- (15);
  \draw[edge] (13) -- (16);
  \draw[edge] (11) -- (17);
  \draw[edge] (22) -- (17);
  \draw[edge] (14) -- (18);
  \draw[edge] (26) -- (18);
  \draw[edge] (15) -- (19);
  \draw[edge] (24) -- (19);
  \draw[edge] (25) -- (20);
  \draw[edge] (18) -- (21);
  \draw[edge] (16) -- (22);
  \draw[edge] (19) -- (23);
  \draw[edge] (20) -- (24);
  \draw[edge] (23) -- (26);

  \node[plabel] at (0,0) {3\,1\,5\,8\,10\,4\,2\,7\,9\,11\,6-6};
\end{tikzpicture}%
}%
\makebox[0.2000\linewidth][c]{
\begin{tikzpicture}[x=\rankstep, y=\rankstep]
  \node[dot] (32) at ( -0.74, 22.12) {};
  \node[dot] (25) at ( -0.74, 21.00) {};
  \node[dot] (31) at (  0.10, 20.00) {};
  \node[dot] (28) at (  0.13, 19.00) {};
  \node[dot] (30) at (  0.19, 18.00) {};
  \node[gdot] (18) at ( -1.77, 17.00) {};
  \node[dot] (27) at (  0.34, 17.00) {};
  \node[dot] (24) at (  0.19, 16.00) {};
  \node[dot] (21) at (  0.08, 15.00) {};
  \node[dot] (26) at (  2.19, 15.00) {};
  \node[dot] (23) at (  1.37, 14.00) {};
  \node[odot] (11) at ( -1.77, 13.00) {};
  \node[dot] (20) at (  0.55, 13.00) {};
  \node[rdot] (4) at ( -2.37, 12.00) {};
  \node[gdot] (17) at (  0.16, 12.00) {};
  \node[gdot] (14) at ( -1.16, 11.00) {};
  \node[dot] (19) at (  0.95, 11.00) {};
  \node[gdot] (16) at ( -0.16, 10.00) {};
  \node[dot] (29) at (  2.37, 10.00) {};
  \node[gdot] (13) at ( -0.37,  9.00) {};
  \node[dot] (22) at (  1.76,  9.00) {};
  \node[odot] (10) at ( -1.30,  8.00) {};
  \node[rdot] (7) at ( -2.10,  7.00) {};
  \node[gdot] (12) at (  0.02,  7.00) {};
  \node[odot] (9) at ( -0.13,  6.00) {};
  \node[rdot] (6) at ( -0.37,  5.00) {};
  \node[gdot] (15) at (  1.76,  5.00) {};
  \node[rdot] (3) at ( -0.13,  4.00) {};
  \node[rdot] (5) at ( -0.01,  3.00) {};
  \node[rdot] (2) at (  0.05,  2.00) {};
  \node[odot] (8) at (  0.84,  1.00) {};
  \node[rdot] (1) at (  0.84,  0.00) {};

  \draw[edge] (8) -- (1);
  \draw[edge] (5) -- (2);
  \draw[edge] (6) -- (3);
  \draw[edge] (11) -- (4);
  \draw[edge] (3) -- (5);
  \draw[edge] (9) -- (6);
  \draw[edge] (10) -- (7);
  \draw[edge] (2) -- (8);
  \draw[edge] (15) -- (8);
  \draw[edge] (7) -- (9);
  \draw[edge] (12) -- (9);
  \draw[edge] (4) -- (10);
  \draw[edge] (13) -- (10);
  \draw[edge] (18) -- (11);
  \draw[edge] (10) -- (12);
  \draw[edge] (16) -- (13);
  \draw[edge] (17) -- (14);
  \draw[edge] (9) -- (15);
  \draw[edge] (22) -- (15);
  \draw[edge] (14) -- (16);
  \draw[edge] (19) -- (16);
  \draw[edge] (11) -- (17);
  \draw[edge] (20) -- (17);
  \draw[edge] (25) -- (18);
  \draw[edge] (17) -- (19);
  \draw[edge] (23) -- (20);
  \draw[edge] (24) -- (21);
  \draw[edge] (16) -- (22);
  \draw[edge] (29) -- (22);
  \draw[edge] (21) -- (23);
  \draw[edge] (26) -- (23);
  \draw[edge] (18) -- (24);
  \draw[edge] (27) -- (24);
  \draw[edge] (32) -- (25);
  \draw[edge] (24) -- (26);
  \draw[edge] (30) -- (27);
  \draw[edge] (31) -- (28);
  \draw[edge] (23) -- (29);
  \draw[edge] (28) -- (30);
  \draw[edge] (25) -- (31);

  \node[plabel] at (0,0) {1\,3\,5\,10\,4\,6\,8\,2\,7\,9\,11-4};
\end{tikzpicture}%
}%
\par

\caption{Posets for 4 occurrences of various patterns, pattern-overlap in box
  under each one. Red nodes are first occurrence, yellow second
  occurrence, circled nodes in both.}
\label{fig-posets-zoo-1}
\end{figure}
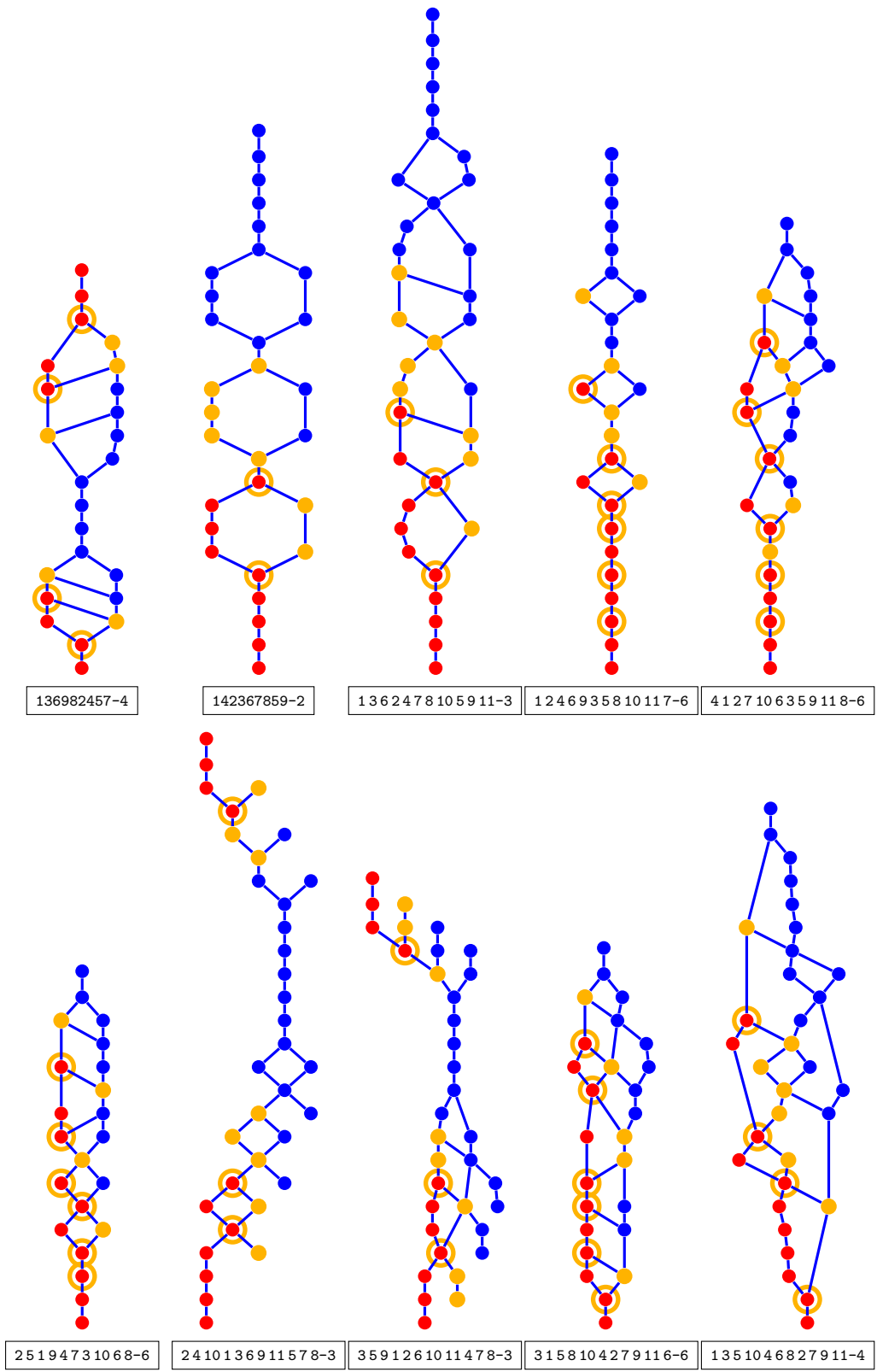


\begin{figure}[p]
\centering
\vspace*{-12ex}
\def\rankstep{0.385cm}
\def\plabelfont{\ttfamily\footnotesize}

\noindent%
\makebox[0.3333\linewidth][c]{
\begin{tikzpicture}[x=\rankstep, y=\rankstep]
  \node[rdot] (3) at (  1.73, 12.12) {};
  \node[odot] (7) at (  1.73, 11.00) {};
  \node[gdot] (11) at (  1.73, 10.00) {};
  \node[dot] (15) at (  1.73,  9.00) {};
  \node[dot] (19) at (  1.73,  8.00) {};
  \node[dot] (18) at (  1.73,  7.00) {};
  \node[dot] (14) at (  1.73,  6.00) {};
  \node[gdot] (10) at (  0.58,  5.00) {};
  \node[dot] (16) at (  2.90,  5.00) {};
  \node[odot] (6) at ( -1.73,  4.00) {};
  \node[dot] (12) at (  0.58,  4.00) {};
  \node[dot] (17) at (  2.90,  4.00) {};
  \node[rdot] (2) at ( -2.90,  3.00) {};
  \node[gdot] (8) at ( -0.58,  3.00) {};
  \node[dot] (13) at (  1.73,  3.00) {};
  \node[rdot] (4) at ( -2.90,  2.00) {};
  \node[gdot] (9) at ( -0.58,  2.00) {};
  \node[odot] (5) at ( -1.73,  1.00) {};
  \node[rdot] (1) at ( -1.73,  0.00) {};

  \draw[edge] (5) -- (1);
  \draw[edge] (6) -- (2);
  \draw[edge] (2) -- (4);
  \draw[edge] (4) -- (5);
  \draw[edge] (9) -- (5);
  \draw[edge] (10) -- (6);
  \draw[edge] (3) -- (7);
  \draw[edge] (6) -- (8);
  \draw[edge] (8) -- (9);
  \draw[edge] (13) -- (9);
  \draw[edge] (14) -- (10);
  \draw[edge] (7) -- (11);
  \draw[edge] (10) -- (12);
  \draw[edge] (12) -- (13);
  \draw[edge] (17) -- (13);
  \draw[edge] (18) -- (14);
  \draw[edge] (11) -- (15);
  \draw[edge] (14) -- (16);
  \draw[edge] (16) -- (17);
  \draw[edge] (19) -- (18);
  \draw[edge] (15) -- (19);

  \node[plabel] at (0,0) {1473256-3};
\end{tikzpicture}%
}%
\makebox[0.3333\linewidth][c]{
\begin{tikzpicture}[x=\rankstep, y=\rankstep]
  \node[rdot] (7) at ( -2.62, 18.12) {};
  \node[rdot] (6) at ( -2.62, 17.00) {};
  \node[gdot] (14) at ( -1.19, 17.00) {};
  \node[rdot] (3) at ( -2.62, 16.00) {};
  \node[gdot] (13) at ( -1.19, 16.00) {};
  \node[dot] (21) at (  0.24, 16.00) {};
  \node[odot] (10) at ( -1.19, 15.00) {};
  \node[dot] (20) at (  0.24, 15.00) {};
  \node[dot] (28) at (  1.69, 15.00) {};
  \node[rdot] (2) at ( -1.19, 14.00) {};
  \node[gdot] (17) at (  0.24, 14.00) {};
  \node[dot] (27) at (  1.69, 14.00) {};
  \node[odot] (9) at ( -0.47, 13.00) {};
  \node[dot] (24) at (  1.24, 13.00) {};
  \node[gdot] (16) at (  0.09, 12.00) {};
  \node[dot] (31) at (  1.52, 12.00) {};
  \node[dot] (23) at (  1.20, 11.00) {};
  \node[dot] (30) at (  2.34, 10.00) {};
  \node[rdot] (5) at ( -1.01,  9.00) {};
  \node[dot] (26) at (  2.42,  9.00) {};
  \node[gdot] (12) at (  0.09,  8.00) {};
  \node[dot] (29) at (  2.58,  8.00) {};
  \node[dot] (19) at (  1.19,  7.00) {};
  \node[dot] (25) at (  2.62,  7.00) {};
  \node[dot] (22) at (  1.19,  6.00) {};
  \node[dot] (18) at (  1.19,  5.00) {};
  \node[gdot] (15) at (  0.09,  4.00) {};
  \node[gdot] (11) at (  0.09,  3.00) {};
  \node[odot] (8) at ( -0.47,  2.00) {};
  \node[rdot] (4) at ( -0.47,  1.00) {};
  \node[rdot] (1) at ( -0.47,  0.00) {};

  \draw[edge] (4) -- (1);
  \draw[edge] (10) -- (2);
  \draw[edge] (6) -- (3);
  \draw[edge] (8) -- (4);
  \draw[edge] (9) -- (5);
  \draw[edge] (7) -- (6);
  \draw[edge] (5) -- (8);
  \draw[edge] (11) -- (8);
  \draw[edge] (2) -- (9);
  \draw[edge] (17) -- (9);
  \draw[edge] (3) -- (10);
  \draw[edge] (13) -- (10);
  \draw[edge] (15) -- (11);
  \draw[edge] (16) -- (12);
  \draw[edge] (14) -- (13);
  \draw[edge] (12) -- (15);
  \draw[edge] (18) -- (15);
  \draw[edge] (9) -- (16);
  \draw[edge] (24) -- (16);
  \draw[edge] (10) -- (17);
  \draw[edge] (20) -- (17);
  \draw[edge] (22) -- (18);
  \draw[edge] (23) -- (19);
  \draw[edge] (21) -- (20);
  \draw[edge] (19) -- (22);
  \draw[edge] (25) -- (22);
  \draw[edge] (16) -- (23);
  \draw[edge] (31) -- (23);
  \draw[edge] (17) -- (24);
  \draw[edge] (27) -- (24);
  \draw[edge] (29) -- (25);
  \draw[edge] (30) -- (26);
  \draw[edge] (28) -- (27);
  \draw[edge] (26) -- (29);
  \draw[edge] (23) -- (30);
  \draw[edge] (24) -- (31);

  \node[plabel] at (0,0) {1\,6\,8\,2\,4\,9\,10\,3\,5\,7-3};
\end{tikzpicture}%
}%
\makebox[0.3333\linewidth][c]{
\begin{tikzpicture}[x=\rankstep, y=\rankstep]
  \node[dot] (28) at (  3.97, 22.12) {};
  \node[dot] (21) at (  1.35, 21.00) {};
  \node[dot] (32) at (  3.97, 21.00) {};
  \node[gdot] (14) at (  0.03, 20.00) {};
  \node[dot] (25) at (  2.66, 20.00) {};
  \node[rdot] (7) at ( -1.97, 19.00) {};
  \node[gdot] (18) at (  0.66, 19.00) {};
  \node[odot] (11) at ( -1.67, 18.00) {};
  \node[rdot] (4) at ( -3.49, 17.00) {};
  \node[rdot] (6) at ( -3.49, 16.00) {};
  \node[dot] (27) at (  2.66, 16.00) {};
  \node[rdot] (3) at ( -3.78, 15.00) {};
  \node[dot] (20) at (  0.66, 15.00) {};
  \node[rdot] (5) at ( -3.97, 14.00) {};
  \node[gdot] (13) at ( -1.34, 14.00) {};
  \node[odot] (10) at ( -1.34, 13.00) {};
  \node[gdot] (12) at ( -1.34, 12.00) {};
  \node[gdot] (17) at (  0.66, 11.00) {};
  \node[dot] (19) at (  0.66, 10.00) {};
  \node[dot] (24) at (  1.65,  9.00) {};
  \node[dot] (26) at (  1.65,  8.00) {};
  \node[dot] (31) at (  1.65,  7.00) {};
  \node[dot] (30) at (  1.65,  6.00) {};
  \node[dot] (23) at (  1.65,  5.00) {};
  \node[gdot] (16) at (  0.33,  4.00) {};
  \node[dot] (29) at (  2.95,  4.00) {};
  \node[odot] (9) at (  0.33,  3.00) {};
  \node[dot] (22) at (  2.95,  3.00) {};
  \node[rdot] (2) at (  0.33,  2.00) {};
  \node[gdot] (15) at (  2.95,  2.00) {};
  \node[odot] (8) at (  1.65,  1.00) {};
  \node[rdot] (1) at (  1.65,  0.00) {};

  \draw[edge] (8) -- (1);
  \draw[edge] (9) -- (2);
  \draw[edge] (6) -- (3);
  \draw[edge] (11) -- (4);
  \draw[edge] (3) -- (5);
  \draw[edge] (4) -- (6);
  \draw[edge] (2) -- (8);
  \draw[edge] (15) -- (8);
  \draw[edge] (16) -- (9);
  \draw[edge] (5) -- (10);
  \draw[edge] (13) -- (10);
  \draw[edge] (7) -- (11);
  \draw[edge] (18) -- (11);
  \draw[edge] (10) -- (12);
  \draw[edge] (11) -- (13);
  \draw[edge] (9) -- (15);
  \draw[edge] (22) -- (15);
  \draw[edge] (23) -- (16);
  \draw[edge] (12) -- (17);
  \draw[edge] (20) -- (17);
  \draw[edge] (14) -- (18);
  \draw[edge] (25) -- (18);
  \draw[edge] (17) -- (19);
  \draw[edge] (18) -- (20);
  \draw[edge] (16) -- (22);
  \draw[edge] (29) -- (22);
  \draw[edge] (30) -- (23);
  \draw[edge] (19) -- (24);
  \draw[edge] (27) -- (24);
  \draw[edge] (21) -- (25);
  \draw[edge] (32) -- (25);
  \draw[edge] (24) -- (26);
  \draw[edge] (25) -- (27);
  \draw[edge] (23) -- (29);
  \draw[edge] (31) -- (30);
  \draw[edge] (26) -- (31);
  \draw[edge] (28) -- (32);

  \node[plabel] at (0,0) {1\,3\,7\,9\,6\,8\,11\,2\,4\,5\,10-4};
\end{tikzpicture}%
}%
\par\vspace*{5ex}

\noindent%
\makebox[0.3333\linewidth][c]{
\begin{tikzpicture}[x=\rankstep, y=\rankstep]
  \node[rdot] (4) at ( -1.78, 19.12) {};
  \node[odot] (11) at ( -1.78, 18.00) {};
  \node[rdot] (7) at ( -2.56, 17.00) {};
  \node[gdot] (18) at ( -0.99, 17.00) {};
  \node[rdot] (3) at ( -2.56, 16.00) {};
  \node[gdot] (14) at ( -0.99, 16.00) {};
  \node[dot] (25) at (  0.58, 16.00) {};
  \node[odot] (10) at ( -1.78, 15.00) {};
  \node[dot] (21) at ( -0.19, 15.00) {};
  \node[dot] (32) at (  1.38, 15.00) {};
  \node[gdot] (17) at ( -0.22, 14.00) {};
  \node[dot] (28) at (  1.35, 14.00) {};
  \node[dot] (24) at (  0.99, 13.00) {};
  \node[dot] (20) at (  0.99, 12.00) {};
  \node[dot] (31) at (  2.56, 12.00) {};
  \node[rdot] (6) at ( -1.51, 11.00) {};
  \node[dot] (27) at (  2.36, 11.00) {};
  \node[gdot] (13) at ( -0.22, 10.00) {};
  \node[dot] (30) at (  2.26, 10.00) {};
  \node[dot] (23) at (  0.99,  9.00) {};
  \node[gdot] (16) at ( -0.25,  8.00) {};
  \node[dot] (29) at (  1.33,  8.00) {};
  \node[odot] (9) at ( -1.29,  7.00) {};
  \node[dot] (26) at (  1.24,  7.00) {};
  \node[dot] (22) at (  0.50,  6.00) {};
  \node[rdot] (2) at ( -1.90,  5.00) {};
  \node[dot] (19) at (  0.50,  5.00) {};
  \node[gdot] (15) at ( -0.71,  4.00) {};
  \node[gdot] (12) at ( -0.71,  3.00) {};
  \node[odot] (8) at ( -1.32,  2.00) {};
  \node[rdot] (5) at ( -1.32,  1.00) {};
  \node[rdot] (1) at ( -1.32,  0.00) {};

  \draw[edge] (5) -- (1);
  \draw[edge] (9) -- (2);
  \draw[edge] (7) -- (3);
  \draw[edge] (8) -- (5);
  \draw[edge] (10) -- (6);
  \draw[edge] (11) -- (7);
  \draw[edge] (2) -- (8);
  \draw[edge] (12) -- (8);
  \draw[edge] (6) -- (9);
  \draw[edge] (16) -- (9);
  \draw[edge] (3) -- (10);
  \draw[edge] (14) -- (10);
  \draw[edge] (4) -- (11);
  \draw[edge] (15) -- (12);
  \draw[edge] (17) -- (13);
  \draw[edge] (18) -- (14);
  \draw[edge] (9) -- (15);
  \draw[edge] (19) -- (15);
  \draw[edge] (13) -- (16);
  \draw[edge] (23) -- (16);
  \draw[edge] (10) -- (17);
  \draw[edge] (21) -- (17);
  \draw[edge] (11) -- (18);
  \draw[edge] (22) -- (19);
  \draw[edge] (24) -- (20);
  \draw[edge] (25) -- (21);
  \draw[edge] (16) -- (22);
  \draw[edge] (26) -- (22);
  \draw[edge] (20) -- (23);
  \draw[edge] (30) -- (23);
  \draw[edge] (17) -- (24);
  \draw[edge] (28) -- (24);
  \draw[edge] (18) -- (25);
  \draw[edge] (29) -- (26);
  \draw[edge] (31) -- (27);
  \draw[edge] (32) -- (28);
  \draw[edge] (23) -- (29);
  \draw[edge] (27) -- (30);
  \draw[edge] (24) -- (31);
  \draw[edge] (25) -- (32);

  \node[plabel] at (0,0) {1\,4\,8\,11\,2\,6\,9\,3\,5\,7\,10-4};
\end{tikzpicture}%
}%
\makebox[0.3333\linewidth][c]{
\begin{tikzpicture}[x=\rankstep, y=\rankstep]
  \node[rdot] (4) at ( -0.85, 25.12) {};
  \node[odot] (11) at ( -0.85, 24.00) {};
  \node[gdot] (18) at ( -0.01, 23.00) {};
  \node[dot] (25) at (  1.67, 22.00) {};
  \node[dot] (21) at (  1.67, 21.00) {};
  \node[dot] (32) at (  3.88, 21.00) {};
  \node[rdot] (7) at ( -1.69, 20.00) {};
  \node[dot] (28) at (  3.73, 20.00) {};
  \node[gdot] (14) at ( -0.01, 19.00) {};
  \node[dot] (31) at (  3.60, 19.00) {};
  \node[dot] (24) at (  1.67, 18.00) {};
  \node[gdot] (17) at ( -0.01, 17.00) {};
  \node[odot] (10) at ( -1.69, 16.00) {};
  \node[rdot] (3) at ( -3.88, 15.00) {};
  \node[gdot] (13) at ( -1.69, 15.00) {};
  \node[rdot] (6) at ( -3.73, 14.00) {};
  \node[dot] (27) at (  1.67, 14.00) {};
  \node[rdot] (2) at ( -3.48, 13.00) {};
  \node[dot] (20) at ( -0.01, 13.00) {};
  \node[odot] (9) at ( -1.69, 12.00) {};
  \node[gdot] (16) at ( -0.01, 11.00) {};
  \node[dot] (23) at (  0.83, 10.00) {};
  \node[dot] (30) at (  0.83,  9.00) {};
  \node[dot] (29) at (  0.83,  8.00) {};
  \node[dot] (26) at (  0.83,  7.00) {};
  \node[dot] (22) at (  0.83,  6.00) {};
  \node[dot] (19) at (  0.83,  5.00) {};
  \node[gdot] (15) at (  0.83,  4.00) {};
  \node[gdot] (12) at (  0.83,  3.00) {};
  \node[odot] (8) at (  0.83,  2.00) {};
  \node[rdot] (5) at (  0.83,  1.00) {};
  \node[rdot] (1) at (  0.83,  0.00) {};

  \draw[edge] (5) -- (1);
  \draw[edge] (6) -- (2);
  \draw[edge] (10) -- (3);
  \draw[edge] (8) -- (5);
  \draw[edge] (3) -- (6);
  \draw[edge] (11) -- (7);
  \draw[edge] (12) -- (8);
  \draw[edge] (2) -- (9);
  \draw[edge] (13) -- (9);
  \draw[edge] (7) -- (10);
  \draw[edge] (17) -- (10);
  \draw[edge] (4) -- (11);
  \draw[edge] (15) -- (12);
  \draw[edge] (10) -- (13);
  \draw[edge] (18) -- (14);
  \draw[edge] (19) -- (15);
  \draw[edge] (9) -- (16);
  \draw[edge] (20) -- (16);
  \draw[edge] (14) -- (17);
  \draw[edge] (24) -- (17);
  \draw[edge] (11) -- (18);
  \draw[edge] (22) -- (19);
  \draw[edge] (17) -- (20);
  \draw[edge] (25) -- (21);
  \draw[edge] (26) -- (22);
  \draw[edge] (16) -- (23);
  \draw[edge] (27) -- (23);
  \draw[edge] (21) -- (24);
  \draw[edge] (31) -- (24);
  \draw[edge] (18) -- (25);
  \draw[edge] (29) -- (26);
  \draw[edge] (24) -- (27);
  \draw[edge] (32) -- (28);
  \draw[edge] (30) -- (29);
  \draw[edge] (23) -- (30);
  \draw[edge] (28) -- (31);
  \draw[edge] (25) -- (32);

  \node[plabel] at (0,0) {1\,5\,7\,11\,2\,6\,9\,3\,4\,8\,10-4};
\end{tikzpicture}%
}%
\makebox[0.3333\linewidth][c]{
\begin{tikzpicture}[x=\rankstep, y=\rankstep]
  \node[dot] (29) at ( -1.62, 22.12) {};
  \node[dot] (23) at ( -1.62, 21.00) {};
  \node[gdot] (17) at ( -2.59, 20.00) {};
  \node[dot] (28) at ( -0.63, 20.00) {};
  \node[odot] (11) at ( -2.59, 19.00) {};
  \node[dot] (22) at ( -0.63, 19.00) {};
  \node[rdot] (5) at ( -2.59, 18.00) {};
  \node[gdot] (16) at ( -0.63, 18.00) {};
  \node[odot] (10) at ( -1.62, 17.00) {};
  \node[rdot] (4) at ( -1.62, 16.00) {};
  \node[rdot] (3) at ( -1.62, 15.00) {};
  \node[odot] (9) at ( -1.62, 14.00) {};
  \node[gdot] (15) at ( -0.85, 13.00) {};
  \node[dot] (21) at (  0.63, 12.00) {};
  \node[dot] (27) at (  2.19, 11.00) {};
  \node[rdot] (6) at ( -2.34, 10.00) {};
  \node[dot] (24) at (  2.30, 10.00) {};
  \node[gdot] (12) at ( -0.85,  9.00) {};
  \node[dot] (26) at (  2.51,  9.00) {};
  \node[dot] (18) at (  0.63,  8.00) {};
  \node[dot] (25) at (  2.59,  8.00) {};
  \node[dot] (20) at (  0.63,  7.00) {};
  \node[dot] (19) at (  0.63,  6.00) {};
  \node[gdot] (14) at ( -0.85,  5.00) {};
  \node[gdot] (13) at ( -0.85,  4.00) {};
  \node[odot] (8) at ( -1.62,  3.00) {};
  \node[odot] (7) at ( -1.62,  2.00) {};
  \node[rdot] (2) at ( -1.62,  1.00) {};
  \node[rdot] (1) at ( -1.62,  0.00) {};

  \draw[edge] (2) -- (1);
  \draw[edge] (7) -- (2);
  \draw[edge] (4) -- (3);
  \draw[edge] (10) -- (4);
  \draw[edge] (11) -- (5);
  \draw[edge] (9) -- (6);
  \draw[edge] (8) -- (7);
  \draw[edge] (6) -- (8);
  \draw[edge] (13) -- (8);
  \draw[edge] (3) -- (9);
  \draw[edge] (5) -- (10);
  \draw[edge] (16) -- (10);
  \draw[edge] (17) -- (11);
  \draw[edge] (15) -- (12);
  \draw[edge] (14) -- (13);
  \draw[edge] (12) -- (14);
  \draw[edge] (19) -- (14);
  \draw[edge] (9) -- (15);
  \draw[edge] (11) -- (16);
  \draw[edge] (22) -- (16);
  \draw[edge] (23) -- (17);
  \draw[edge] (21) -- (18);
  \draw[edge] (20) -- (19);
  \draw[edge] (18) -- (20);
  \draw[edge] (25) -- (20);
  \draw[edge] (15) -- (21);
  \draw[edge] (17) -- (22);
  \draw[edge] (28) -- (22);
  \draw[edge] (29) -- (23);
  \draw[edge] (27) -- (24);
  \draw[edge] (26) -- (25);
  \draw[edge] (24) -- (26);
  \draw[edge] (21) -- (27);
  \draw[edge] (23) -- (28);

  \node[plabel] at (0,0) {1\,2\,7\,8\,10\,5\,3\,4\,6\,9\,11-5};
\end{tikzpicture}%
}%
\par

\caption{Posets for 4 occurrences of various patterns, pattern-overlap in box
  under each one. Red nodes are first occurrence, yellow second occurrence,
  circled nodes in both.}
\label{fig-posets-zoo-2}
\end{figure}


\newpage
\begin{table}[H] \centering \small
\begin{tabular}{|c|c|c|c|} \hline pattern & overlaps 1 & overlaps 2 &
overlaps 3 \\ \hline \hline $132$ & {$(2j-1)\sfact{2}$} & &\\\hline
\hline $1243$ & $(3j-2)\sfact{3}$ \blu{Coro. \ref{coro-1-k-1}} & \one
& \\\hline $1324$ & 1 & $C_{j}$ \blu{Prop. \ref{prop-Fuss-Catalan}} &
\\\hline $1342$ & $\prod_{i=0}^{j-1}\binom{3i+2}{2}$ & \one & \\\hline
$1423$ & $(3j-2)\sfact{3}$ \blu{Coro. \ref{coro-1-k-1}} & 1 & \\\hline
$2143$ & ~ $1, 9, 234, 12204, 1067040,\ldots $ & 1 & \\\hline $2413$ &
~ $1, 9, 234, 12204, 1067040,\ldots$ & $C_{j}$,
\blu{Prop. \ref{prop-Fuss-Catalan}} & \\\hline \hline \hline $12354$ &
$(4j-3)\sfact{4}$ \blu{Coro. \ref{coro-1-k-1}} & \one & \one \\\hline
$12435$ & 1 & 1 & \one \\\hline $12453$
&$\prod_{i=0}^{j-1}\binom{4i+2}{2}$
\blu{Prop.~\ref{prop-prod-formula-trees}} & \one & \one \\\hline
$12534$ & $(4j-3)\sfact{4}$ \blu{Coro. \ref{coro-1-k-1}} &
$(3j-2)\sfact{3}$ & \one \\\hline $13254$ & $(4j-3)\sfact{4}$
\blu{Coro. \ref{coro-1-k-1}} & \one & 1 \\\hline $13425$ & 1 &
$\disp\frac{1}{2j+1}\binom{3j}{j}$ \blu{Prop. \ref{prop-Fuss-Catalan}}
& \one \\\hline $13452$ & $\prod_{i=0}^{j-1}\binom{4i+3}{3}$
\blu{Prop. \ref{prop-prod-formula-trees}} & \one & \one \\\hline
$13524$ & $(4j-3)\sfact{4}$ \blu{Coro. \ref{coro-1-k-1}} & ~ $1, 6,
81, 1806, 57447, \ldots$ & \one \\\hline $14253$ &
\parbox{10.3em}{$\prod_{i=0}^{j-1}\binom{4i+2}{2}$\blu{Prop.~\ref{prop-prod-formula-trees}}}
&\one ~& $C_{j}$ \blu{Prop. \ref{prop-Fuss-Catalan}} \\\hline $14325$
& 1 & \oeis\ A274644 & \one \\\hline $14523$ &
\parbox{10.3em}{$\prod_{i=0}^{j-1}\binom{4i+2}{2}$
\blu{Prop.~\ref{prop-prod-formula-trees}}} & $(2j-1)\sfact{2}$ & \one
\\\hline $15234$ & $(4j-3)\sfact{4}$ \blu{Coro. \ref{coro-1-k-1}} & 1
& \one \\\hline $15243$ &
\parbox{10.3em}{$\prod_{i=0}^{j-1}\binom{4i+2}{2}$
\blu{Prop.~\ref{prop-prod-formula-trees}}} & \one & 1 \\\hline $15324$
& $(4j-3)\sfact{4}$ \blu{Coro. \ref{coro-1-k-1}} & $(3j-2)\sfact{3}$ &
\one \\\hline $15423$ & \parbox{10.3em}{
$\prod_{i=0}^{j-1}\binom{4i+2}{2}$
\blu{Prop.~\ref{prop-prod-formula-trees}}} & 1 & \one \\\hline $21354$
& ~ $1, 16, 816, 86528, 15661440,\ldots$ & 1 & \one \\\hline $21453$ &
~ $1, 30, 4440, 1867920,\ldots$ & 1 & \one \\\hline $21534$ & ~ $1,
16, 816, 86528, 15661440,\ldots$ & \one & \one \\\hline $21543$ & ~
$1, 30, 4440, 1867920,\ldots$ & $(3j-2)\sfact{3}$ & \one \\\hline
$23514$ & ~ $1, 16, 816, 86528, 15661440,\ldots$& $(2j-1)\sfact{2}$ &
\one \\\hline $24153$ & ~ $1, 30, 4440, 1867920,\ldots$& \one & \one
\\\hline $24513$ & ~ $1, 30, 4440, 1867920,\ldots$ & ~ $1, 6, 81,
1806, 57447, \ldots$& \one \\\hline $25314$ & ~ $1, 16, 816, 86528,
15661440,\ldots$ & \oeis\ A274644 & \one \\\hline $25413$ & ~ $1, 30,
4440, 1867920,\ldots$& $\disp\frac{1}{2j+1}\binom{3j}{j}$
\blu{Prop. \ref{prop-Fuss-Catalan}} & \one \\\hline
\end{tabular}
\caption{Enumerating sequences for all
equivalence classes of patterns of length up to~$5$, excluding the
pattern $123\ldots K$.  Here, $j$ is the number of occurrences and
$(n)\sfact{k}$ is $n(n-k)(n-2k)\cdots$, see
Definition~\ref{def-sfact}.}
\label{table-contain}
\end{table}


\section{Symmetries}\label{sec-symmetries}

Prior to moving into the classification and enumeration of
permutations packing a given pattern --- or, equivalently, the
enumeration of linear extensions of pattern posets --- we take note of
a number of elementary symmetries that will simplify the task.

\begin{definition}\label{def-rev-compl} Consider a permutation $\pi\in
\cls_n$, written as a word in one-line notation $\pi_1\ldots
\pi_n$. The \emph{reverse of $\pi$} is the permutation
$\pi^R=\pi_n\ldots\pi_1$.  The \emph{complement of $\pi$} is the
permutation $\pi^C=(n+1-\pi_1)\ldots (n+1-\pi_n)$.  The \emph{reverse
complement of $\pi$} is obtained by composing the two operations,
namely $\pi^{RC}=(n+1-\pi_n)\ldots (n+1-\pi_1)$.
\end{definition}

Let $\cls_{\pi}^m(j)$ denote the set of permutations $m$-packing $j$
occurrences of a $K$-pattern $\pi$.  The proof of the following lemma
consists of applying the same symmetry operation to each permutation
$\sigma$ $m$-packing a pattern $\pi$ as is applied to $\pi$ itself,
since each of them transforms an occurrence of $\pi$ in $\sigma$ in
the same way as it transforms $\pi$.

\begin{lemma}\label{lemma-symmetries} Fix a $K$-pattern $\pi$. Then,
  for any overlap $m$ and number of occurrences
  $j$,
  $$|\cls_{\pi}^m(j)|=|\cls_{\pi^R}^m(j)|=|\cls_{\pi^C}^m(j)|=|\cls_{\pi^{RC}}^m(j)|.$$
\end{lemma}

For example, \emph{for any given overlap $m$}, the number of
permutations $m$-packing 14253 is the same as that $m$-packing any of
the following patterns:

\begin{table}[h] \centering
    \begin{tabular}{|c|c|}\hline $\pi$&Symmetry\\\hline 14253&
original pattern \\ 52413& complement\\ 35241& reverse\\ 31425&
reverse complement\\\hline
    \end{tabular}
    \caption{Example of the usual symmetries of consecutive
permutation patterns.}
    \label{tab:my_label}
\end{table}

Nevertheless, a closer inspection of Table~\ref{table-contain} reveals
many more equinumerations than these trivial ones. Prior to discussing
these, we require some additional notation.

\begin{definition}\label{def-pref-suff}\label{def-prefix} Let $\pi$ be
a $K$-pattern and $m<K$. Its \emph{$m$-prefix} and \emph{$m$-suffix}
are $\bp=\bp_1\ldots\bp_m=\pi_1\ldots\pi_m$ and
$\bs=\bs_1\ldots\bs_m=\pi_{K-m+1}\ldots\pi_K$, respectively.  For
$m\le K/2$, its \emph{interfix} is $\bw=\pi_{m+1}\ldots\pi_{K-m}$.
\end{definition} The interfix is empty when $K=2m$. For $j\ge2$, a
packing can exist only if $\bp\sim\bs$, since the orders imposed on
each shared block must agree.  If this fails, $|\cls_\pi^m(1)|=1$ and
$|\cls_\pi^m(j)|=0$ for $j\ge2$.  Henceforth, whenever pattern packing
is considered, we assume $\bp\sim\bs$.

Any permutation $\sigma$ $(m,j)$-packing a $K$-pattern $\pi$ naturally
decomposes into distinct subwords.  Namely, if $m>K/2$ the subwords
corresponding to the overlap of the prefix and the suffix are
necessarily order isomorphic to the intersection of the prefix and
suffix. In contrast, when $m\le K/2$, the packing does not place
further constraints on the $j$ subwords that lie between the
prefix/suffix components, apart from having to be isomorphic to the
interfix of $\pi$. We denote the resulting decomposition as follows.

\begin{notation}\label{notation-packing}[Decomposition of permutations
packing a pattern] Fix a $K$-pattern $\pi$ and overlap $m\le K/2$. For
$j\in\mathbb N$ and $\sigma\in\cls_\pi^m(j)$, we write
\begin{equation}
\sigma=\rho^{(1)}\omega^{(1)}\rho^{(2)}\omega^{(2)}\ldots\rho^{(j)}\omega^{(j)}\rho^{(j+1)},
\end{equation} where, for all $i\in[j+1]$, $\rho^{(i)}\sim\bp$, that
is, each $\rho^{(i)}$ must be isomorphic to the $m$-prefix/suffix of
$\pi$.  When $m=K/2$ the $\omega^{(i)}$ are empty, in which case
$\sigma$ also packs $j+1$ non-overlapping occurrences of the
$m$-prefix/suffix of $\pi$.
\end{notation}

\begin{example}\label{ex-2.5} For $\sigma=
1\,\,\,11\,\,\,5\,\,2\,\,\,10\,\,\,7\,\,\,3\,\,9\,\,\,6\,\,\,4\,\,\,8\in\cls_{15324}^2(3)$,
we have
$$\sigma=\underbrace{1\,\,\,\,11}_{\rho^{(1)}}\,\underbrace{5}_{\omega^{(1)}}\,\underbrace{2\,\,\,\,\,10}_{\rho^{(2)}}\,\underbrace{7}_{\omega^{(2)}}\,\underbrace{3\,\,\,\,\,\,9}_{\rho^{(3)}}\,\underbrace{6}_{\omega^{(3)}}\, \underbrace{4\,\,\,\,\,\,8}_{\rho^{(4)}}.$$
\end{example}

When we talk about permuting a word $w_1w_2\ldots w_n$ consisting of distinct
letters by a permutation $\tau=\tau_1\tau_2\ldots \tau_n$ that means arranging
the letters of $w$ so that $w_{\tau_i}$ is in place~$i$ in the permuted
word. For example, permuting 46927 by 31542 results in 94726.

For fixed overlap $m$, we can now state the following additional
symmetry.
\begin{lemma}\label{Lemma-prefix_suffix} The number of permutations
$(m,j)$-packing a $K$-pattern $\pi$ depends only on the $m$-prefix and
$m$-suffix of $\pi$.
\label{lemma-symmetry-prefix}
\end{lemma}
\begin{proof}\label{only-pref-suff} When $m\ge K/2$, the result is
trivial because one can then reconstruct $\pi$ from its $m$-prefix and
$m$-suffix, which between them contain all the letters of
$\pi$. Suppose now that $m<K/2$. Let $\pi=\bp\bw\bs$, interpreted as a
word and written as the concatenation of the $m$-prefix $\bp$, the
interfix $\bw$, and the $m$-suffix~$\bs$. Any permutation
$\sigma\in\cls_{\pi}^m(j)$ can be expressed as
(Notation~\ref{notation-packing})
$$
\sigma=\rho^{(1)}\omega^{(1)}\rho^{(2)}\omega^{(2)}\ldots\rho^{(j)}\omega^{(j)}\rho^{(j+1)},
$$
where for all $i\in[j+1]$, $\rho^{(i)}\sim \bp$.

Let $\tau$ be a permutation of length $K-2m=|\bw|$.  Permuting each
$\omega^{(i)}$ by $\tau$ and letting $\bw'$ be $\bw$ permuted by $\tau$, we get
a permutation $\sigma'$ that is covered by $\pi'=\bp\bw'\bs$ with overlaps~$m$.
This process is clearly reversible, which proves the claim.
\end{proof}

\begin{example}\label{ex-1-pack-same-enum} In the first column of
  Table~\ref{table-contain} we see that 1-packings (overlap $m=1$) of all the
  5-patterns with $\pi_1=1$ and $\pi_5=4$ have the same enumeration.
\end{example}

In the following lemma, an $n$-permutation is a permutation of
$\{1,2,\ldots,n\}$.

\begin{lemma}\label{lem-overlap-pref-suff} Let $\pi=\bp \bw\bs$ be a
$K$-pattern whose $m$-prefix $\bp$ and $m$-suffix $\bs$ are order
isomorphic, $m\le K/2$ and $\bw$ is any permutation of the remaining
letters.  Let $\bp'$ and $\bs'$ be $\bp$ and $\bs$ permuted by the
same permutation $\tau$ and let $\bw'$ be $\bw$ permuted by some
permutation $\rho.$ Then the number of $n$-permutations covered by
$\pi$ with overlaps $m$ equals the number of $n$-permutations covered
by $\pi'=\bp'\bw'\bs'$ with overlaps $m$.
\end{lemma}
\begin{proof} Given a permutation~$\sigma$ $m$-covered by $\pi$,
permute by $\tau$ the letters corresponding to $\bp$ and $\bs$,
respectively, in every occurrence of~$\pi$ in~$\sigma$, and permute by
$\rho$ the letters between these occurrences of $\bp$/$\bs$, that is,
the letters of the $\omega^{(i)}$ in Notation~\ref{notation-packing}.
This turns every occurrence of $\pi$ into an occurrence of $\pi'$, and
this is clearly reversible, so that the map thus defined is a
bijection.
\end{proof}

\begin{example}\label{ex-2-pack-same-enum} Referring to the second
column of Table~\ref{table-contain}, 2-packing the pattern $12435$
gives the same enumeration as 2-packing $21453$.
\end{example}
\begin{lemma}\label{lem-pref-suff-wilf} Let $\pi=\bp\bw\bs$ be a
$K$-pattern with $m$-prefix $\bp$ and $m$-suffix $\bs$ where $m\le
K/2$.  The number of permutations $m$-packing $j$ occurrences of $\pi$
equals the number of permutations $m$-packing $j$ occurrences of
$\pi'=\bs\bw\bp$.
\end{lemma}
\begin{proof} By reversing each of $\bp,\bw$ and $\bs$ and then
reversing the resulting word, we obtain $\pi'$, and by
Lemmas~\ref{lemma-symmetries} and~\ref{lem-overlap-pref-suff} this
does not affect the enumeration of permutations $m$-packing $\pi$.
\end{proof}

\begin{example}\label{ex-pref-suff-same-enum} Referring to the second
column of Table~\ref{table-contain}, 2-packing the pattern $13425$
gives the same enumeration as 2-packing $25413$.
\end{example}

\begin{remark}\label{remark-canon-patt} By
Lemma~\ref{lem-overlap-pref-suff}, when $m\le K/2$ and $\bp\sim\bs$,
we may sort the prefix, the suffix, and the letters between them,
respectively, to increasing order without affecting the enumeration of
permutations covered by $\pi$ for a given overlap. Without affecting
the enumeration, we may also swap the prefix and suffix, by
Lemma~\ref{lem-pref-suff-wilf}, and also apply any of the symmetries
in Lemma~\ref{lemma-symmetries}.  In what follows most of the examples
and results we give for particular patterns or families of patterns
consist of the \emph{canonical representative} of each equivalence
class with respect to these symmetries, which is the lexicographically
least pattern in the class.  As an example, the canonical
representative of the class containing 37461285, for overlap 3, is
14738256, obtained by taking the reverse complement 41783526 of
37461285 and then sorting each of 417, 83 and 526 to increasing order.
This is equivalent to first switching prefix and suffix, then taking
complement and finally sorting each of the three parts.
\end{remark}

\begin{definition}\label{def-content} The \emph{content} of a word
$w=a_1a_2\ldots a_n$ is the set $\{a_1,a_2,\ldots,a_n\}$.
\end{definition}

\begin{lemma}\label{lem-content}

  Suppose $\pi$ and $\pi'$ are two $K$-patterns, each having
$m$-prefix isomorphic to its $m$-suffix where $K/2<m\le2K/3$.  If the
$m$-prefixes of $\pi$ and $\pi'$ have the same content, and likewise
their $m$-suffixes, then $\pi$ and $\pi'$ have the same enumeration.
\end{lemma}

\begin{proof} Since $m>K/2$, the $m$-prefix $\bp$ and $m$-suffix $\bs$
of $\pi$ intersect in $\ell=2m-K$ letters, and likewise for $\pi'$.
We denote the word consisting of this intersection by $\pi_x$,
respectively $\pi'_x$.  The content of $\pi_x$ and $\pi'_x$ must be
the same since the prefixes of $\pi$ and $\pi'$ have the same content
and their suffixes have the same content, respectively.

Let $\pi=\pi_i\pi_a\pi_x\pi_b\pi_f$, where $\pi_i,\pi_x$ and $\pi_f$
each have length $\ell=2m-K$, so $\pi_a$ and $\pi_b$ each have length
$m-2\ell$, and define $\pi'_i$ etc. analogously.  Thus,
$\bp=\pi_i\pi_a\pi_x$ and $\bs=\pi_x\pi_b\pi_f$, and therefore, since
$\bp$ and $\bs$ are order isomorphic, $\pi_i,\pi_x$ and $\pi_f$ must
all be order isomorphic, and $\pi_a$ and $\pi_b$ must be order
isomorphic, and likewise for $\pi'$.  Note that if $m=2K/3$ (so
$3\ell=K$) then $\pi_a$ and $\pi_b$ are empty, but this does not
affect the following argument.

We now show that given the content of $\bp$ and $\bs$ the content of
$\pi_x$ determines the content of each of $\pi_i,\pi_a,\pi_b,\pi_f$:
Since $\bp$ and $\bs$ are order isomorphic, the letters of $\pi_x$
stand in the same relation in terms of size to the letters of the word
$\pi_b\pi_f$ as $\pi_i$ does to the letters of the word $\pi_a\pi_x$.
As an example, if the letters of $\pi_x$ were the second, fifth and
sixth in size among the letters of $\pi_x\pi_b\pi_f$, then the letters
of $\pi_i$ must be the second, fifth and sixth in size among the
letters of $\pi_i\pi_a\pi_x$.  Thus, the content of $\pi_x$ determines
the content of $\pi_i$ and, analogously, the content of $\pi_f$.  The
remaining letters of $\bp$ are just $\pi_a$, whose content is thus
also determined, and likewise for $\pi_b$.

Since the content of $\pi_x$ equals that of $\pi'_x$, the argument in
the preceding paragraph implies that $\pi_i$ and $\pi'_i$ must have
the same content, and likewise for $\pi_f,\pi'_f$ and thus also
$\pi_a,\pi'_a$ and $\pi_b,\pi'_b$.

Now, $\pi_i,\pi_x,\pi_f$ are order isomorphic and $\pi_a,\pi_b$ are
order isomorphic, and likewise for $\pi'$.  There is thus a
permutation $\tau$ that takes $\pi_i,\pi_x$ and $\pi_f$ to
$\pi'_i,\pi'_x$ and $\pi'_f$, respectively, and a permutation $\rho$
that takes $\pi_a$ to $\pi'_a$ and $\pi_b$ to $\pi'_b$.  Given a
permutation covered by $\pi$ with overlaps $m$ we can split it into
segments corresponding to $\pi_i,\pi_a,\pi_x,\pi_b,\pi_f\ldots$.
Permuting each segment corresponding to $\pi_i,\pi_x$ or $\pi_f$ by
$\tau$ and each segment corresponding to $\pi_a$ or $\pi_b$ by $\rho$
we get a permutation covered by $\pi'$ with overlaps $m$.  Since this
is clearly reversible it constitutes a bijection proving the lemma.
\end{proof}

The following lemma follows from the proof of Lemma~\ref{lem-content}.

\begin{lemma}\label{lem-porism} Given a pattern $\pi$ with
$\pi_i,\pi_a,\pi_x,\pi_b,\pi_f$ as in Lemma~\ref{lem-content} and
$K/2<m\le2K/3$, let $\pi'$ be the pattern obtained from $\pi$ by
permuting all of $\pi_i,\pi_x,\pi_f$ by an arbitrary permutation
$\tau$ and permuting both $\pi_a$ and $\pi_b$ by an arbitrary
permutation $\rho$.  Then $\pi$ and $\pi'$ have the same enumeration
with respect to overlap $m$.
\end{lemma} As an example, for overlap 5 the patterns 13725846 and
31752864 have the same enumeration, since their prefixes and suffixes
have content $\{1,2,3,5,7\}$ and $\{2,4,5,6,8\}$, respectively.

\begin{remark}\label{rem-content-analog} For $m\le K/2$,
Lemma~\ref{lem-overlap-pref-suff} is analogous to
Lemma~\ref{lem-content}.  In the former case, however, there are no
restrictions on how we may permute the prefix/suffix, as they don't
intersect.
\end{remark}

Note that Table~\ref{table-contain} also contains examples of patterns
$\pi$ and $\pi'$ (such as 2143 and 2413) that have the same size and
identical numbers of $1$-packings, but different numbers of
$2$-packings.

\emph{\textbf{Important convention:}} For a $K$-pattern, the overlap
$m$ is always taken to satisfy $1\le m<K$. For $m\le K/2$, we will
generally take the $m$-prefix $\bp$ and $m$-suffix $\bs$ to be
increasing and $\bp_1<\bs_1$, since any pattern can be transformed to
adhere to that without affecting its enumerative properties. For a
$K$-pattern with $m>K/2$, we will generally assume that $1$ does not
belong to the suffix, which can be accomplished by reversal if need
be, since $1$ can not belong to the intersection of $\bp$ and $\bs$ as
we shall later show.  We also always assume that, given a $K$-pattern
$\pi$ and overlap $m$, $\bp$ and $\bs$ are order isomorphic, since
otherwise there are no permutations $(m,j)$-covered by $\pi$ for
$j>1$.

In the study of permutation patterns, two patterns $\pi$ and $\pi'$
are said to be \emph{Wilf equivalent} if the permutations avoiding
$\pi$ are equinumerous with the permutations avoiding $\pi'$. (A
stronger notion of equivalence is achieved by not only considering
avoidance, but any given number of occurrences.)  In our context we
declare two pairs $(\pi,m)$ and $(\pi',m')$ to be Wilf equivalent if
$\cls_{\pi}^m(j)$ and $\cls_{\pi'}^{m'}(j)$ have the same size for all
$j$.  Of course, the numerical equivalences presented in the symmetry
lemmas in this section all imply Wilf equivalence, but it seems there
are many other Wilf equivalences, even for patterns and overlaps of
different lengths, and we mention some of these in
Section~\ref{sec-open}.

\section{Generating function and general recurrence}\label{sec-gf}

In this section, we establish two general enumeration results: an
explicit rational multivariate generating function for $1$-packings of
an arbitrary pattern, and a recurrence for $m$-packings when $m\leq
K/2$.
\subsection{A rational generating function for $1$-packings}

At first sight, the enumeration of \(1\)-packings appears to depend
rather sensitively on the pattern
(cf. Table~\ref{table-contain}). Surprisingly, however, despite the
variety of packing sequences that arise in subsequent sections, all
overlap-\(1\) counts admit a single multivariate generating function
formula.

For $j\geq 1$, let $x_1,\ldots,x_j,y_1,\ldots,y_j,z_1,\ldots,z_j$ be
commuting indeterminates and work in the formal power series ring
$$
\mathbb Z[[x_1,\ldots,x_j,y_1,\ldots,y_j,z_1,\ldots,z_j]].
$$
Write
$$
\bx=(x_1,\ldots,x_j),\qquad \by=(y_1,\ldots,y_j),\qquad
\bz=(z_1,\ldots,z_j),
$$
and set
$$
B_0:=\sum_{r=1}^j x_r,\quad\quad
B_i:=\sum_{r=1}^{i-1}z_r+y_i+\sum_{r=i+1}^j x_r,\qquad 1\leq i\leq
j,\quad\quad B_{j+1}:=\sum_{r=1}^j z_r,
$$
where an empty sum is understood to be zero. Furthermore, for
$a\in\mathbb Z_{\geq 0}$, denote
$$
\bx_j^a:=\left(\prod_{i=1}^j x_i\right)^a,\qquad
\by_j^a:=\left(\prod_{i=1}^j y_i\right)^a,\qquad
\bz_j^a:=\left(\prod_{i=1}^j z_i\right)^a,
$$
with zeroth powers equal to $1$, the multiplicative identity in the
ring of formal power series.

\begin{theorem}\label{thm-gf-1-packing} Let $\pi$ be a $K$-pattern
with $\pi_1<\pi_K$. The number of permutations $1$-packing $j$
occurrences of $\pi$ is the coefficient of
$$
\bx_j^{\pi_1-1}\by_j^{\pi_K-\pi_1-1}\bz_j^{K-\pi_K}
$$
in the formal power series
\begin{equation}\label{eq-GF}
F_j(\bx,\by,\bz):=\prod_{i=0}^{j+1}\frac{1}{1-B_i}.
\end{equation}
\end{theorem}

\begin{example}\label{ex-coeff-2-3-occs} Expanding the notation we
have the following expression: \footnote{Note that setting
$x_i=y_i=z_i=0$ when $i<1$ or $i>j$, Equation~\eqref{eq-GF-expanded}
simplifies to this:
\begin{equation*}\label{eq-GF-simplified} F_j(\bx,\by,\bz) =
\dfrac{1}{ \displaystyle\prod_{i=0}^{j+1}
\left(1-\left(\displaystyle\sum_{r=1}^{i-1}z_r+y_i+\displaystyle\sum_{r=i+1}^j
x_r\right)\right)}.
\end{equation*} }
\begin{equation}\label{eq-GF-expanded} F_j(\bx,\by,\bz) = \dfrac{1}{
\left(1-\displaystyle\sum_{r=1}^j x_r\right)
\displaystyle\prod_{i=1}^j
\left(1-\left(\displaystyle\sum_{r=1}^{i-1}z_r+y_i+\displaystyle\sum_{r=i+1}^j
x_r\right)\right) \left(1-\displaystyle\sum_{r=1}^j z_r\right)}.
\end{equation}

Therefore, for $2$ occurrences, we look for the coefficient of
$$
(x_1x_2)^{\pi_1-1}(y_1y_2)^{\pi_K-\pi_1-1}(z_1z_2)^{K-\pi_K}
$$
in
\begin{equation}\label{eq-GF-2}
\dfrac{1}{(1-(x_1+x_2))(1-(y_1+x_2))(1-(z_1+y_2))(1-(z_1+z_2))}.
\end{equation}

For $3$ occurrences, we look for the coefficient of
$$
(x_1x_2x_3)^{\pi_1-1}(y_1y_2y_3)^{\pi_K-\pi_1-1}(z_1z_2z_3)^{K-\pi_K}
$$
in
\begin{equation}
\begin{aligned}
&\frac{1}{(1-x_1-x_2-x_3)(1-y_1-x_2-x_3)(1-z_1-y_2-x_3)}\\
&\hspace{3em}\times\frac{1}{(1-z_1-z_2-y_3)(1-z_1-z_2-z_3)}.
\end{aligned}
\end{equation}
\end{example}

\begin{corollary}\label{coro-1-pack-2-occs} The number of permutations
$1$-packing $2$ occurrences of a $K$-pattern $\pi$ is
$$
\binom{\pi_1+\pi_K-2}{\pi_K-1} \binom{2K-\pi_1-\pi_K}{K-\pi_K}.
$$
\end{corollary}
\begin{proof} By complementation, we may assume $\pi_1<\pi_K$. Set
$$
a=\pi_1-1,\qquad b=\pi_K-\pi_1-1,\qquad c=K-\pi_K.
$$
Extracting the coefficient of $(x_1x_2)^a(y_1y_2)^b(z_1z_2)^c$

in \eqref{eq-GF-2} gives the product
$$
\left(\sum_{r=0}^a \binom{a+r}{a}\binom{a+b-r}{b}\right)
\left(\sum_{s=0}^c \binom{b+s}{b}\binom{2c-s}{c}\right).
$$
Since
$$
\sum_{r=0}^n \binom{p+r}{p}\binom{q+n-r}{q} = \binom{p+q+n+1}{n},
$$
the two factors are, respectively,
$$
\binom{2a+b+1}{a} \qquad\text{and}\qquad \binom{b+2c+1}{c},
$$
which gives the stated expression.
\end{proof}

Prior to proving the main theorem of this section, we introduce some
additional notation. Fix $j,K\in\mathbb N$ and, for each $i\in[j]$,
denote by $I_i$ the set of positions occupied by the $i$th prescribed
occurrence of $\pi$ in a permutation that $1$-packs $j$ occurrences,
i.e.
$$
I_i:=\{1+(i-1)(K-1),\ldots,K+(i-1)(K-1)\}.
$$
For $0\leq i\leq j$, set
$$
e_i:=1+i(K-1).
$$
Then $e_i$ is the unique element of $I_i\cap I_{i+1}$ for $1\leq i\leq
j-1$, while $e_0$ and $e_j$ are the first and last positions of the
packing, respectively. For $i\in[j]$, we call $e_{i-1}$ and $e_i$ the
\emph{boundary positions} of the $i$th occurrence and set
$$
\widehat I_i:=I_i\setminus\{e_{i-1},e_i\}.
$$
Accordingly, we refer to $\sigma(e_{i-1})$ and $\sigma(e_i)$ as
\emph{boundary values}. Viewing $\sigma$ as a map
$[jK-(j-1)]\to[jK-(j-1)]$, the occurrences of $\pi$ naturally
partition the domain of $\sigma$ into the blocks
$$
\{e_0\},\widehat I_1,\{e_1\},\widehat I_2,\ldots,\{e_{j-1}\},\widehat
I_j,\{e_j\}.
$$

In the following, we once again denote by $\cls_{\pi}^1(j)$ the set of
permutations in $\cls_{jK-j+1}$ $1$-packing $j$ occurrences of a
$K$-pattern $\pi$.

We now state an elementary lemma concerning $1$-packings of a pattern
$\pi$ that will be used in the construction below.

\begin{lemma}\label{lemma-basic} Consider a permutation
$\sigma\in\clsn$ with an occurrence of a $K$-pattern $\pi$ spanning a
set of consecutive indices $A=\{a,\ldots,a+K-1\}$. Let $\widehat
A:=\{a+1,\ldots,a+K-2\}$. Assume $\pi_1<\pi_K$. Then there are exactly
$\pi_1-1$ indices $k\in\widehat A$ such that $\sigma(k)<\sigma(a)$,
exactly $\pi_K-\pi_1-1$ indices such that
$\sigma(a)<\sigma(k)<\sigma(a+K-1)$, and exactly $K-\pi_K$ indices
such that $\sigma(k)>\sigma(a+K-1)$.
\end{lemma}
\begin{proof} Immediate from the relative order prescribed by $\pi$.
\end{proof}

\begin{definition}\label{def-words-W} Given a $K$-pattern $\pi$ with
$\pi_1<\pi_K$ and $j\in\mathbb N$, let $\clw_{1,\pi,j}$ denote the set
of words of length $jK-j+1$ on the alphabet
$$
\cla_j:=\{x_1,\ldots,x_j,y_1,\ldots,y_j,z_1,\ldots,z_j,w_0,w_1,\ldots,w_j\}
$$
satisfying the following conditions:
\begin{enumerate}
\item[A1.] Each element of $\{w_0,w_1,\ldots,w_j\}$ appears exactly
once, and these elements appear in increasing order of their indices.
\item[A2.] Every place to the left of $w_0$ is occupied by a letter in
$\{x_1,\ldots,x_j\}$.
\item[A3.] Every place to the right of $w_j$ is occupied by a letter
in $\{z_1,\ldots,z_j\}$.
\item[A4.] For each $1\leq i\leq j$, every place between $w_{i-1}$ and
$w_i$ is occupied by a letter in
$$
\{z_1,\ldots,z_{i-1},y_i,x_{i+1},\ldots,x_j\}
$$
\item[A5.] For each $1\leq i\leq j$, there are exactly $\pi_1-1$
occurrences of $x_i$, $\pi_K-\pi_1-1$ occurrences of $y_i$, and
$K-\pi_K$ occurrences of $z_i$.
\end{enumerate}
\label{def-set-W}
\end{definition}

\begin{lemma}\label{lemma-bijection} For every $K$-pattern $\pi$ with
$\pi_1<\pi_K$ and $j\in\mathbb N$, the sets $\cls_{\pi}^1(j)$ and
$\clw_{1,\pi,j}$ are in bijection.
\end{lemma}

\begin{proof} Given $\sigma\in\cls_{\pi}^1(j)$, construct a word $W$
as follows. For $0\leq i\leq j$, put
$$
W_{\sigma(e_i)}:=w_i.
$$
For each $i\in[j]$ and $k\in\widehat I_i$, put in position $\sigma(k)$
the letter
$$
\begin{cases} x_i&\text{ if }\sigma(k)<\sigma(e_{i-1}),\\ y_i &\text{
if }\sigma(e_{i-1})<\sigma(k)<\sigma(e_i),\\ z_i &\text{ if }
\sigma(e_i)<\sigma(k).\\
\end{cases}
$$
(In other words, in terms of the underlying permutation, $x_i$, $y_i$,
and $z_i$ record whether an interior value of the $i$th occurrence
lies below its first boundary value, between its two boundary values,
or above its second boundary value.)

The blocks
$$
\{e_0\},\widehat I_1,\{e_1\},\ldots,\{e_{j-1}\},\widehat I_j,\{e_j\}
$$
partition the domain of $\sigma$. Since $\sigma$ is a permutation, the
construction assigns exactly one letter to each position of $W$. Thus
$W$ is a word of length $jK-j+1$ on $\cla_j$.

Since $\pi_1<\pi_K$, we have
$$
\sigma(e_0)<\sigma(e_1)<\cdots<\sigma(e_j),
$$
so A1 holds. If a value $v<\sigma(e_0)$, then $v$ is not a boundary
value and therefore $v=\sigma(k)$ for some $k\in\widehat I_r$. Since
$$
v<\sigma(e_0)\leq \sigma(e_{r-1}),
$$
the position $v$ in $W$ is marked $x_r$, proving A2. Similarly, if
$v>\sigma(e_j)$, then $v=\sigma(k)$ for some $k\in\widehat I_r$ and
$$
v>\sigma(e_j)\geq \sigma(e_r),
$$
so the position $v$ is marked $z_r$, proving A3.

Now suppose
$$
\sigma(e_{i-1})<v<\sigma(e_i)
$$
for some $1\leq i\leq j$. Again, $v=\sigma(k)$ for some $k\in\widehat
I_r$. If $r<i$, then
$$
v>\sigma(e_{i-1})\geq \sigma(e_r),
$$
so the position $v$ is marked $z_r$. If $r=i$, it is marked $y_i$. If
$r>i$, then
$$
v<\sigma(e_i)\leq \sigma(e_{r-1}),
$$
so it is marked $x_r$. This proves A4. Finally, A5 follows from
Lemma~\ref{lemma-basic}. Hence $W\in\clw_{1,\pi,j}$.

Conversely, let $W\in\clw_{1,\pi,j}$. Since positions in $W$ are
indexed by the values $1,\ldots,jK-j+1$, let $v_i$ denote the position
marked $w_i$. For each $i\in[j]$, let $V_i$ be the set consisting of
$v_{i-1}$, $v_i$, and all values whose positions in $W$ are marked
$x_i$, $y_i$, or $z_i$. By A5, $|V_i|=K$.

Reading A2--A4 as an exhaustive description of the letters permitted
in each interval determined by $w_0,\ldots,w_j$, every occurrence of
$x_i$ lies strictly to the left of $w_{i-1}$, every occurrence of
$y_i$ lies strictly between $w_{i-1}$ and $w_i$, and every occurrence
of $z_i$ lies strictly to the right of $w_i$. Hence all values marked
$x_i$ are smaller than $v_{i-1}$, all values marked $y_i$ lie strictly
between $v_{i-1}$ and $v_i$, and all values marked $z_i$ are larger
than $v_i$. By A5, in the increasing ordering of $V_i$, the value
$v_{i-1}$ therefore has rank $\pi_1$ and $v_i$ has rank $\pi_K$.

There is therefore a unique way to assign the values in $V_i$ to the positions
in $I_i$ so that the resulting word on $I_i$ is order-equivalent to $\pi$, with
$e_{i-1}$ receiving~$v_{i-1}$ and $e_i$ receiving~$v_i$. These assignments agree
on consecutive overlaps because both assign the common position~$e_i$ the value
$v_i$. Moreover, every value whose position in $W$ is marked $x_i$, $y_i$,
or~$z_i$ belongs to exactly one set $V_i$, while $v_i$ belongs to $V_i$ and
$V_{i+1}$ for $1\leq i\leq j-1$, with~$v_0$ belonging only to $V_1$ and $v_j$
only to $V_j$. Thus the overlaps among the sets $V_i$ agree exactly with the
overlaps among the sets $I_i$. Since every position of $W$ carries one of these
letters, the local assignments use every value in $[jK-j+1]$ exactly once and
define a permutation~$\sigma$. By construction, the restriction of $\sigma$ to
each $I_i$ is order-equivalent to $\pi$, so $\sigma\in\cls_{\pi}^1(j)$.

The two constructions are mutually inverse, completing the
proof.\end{proof}

\begin{proof}[Proof of Theorem~\ref{thm-gf-1-packing}] By A1, every
word $W\in\clw_{1,\pi,j}$ has a unique decomposition
$$
W=U_0w_0U_1w_1\ldots U_jw_jU_{j+1},
$$
where the $U_i$ are possibly empty words. Conditions A2--A4 state
precisely that the letters allowed in $U_0$ are $x_1,\ldots,x_j$,
those allowed in $U_i$ for $1\leq i\leq j$ are
$$
z_1,\ldots,z_{i-1},y_i,x_{i+1},\ldots,x_j,
$$
and those allowed in $U_{j+1}$ are $z_1,\ldots,z_j$.

Each $B_i$ has zero constant term, so
$$
\frac{1}{1-B_i}=\sum_{t\geq 0}B_i^t
$$
is a well-defined element of
$$
\mathbb Z[[x_1,\ldots,x_j,y_1,\ldots,y_j,z_1,\ldots,z_j]].
$$
In turn, the coefficient of a monomial in $B_i^t$ is the number of
words of length $t$ on the corresponding alphabet with the prescribed
multiplicities. Consequently, $(1-B_i)^{-1}$ is the multivariate
generating function for the possible words $U_i$. Since the
decomposition above is unique, the product
$$
\prod_{i=0}^{j+1}\frac{1}{1-B_i}
$$
is therefore the multivariate generating function for words satisfying
A1--A4, where the letters $w_0,\ldots,w_j$ contribute no variables.

Finally, extracting the coefficient of
$$
\bx_j^{\pi_1-1}\by_j^{\pi_K-\pi_1-1}\bz_j^{K-\pi_K}
$$
imposes A5, and the result follows by Lemma~\ref{lemma-bijection}.
\end{proof}

\subsection{General recurrence}

We now give a recursive formula for computing the number of
permutations of a given length that are covered by a given $K$-pattern
with overlaps $m\le K/2$.  For simplicity we assume the $m$-prefix,
and thus the $m$-suffix, of the pattern in question to be increasing,
as Lemma~\ref{lem-overlap-pref-suff} shows that the theorem also
applies if we permute the prefix and suffix in the same way.  The
following remark explains the treatment of separated components.

The following theorem first appeared in the authors' extended abstract
for the {\sc{fpsac}}-conference 2023~\cite{BKS-FPSAC-2023}.  For
$m=1$, the corresponding recurrence, with the first rather than the
last boundary value recorded, appears as (8) in the proof of
Theorem~28 in~\cite{Dotsenko-Khoroshkin-2013}. Their statement
concerns patterns whose distinct occurrences cannot overlap in more
than one position. Here only the prescribed overlaps have size one;
additional occurrences are allowed.

\begin{theorem}\label{thm-main-formula} Let $m\le K/2$ and let $\pi$
be a $K$-pattern with $m$-prefix $\bp$ and $m$-suffix $\bs$, where
$\bp$ and $\bs$ are increasing and $\bp_1<\bs_1$. Put $n_j=K+j(K-m)$
for $j\ge0$.

Let $u_1u_2\ldots u_{2m}$ be the concatenated word $\bp\bs$, and let
$\tau_i$ be the place of the $i$th smallest letter in $\bp\bs$. Let
$\tau_0=0$, $\tau_{2m+1}=2m+1$, and let $u_0=0$ and $u_{2m+1}=K+1$.

For an increasing word $L=\ell_{m+1}\ldots\ell_{2m}$ in $[n_j]$, let
$g_j(L)$ be the number of permutations of size $n_j$ that $m$-pack
$j+1$ occurrences of $\pi$ and end with $L$. Set $g_j(L)=0$ for every
other integer $m$-tuple $L$.

Then $g_0(\bs)=1$, $g_0(L)=0$ if $L\ne\bs$, and for $j\ge1$ we have
\begin{equation}\label{eq-master}
g_j(L)=\sum_{(\ell_1,\ldots,\ell_m)\in\cla_j(L)}
g_{j-1}(\ell'_1,\ldots,\ell'_m) \prod_{i=0}^{2m}
\binom{\ell_{\tau_{i+1}}-\ell_{\tau_i}-1}
{u_{\tau_{i+1}}-u_{\tau_i}-1},
\end{equation} where $\ell'_i=\ell_i-(u_i-i)$ for $1\le i\le m$, and
$\cla_j(L)$ is the set of sequences $\{(\ell_1,\ldots,\ell_m)$
satisfying
$$
0<\ell_{\tau_1}<\cdots<\ell_{\tau_{2m}}<n_j+1\}.
$$
In the product, $\ell_0=0$ and $\ell_{2m+1}=n_j+1$.  Here $\binom
uv=0$ if $u<v$.

Consequently, the number $h_j=|\cls_\pi^m(j+1)|$ of such packings is
\begin{equation}\label{eq-master-sum}
h_j=\sum_{1\le\ell_{m+1}<\cdots<\ell_{2m}\le n_j}
g_j(\ell_{m+1},\ldots,\ell_{2m}).
\end{equation}
\end{theorem}

\begin{proof} We prove this by induction on $j$, the base case, when
$j=0$, being $g_0(L)=1$ if and only if $L=\bs$, since there is
precisely one permutation of size $K$ with an occurrence of the
pattern $\pi$, namely the permutation consisting of $\pi$ itself.
  
  In order to construct a permutation of size $n_j$ covered by $\pi$
with overlaps~$m$ and ending with
$L=\ell_{m+1}\ell_{m+2}\ldots\ell_{2m}$, let $\oc=\oc_F\oc_M\oc_L$ be
the word consisting of the last occurrence of~$\pi$ in such a
permutation $\sigma$, and such that both $\oc_F$ and $\oc_L$ have
length $m$. Thus $\oc_F=\ell_1\cdots\ell_m$ and $\oc_L=L$, and
$(\ell_1,\ldots,\ell_m)\in\cla_j(L)$. We first analyze the possible
choices of letters in $\sigma$ that constitute $\oc_M$.

  Let $a<b$ be two letters in $\oc$ that belong to the union of
$\oc_F$ and $\oc_L$ and such that no letter $c$ in $\oc_F$ or $\oc_L$
satisfies $a<c<b$.  Thus, $a=\ell_{\tau_i}$ and $b=\ell_{\tau_{i+1}}$
for some $i$, where $\tau$ is the permutation defined in the statement
of the theorem.  However, there may be letters in $\oc_M$ that lie
between $a$ and $b$ in size.  In fact, there must be precisely
$u_{\tau_{i+1}}-u_{\tau_i}-1$ such letters in $\oc_M$, in order for
$\oc$ to be an occurrence of $\pi$.  Moreover, these letters can be
any of those that lie strictly between $\ell_{\tau_i}$ and
$\ell_{\tau_{i+1}}$ in size. In addition, there may be letters in
$\oc_M$ that are smaller or greater than all letters in $\oc_F$ and
$\oc_L$, which is taken care of by the definitions of
$\tau_0,\tau_{2m+1}$ etc. There are therefore
$\binom{\ell_{\tau_{i+1}}-\ell_{\tau_i}-1}{u_{\tau_{i+1}}-u_{\tau_i}-1}$
possible choices for these letters in $\oc_M$.  Having thus chosen the
letters to form $\oc_M$ there is only one way to order them so that
they will form the part of an occurrence of $\pi$ that lies between
the $m$-prefix and the $m$-suffix of that occurrence.

In order to complete the desired permutation $\sigma$, we must
construct a word $\sigma'$ of size $n_{j-1}$, with $\pi$-overlaps $m$,
and ending in $\ell_1,\ldots,\ell_m$.  Its alphabet is the set $R$
obtained from $[n_j]$ by removing the letters of $\oc_M$ and
$\oc_L$. This leaves $n_{j-1}$ letters, including all of $\oc_F$.

To do so, we define $\sigma''$ as the permutation of $[n_{j-1}]$ order
isomorphic to $\sigma'$. Standardization preserves every prescribed
occurrence.  $\sigma''$ has a suffix $\ell_1-(u_1-1) \ldots\ell_m
-(u_m-m)$ because among the $u_i-1$ letters of $\oc$ smaller than
$\ell_i$, exactly $i-1$ belong to the increasing prefix $\oc_F$.
Hence there are exactly $u_i-i$ letters in $\oc_M$ and $\oc_L$ smaller
than $\ell_i$.  By the inductive hypothesis, there are precisely
$g_{j-1}(\ell_1-(u_1-1),\ldots,\ell_m-(u_m-m))$ such $\sigma''$.

Each such $\sigma''$ uniquely determines $\sigma'$ on the alphabet
$R$, and appending $\oc_M\oc_L$ gives the required
packing. Conversely, the packing uniquely determines all these
choices. Since the choices in the disjoint intervals are independent,
their numbers multiply. Summing gives \eqref{eq-master}, and
\eqref{eq-master-sum} follows by partitioning packings according to
their final $m$ letters.
\end{proof}

\begin{remark}[Separated
components]\label{remark-separated-components} In cases where not all
pairs of successive prescribed occurrences are required to overlap,
this breaks the problem into smaller pieces of overlapping occurrences
--- whose enumeration can be computed via
Theorem~\ref{thm-main-formula} --- and segments of letters in a
permutation that are not required to belong to an occurrence. We then
need to multiply by the appropriate multinomial coefficient,
corresponding to choices of letters for each piece, and the
appropriate factorial for the letters not required to belong to an
occurrence, since those letters can be ordered arbitrarily.
\end{remark}

The procedure of Theorem~\ref{thm-main-formula} is computationally
efficient, in the sense that for a fixed $K$-pattern and a fixed $m$,
$h_j$ can be computed in time polynomial in $j$.  However, for $j>1$,
the methods developed in the next section, turning this computation
into counting linear extensions of posets, are more efficient in our
implementations.

\section{Pattern-packing posets}\label{sec-posets}

In Table~\ref{table-contain} we presented enumerative results for a
few patterns whose enumeration by combinatorial arguments is fairly
simple in the cases where we have explicit formulas, all of which are
presented in greater generality later in this section. For most
patterns, however, computing the enumeration is far from
straightforward, as they translate to very different combinatorial
problems, requiring distinct arguments and often diverse
techniques. Different patterns also pose vastly different levels of
challenge.  While the explicit results in Table~\ref{table-contain}
are low-hanging fruit, the other sequences, apart from 14325 with
overlap 2, were not known at the time of initiating this study (at
least not found in the \OEIS~\cite{OEIS}), and that is the case for
the vast majority of sequences we have encountered.

In this section we show that the permutations $(m,j)$-packing a
pattern $\pi$ are in one-to-one correspondence with the \emph{linear
extensions} of a poset determined by $\pi$, $m$ and $j$, where the
poset for $j$ is an induced subposet of the one for $j+1$.  A linear
extension of a poset $\clp$ is a total ordering $\clo$ of its elements
so that if $x<y$ in $\clp$ then $x$ precedes $y$ in $\clo$.

This correspondence was previously described for certain classes of
patterns by Dotsenko and Khoroshkin~\cite{Dotsenko-Khoroshkin-2013}
and by Elizalde and Noy~\cite{Elizalde-Noy-2012} and used to derive
generating functions counting packings of certain patterns.  Here we
elicit several general structural properties of these posets and
characterize completely those patterns the Hasse diagrams of whose
posets are trees or, in particular, chains (total orders).  These
structural properties in many cases suggest effective techniques for
enumerating linear extensions, which often explains why permutations
covered by different patterns have the same enumeration.

We begin by observing that the act of packing a given number of
occurrences of a consecutive permutation pattern $\pi$ forces
relations between the elements of the permutation thus covered by
$\pi$. For example, recalling Definition~\ref{def-order-isomorphism}
of the order isomorphism of permutations, any permutation $\sigma$
1-packing two occurrences of $\pi=2413$ must simultaneously satisfy:
$$
\sigma_1\,\sigma_2\,\sigma_3\,\sigma_4\sim \pi, ~\mbox{which is
equivalent to}~~ \sigma_3<\sigma_1<\sigma_4<\sigma_2 ~\mbox{(first
occurrence)}
$$
and also
$$
\sigma_4\,\sigma_5\,\sigma_6\,\sigma_7 \sim \pi, ~~\mbox{which is
equivalent to }~ \sigma_6<\sigma_4<\sigma_7<\sigma_5 ~\mbox{(second
occurrence)}.
$$
This implies, for example, that $\sigma_2$ must be greater than
$\sigma_6$, because $\sigma_6<\sigma_4<\sigma_2$.

As before, let $\cls_{\pi}^m(j)$ denote the set of permutations
$m$-packing $j$ occurrences of a $K$-pattern $\pi$.  Given $\pi$, any
permutation $\sigma\in \cls_{\pi}^m(j)$ (trivially) satisfies the
following equivalences, where $d=K-m$:

\begin{align*} &\sigma_1\,\ldots\,\sigma_K\sim \pi\notag\\
&\sigma_{1+d}\, \ldots\,\sigma_{K+d}\sim \pi\notag\\
&~~~~~~~~~~~~~\vdots\\
&\sigma_{1+(j-1)d}\ldots\sigma_{K+(j-1)d}\sim\pi\notag.
\end{align*} These statements are straightforwardly seen to correspond
to the following inequalities.

\begin{lemma}\label{lemma-inequalities} A permutation $\sigma$
$(m,j)$-packs a $K$-pattern $\pi$ if and only if $\sigma$ has length
$K+(j-1)(K-m)$ and satisfies the following $j$ inequalities, where
$\tau=\pi^{-1}$ and $d=K-m$:
\begin{align}
&\sigma(\tau_1)<\sigma(\tau_2)<\cdots<\sigma(\tau_K)\notag\\
&\sigma(\tau_1+d)<\sigma(\tau_2+d)<\cdots<\sigma(\tau_K+d)\notag\\
&~~~~~~~~~~~~~~~~~~~~~~\vdots\label{eqs-poset}\\
&\sigma(\tau_1+(j-1)d)<\sigma(\tau_2+(j-1)d)<\cdots<\sigma(\tau_K+(j-1)d)\notag
\end{align}
\end{lemma}

\begin{definition}[Pattern-packing posets]\label{def-patt-pack-poset}
Given a $K$-pattern $\pi$, overlap $m$ and number of occurrences $j$,
the poset $\ppm(j)$ defined by the inequalities \eqref{eqs-poset} is
said to be a \emph{pattern-packing poset}, or just a \emph{pattern
poset}, for $\pi$.
\end{definition}

The relationship between pattern-packing permutations and
pattern-packing posets follows immediately from
Lemma~\ref{lemma-inequalities}:

\begin{theorem}\label{thm-numoccs-linext} Given a pattern $\pi$,
overlap $m$ and number of occurrences $j$, the set of linear
extensions of the poset $\ppm(j)$ is in natural bijection with
$\cls_{\pi}^m(j)$.
\end{theorem} More precisely, the inverses of the permutations
$(m,j)$-packing a pattern $\pi$ \emph{are} exactly the linear
extensions of the poset $\ppmj$.

In what follows, we will frequently refer to the sequence of
inequalities in Lemma~\eqref{lemma-inequalities} as the
\emph{$j$-stack of $\ppmj$}, which we now define. Given a $K$-pattern
$\pi$ and overlap $m$, let $\tau$ be the inverse, written in one-line
notation, of $\pi$ (that is,\ if $\pi$ has $j$ in the $i$-th place
then $\tau$ has $i$ in the $j$-th place).  For $\ell\geq 0$, let
\begin{equation} \shift_\ell=s_{\ell1}s_{\ell2}\ldots s_{\ell K}
\end{equation} where $s_{\ell i}=\tau_i+\ell(K-m)$.  That is,
$\shift_0=\tau$ and, for $\ell>0$, $\shift_\ell$ is $\shift_{\ell-1}$
with each letter incremented by $(K-m)$.  Generically, by
Lemma~\ref{lemma-inequalities}, for any given $K$-pattern $\pi$ and
overlap $m$, letting $d=K-m$ we have:
\begin{equation}
\begin{array}{ccccccccc}
\shift_0&=&\tau_1&\tau_2&\tau_3&\ldots&\tau_K&\\
\shift_1&=&\tau_1+d&\tau_2+d&\tau_3+d&\ldots&\tau_K+d&\\
\shift_2&=&\tau_1+2d&\tau_2+2d&\tau_3+2d&\ldots&\tau_K+2d&\\ &\vdots
\end{array}
\end{equation}

\begin{definition}\label{def-stack} We call the sequence
$\shift_0,\shift_1,\ldots,\shift_{j-1}$ the \emph{$j$-stack} for
$\ppmj$ and refer to $\shift_0,\shift_1,\ldots,\shift_{j-1}$ as
\emph{lines of the stack}.
\end{definition}

It follows that the poset $\ppmj$ consists of the elements of the
$j$-stack for $\pi$, with order generated by $x<y$ whenever $x$
appears before $y$ on some line.  Here $<$ denotes the poset order,
not numerical order of the labels.

\begin{example}\label{ex-stack} If $\pi=143625$, and overlap is $m=2$,
so $d=4$, then $\tau=\pi^{-1}=153264$, so
$$
\begin{array}{cccccccccc} \shift_0&=&1&5&3&2&6&4\\
\shift_1&=&5&9&7&6&10&8\\ \shift_2&=&9&13&11&10&14&12
\end{array}
$$
and the poset $\pp{2}(3)$ is pictured in Figure~\ref{fig-143625}.
\end{example}

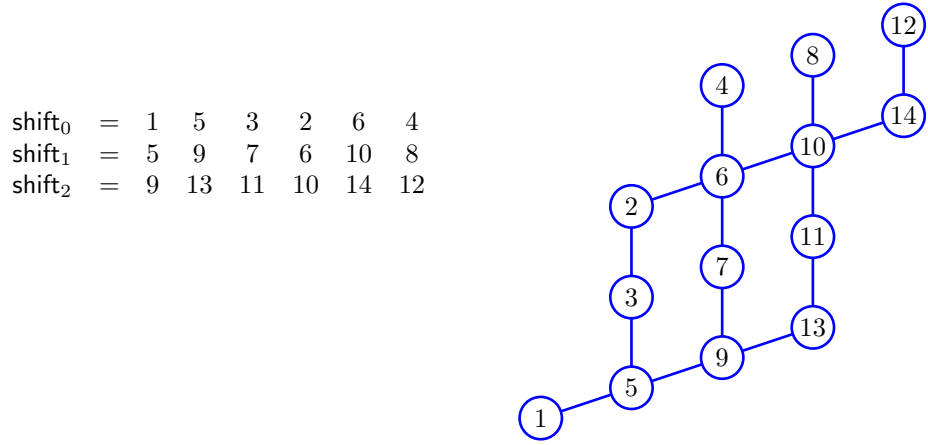
\begin{figure}[h] \hspace*{2em}
\begin{tikzpicture} \draw node { $
\begin{array}{cccccccccc} \shift_0&=&1&5&3&2&6&4\\
\shift_1&=&5&9&7&6&10&8\\ \shift_2&=&9&13&11&10&14&12\\[20ex]
\end{array} $ };
\end{tikzpicture} \hspace*{2em}
\begin{tikzpicture}[scale=0.4, line width=1pt]
\def\vertex#1,#2,#3,{\draw [blue, fill=white] (#1,#2) circle
[radius=0.7] node [black] {$#3$};}
  \begin{scope}[shift={(0,0)}] \draw [blue] (3,1)--(3,4)--(3,7); \draw
[blue] (6,2)--(6,5)--(6,8)--(6,11); \draw [blue]
(9,3)--(9,6)--(9,9)--(9,12); \draw [blue] (12,10)--(12,13); \draw
[blue] (0,0)--(3,1)--(6,2)--(9,3); \vertex0,0,1, \vertex3,1,5,
\vertex6,2,9, \vertex9,3,13, \vertex3,4,3, \vertex6,5,7,
\vertex9,6,11, \draw [blue] (3,7)--(6,8)--(9,9)--(12,10);
\vertex3,7,2, \vertex6,8,6, \vertex9,9,10, \vertex12,10,14,
\vertex6,11,4, \vertex9,12,8, \vertex12,13,12,
\end{scope}
\end{tikzpicture}
\caption{\label{fig-143625} The Hasse diagram of the poset $\ppm(3)$
for three occurrences of the pattern $\pi=143625$ (whose inverse is
153264) with overlap 2 and its defining stack.}
\end{figure} In Example~\ref{ex-stack}, the letters 5 and 6 that
appear in both $\shift_0$ and $\shift_1$ (but not in $\shift_2$) have
different sets of letters between them in $\shift_0$ and $\shift_1$
(namely $3,2$ and $9,7$, respectively).  This implies that the
corresponding poset is not a total order.  In fact, we will show (see
Theorem~\ref{thm-posets-trees}) that this is enough to deduce that the
poset is not a tree, when seen as the graph in
Figure~\ref{fig-143625}.  The \emph{Hasse diagram} of a poset $\clp$
is a directed graph whose elements are the elements of $\clp$ and
where there is an edge from $x$ to $y$ if and only if $x$ is
\emph{covered} by $y$, that is, $x<y$ and there is no $z$ so that
$x<z<y$. Instead of indicating the orientation of edges by arrows, a
Hasse diagram is typically drawn so that an edge from $x$ to $y$ goes
upwards from $x$ to $y$.  In Figure~\ref{fig-143625}, 5 is covered by
3 and 9, and 6 \emph{covers} 7 and 2, for example.  Thus, the Hasse
diagram of $\ppm(\ell)$ has vertices consisting of all the letters
appearing in its $\ell$-stack, and the edges $x\rightarrow y$ if $x$
is covered by $y$. Note, however, that although $x$ must immediately
precede $y$ on some line of the stack in order to be covered by $y$,
this is not sufficient, since on a different line of the stack we may
have a letter $z$ between $x$ and $y$, as in Figure~\ref{fig-162534},
where 5 immediately precedes 6 and 7, respectively, but 6 does not
cover 5.  We state a couple of observations on the structure of a
stack. First, the letters from the first line that are repeated on the
second line will necessarily be the $m$ largest letters of $\shift_0$
and the $m$ smallest letters of $\shift_1$, namely
$K-m+1,K-m+2,\ldots,K$. In $\shift_0$, these appear in the places
determined by the $m$-suffix of $\pi$ (in the case of
Figure~\ref{fig-143625} in places 2 and 5) and in $\shift_1$ they
appear in the places determined by the $m$-prefix of $\pi$ (places 1
and 4).  Moreover, these repeated letters must appear in the same
order in $\shift_0$ and $\shift_1$, since their places are determined
by the values, respectively, of the prefix and suffix of $\pi$, which
are order isomorphic (if they didn't this would lead to a
contradiction in the poset).  This generalizes in the following
obvious way:
\begin{lemma}\label{lemma-same-order} Given a pattern $\pi$ and
overlap $m$, the letters appearing in both $\shift_i$ and
$\shift_{i+1}$ are the $m$ largest letters of $\shift_i$ and the $m$
smallest letters of $\shift_{i+1}$.  Letters that belong to both
$\shift_i$ and $\shift_\ell$, for any $i$ and $\ell$, must appear in
the same order in both.
\end{lemma} The labels on $\shift_i$ form the integer interval
$[1+id,K+id]$.  Thus the lines containing a given label form an
interval of line indices.  Since incomparable elements cannot occur on
the same line, we obtain:
\begin{lemma}\label{lem-successive-lines} If an element $x$ appears in
$\shift_i$ and $\shift_\ell$ where $i<\ell$, then $x$ appears in
$\shift_k$ for all $k$ with $i<k<\ell$.  In particular, if $x$ and $y$
are incomparable elements in $\ppmj$ then either all lines in the
stack for $\ppmj$ containing $x$ must precede all lines containing
$y$, or vice versa.
\end{lemma} The following lemma will be useful in a couple of places.
\begin{lemma}\label{lemma-only-two-lines} Given a $K$-pattern $\pi$
and overlap $m$, if $m\le K/2$ then no letter will appear on more than
two lines of the stack, and two lines with a common letter must be
adjacent.
\end{lemma}
\begin{proof} The greatest letter of $\shift_i$ is $K+i(K-m)$, which
is smaller than the smallest letter of $\shift_{i+2}$, which is
$1+(i+2)(K-m)$.
\end{proof}
\subsection{The general structure of pattern
posets}\label{subsec-gen-structure-posets} The \emph{dual} of a poset
$\clp$ is the poset with the same elements as $\clp$ where $y\le x$ if
and only if $x\le y$ in $\clp$.  The Hasse diagram of the dual of
$\clp$ is obtained by turning the Hasse diagram of $\clp$ upside down.
\begin{remark}\label{remark-symm-posets} It is straightforward to
verify that apart from taking the complement of a pattern (dealt with
in the following Lemma~\ref{lem-dual}), the maps proving the lemmas
about enumeration equivalences in Section~\ref{sec-symmetries} have
corresponding maps on the posets for the patterns in question that
give isomorphic posets, but with the labels of the elements permuted.
As an example, to obtain the poset of the reverse of a pattern $\pi$
we complement its elements, that is, replace the largest label of an
element in $\ppmj$ with 1, the second largest with 2, and so on.
\end{remark}
\begin{lemma}\label{lem-dual} The dual of the pattern poset $\ppm(j)$
is the pattern poset of the complement $C(\pi)$ of $\pi$.
\end{lemma}
\begin{proof} The reverse of the inverse of a permutation $\pi$ is the
inverse of the complement of~$\pi$.  Thus, the inverses of $\pi$ and
$C(\pi)$ are the reverses of each other.  Therefore, each line of the
stack for $C(\pi)$ is the reverse of the corresponding line of the
stack for $\pi$, which implies that the order relations of the posets
(for a given number $j$ of occurrences) for $\pi$ and $C(\pi)$ are
each other's opposites. Therefore, these posets are each other's
duals.
\end{proof}
\begin{theorem}\label{thm-cover-2} An element in $\ppm(j)$ can cover
at most two elements and can be covered by at most two elements.
\end{theorem}
\begin{proof} If an element $x$ in $\ppm(j)$ is covered by $y$, then
$y$ must immediately follow $x$ in $\shift_i$ for some $i$.  We then
say that $y$ is a successor of $x$ (in $\shift_i$).  In order for an
element $x$ of $\ppm(j)$ to be covered by three different elements in
$\ppmj$, $x$ must have three different successors, in
$\shift_a,\shift_b$ and $\shift_c$, respectively, for some $a<b<c$.
But since $x$ belongs to each of these shifts, every letter in
$\shift_b$ that is smaller (in size) than $x$ must also belong to
$\shift_a$, and every letter in $\shift_b$ that is greater than $x$
must also belong to $\shift_c$.  In particular, the successor $z$ of
$x$ in $\shift_b$ must also belong to either $\shift_a$ or $\shift_c$.
If $z$ also belongs to $\shift_a$ then $z$ and the successor of $x$ in
$\shift_a$ are comparable and so they can't both cover $x$ in
$\ppmj$. The case where $z$ belongs to $\shift_c$ is analogous, and so
no element of $\ppmj$ can be covered by three different elements.
Since the dual of a pattern poset is a pattern poset
(Lemma~\ref{lem-dual}), no letter can cover more than two letters.
\end{proof} We now characterize the minimal and maximal elements of a
pattern poset $\ppmj$ based on whether 1 and $K$ belong to the prefix
or suffix of $\pi$ (or neither).
 \begin{proposition}\label{prop-min-max-elts} Let $\pi$ be a
$K$-pattern with isomorphic $m$-prefix $\bp$ and $m$-suffix $\bs$.
Then the poset $\ppmj$ satisfies the following:
  \begin{enumerate}
  \item If $\bp$ contains 1 then the only minimal element of $\ppmj$
is the first letter of $\shift_0$.
  \item If $\bs$ contains 1 then the only minimal element of $\ppmj$
is the first letter of $\shift_{j-1}$.
  \item If $\bp$ contains $K$ then the only maximal element of $\ppmj$
is the last letter of $\shift_0$.
  \item If $\bs$ contains $K$ then the only maximal element of $\ppmj$
is the last letter of $\shift_{j-1}$.
  \item If 1 belongs to neither $\bp$ nor $\bs$ then the first element
of $\shift_i$ for each $i$ is minimal in $\ppmj$.
  \item If K belongs to neither $\bp$ nor $\bs$ then the last element
of $\shift_i$ for each $i$ is maximal in $\ppmj$.
  \end{enumerate} It follows that, for $j>1$, $\ppmj$ has a single
minimal element if and only if 1 belongs to either $\bp$ or $\bs$, and
a single maximal element if and only if $K$ belongs to either $\bp$ or
$\bs$.
\end{proposition}
\begin{proof} The only elements that can be minimal in the poset $\ppmj$ are
  those appearing as the first letter of $\shift_i$ for some $i$.  Such an
  element $\epsilon$ is minimal if it does not appear on any other line of the
  stack, since if it doesn't then it can't be greater than any other element
  appearing in the stack.  Conversely, if $\epsilon$ does appear on another line
  of the stack then on that line it cannot appear in the first place (since
  those letters differ by $d=K-m$ from one line to the next), and consequently
  $\epsilon$ is greater in $\ppmj$ than the letters preceding it on that line.
  Similarly, the only elements that can be maximal in the poset $\ppmj$ are
  those appearing as the last letter of $\shift_i$ for some $i$ and such an
  element is maximal if and only if it appears on no other line of the stack.
  Suppose 1 belongs to $\bp$.  Then the maximum possible value of the first
  element of $\shift_0$ is $m$, so the maximum possible value of the first
  letter of $\shift_1$ is $m+(K-m)=K$, which is also found in $\shift_0$, so the
  first letter of $\shift_1$ cannot be minimal in $\ppmj$.  Analogous reasoning
  shows that the first letter of $\shift_i$ for any $i>1$ cannot be minimal in
  $\ppmj$.  The first letter $\ell$ of $\shift_0$ is minimal, however, when 1
  belongs to $\bp$, as we now show: If $m\le K/2$ then $\ell$ is at most $K/2$
  and the least element of $\shift_1$ is $1+K-m>K/2$, so $\ell$ does not appear
  in $\shift_1$ (or on any later line of the stack).  If $m>K/2$ then 1 cannot
  appear in the intersection of $\bp$ and $\bs$ since if it did then it of
  course would be the smallest element of each.  But then the smallest element
  of the prefix would have fewer greater letters after it than the smallest
  letter of the suffix, so they couldn't be isomorphic.  So if 1 belongs to
  $\bp$ then the maximum possible value of $\ell$ is $K-m$ and the least element
  in $\shift_1$ is $1+K-m$, so again $\ell$ does not appear in $\shift_1$.  The
  argument in the case when 1 belongs to the suffix is similar, and analogous
  arguments prove the claims about the position of $K$.  Finally, if~1 belongs
  to neither $\bp$ nor $\bs$ then an analysis similar to the one above shows
  that the first letter of $\shift_i$ for any $i$ cannot appear on any other
  line of the stack, and an analogous argument applies when $K$ belongs neither
  to $\bp$ nor $\bs$.  The last part of the theorem follows from its preceding
  parts.
\end{proof} Before deriving a corollary of
Proposition~\ref{prop-min-max-elts} we make the following definition.
Note first that every pattern-packing poset $\calP$ is connected, by
virtue of the occurrences of the pattern overlapping.

\begin{definition}\label{def-tree-chain} A pattern poset $\clp$ whose
Hasse diagram, viewed as an undirected graph, is a tree is referred to
as a \emph{tree poset} or just a \emph{tree}. If, in addition, $\clp$
has a unique minimal and a unique maximal element then $\clp$ is a
\emph{chain}.
\end{definition} A poset $\clp$ being a chain means that there is a
total order on the elements of $\clp$.

If $m\ge K/2$ then each of 1 and $K$ belongs to either prefix or
suffix, and so $\ppmj$ has a unique minimal and a unique maximal
element by Proposition~\ref{prop-min-max-elts}:

\begin{corollary}\label{coro-unique-min-max} For overlap $m\ge K/2$,
if $\ppmj$ is a tree then it is a chain.
\end{corollary}

In the following definition we restrict the pattern $\pi$ to not have
1 in its suffix, since that simplifies the definition and the proof of
the following theorem.  This causes no loss of generality, since if 1
is in the suffix we obtain an equivalent pattern by reversing $\pi$.
\begin{definition} Let $\pi$ be a $K$-pattern that doesn't have 1 in
its $m$-suffix.  Let $M_p$ be the greatest element of the $m$-prefix
of $\pi$ and let $M_s$ be the greatest element of the $m$-suffix of
$\pi$.

  The \emph{bottom element} of $\ppmj$ is the first letter of the
first line of the stack for $\ppmj$.

  The \emph{top element} of $\ppmj$ is the last letter of the first
line of the stack for $\ppmj$ if $M_p>M_s$, but otherwise it is the
last letter of the last line of the stack.

  The \emph{core} of $\ppmj$ is the interval between the top and
bottom elements of $\ppmj$.

  A \emph{branch} of $\ppmj$ is a nonempty connected subset of
$\ppmj$, disjoint from the core and maximal with respect to inclusion.
\defends
\end{definition}

See Figure~\ref{fig-251368947} for an example of the core and branches of a
poset.

For a $K$-pattern $\pi$ and overlap $m$ recall that the \emph{interfix} of $\pi$
is $\pi_{m+1}\pi_{m+2}\ldots \pi_{K-m}$, that is, the segment of letters between
the $m$-prefix and $m$-suffix of $\pi$.  For $m>K/2$, the interfix is understood
to be empty. The length of a chain is its number of elements.

As in the preceding definition, in the following theorem we restrict
the pattern $\pi$ to not have 1 in its suffix but also, when $m\le
K/2$, to have the least element of its $m$-prefix be smaller than the
least element of its $m$-suffix, and to have increasing prefix and
suffix, since this simplifies the proof, without losing generality, by
the lemmas of Section~\ref{sec-symmetries}.

\begin{theorem}\label{thm-branch-is-chain} Let $\pi$ be a $K$-pattern
with $m$-prefix $\bp$ and $m$-suffix $\bs$ that does not have 1 in
$\bs$, where the least element of $\bp$ is smaller than the least
element of $\bs$, and, when $m\le K/2$, assume $\bp$ and $\bs$ are
increasing.
  
  Let $M_p$ be the greatest element of $\bp$ and let $M_s$ be the
greatest element of $\bs$.\\[0.5ex] Let $1,2,3,\ldots,a$ be the
maximal sequence of successive values, starting with~1, contained (in
any order) in the interfix of $\pi$ and set $a=0$ if the sequence is
empty.\\[0.5ex] Let $K-b+1,K-b+2,\ldots,K$ be the maximal sequence of
successive values, ending with $K$, contained (in any order) in the
interfix of $\pi$ and set $b=0$ if the sequence is empty.

  If either sequence above is empty the corresponding following
statements are omitted.

  Then:
\begin{enumerate}
\item For each $i>0$, the first $a$ elements of $\shift_i$ form a
branch in $\ppmj$.

\item If $M_p>M_s$ then, for each $i>0$,the last $b$ elements of
$\shift_i$ form a branch in $\ppmj$.

\item If $M_p<M_s$ then, for each $i<j-1$, the last $b$ elements of
$\shift_i$ form a branch in $\ppmj$.

\item Each branch of $\ppmj$ is a chain, connected by a single edge to
the core of $\ppmj$.

\item All other elements of the stack belong to the core.

\item Every chain of maximal length in $\ppmj$ contains the top and
bottom elements of $\ppmj$.
\end{enumerate}
\end{theorem}

\begin{proof} The case $j=1$ is trivial, since $\ppm(1)$ is a single
chain. Suppose first that $m>K/2$.  Then neither $1$ nor $K$ can
belong to the intersection of $\bp$ and $\bs$, for if 1 belonged to
both then it would be the least element of both, but in different
positions, so they couldn't be order isomorphic.  The case for $K$ is
analogous. Thus $1$ belongs only to $\bp$, since it is assumed to not
belong to $\bs$, and $K$ belongs either to $\bp$ or $\bs$.

By Proposition~\ref{prop-min-max-elts}, the poset has a unique minimal
and a unique maximal element. This minimum must be the first letter of
the first line, because the first letter of every later line also
belongs to the preceding line.  Namely, since 1 is in $\bp$, and not
in the intersection of $\bp$ and $\bs$, its place in $\pi$ is at most
$K-m<m$, so the first element $f_0$ of $\shift_0$ is therefore
strictly smaller than $m$, and the first element $f_i$ of $\shift_i$,
for any $i>0$, is strictly smaller than $m+i(K-m)=K+(i-1)(K-m)$, which
is the greatest letter in $\shift_{i-1}$, so $f_i$ belongs to
$\shift_{i-1}$.

Analogous arguments show that if $K$ belongs to $\bp$, each later
line's last letter belongs to the preceding line and if $K$ belongs to
$\bs$, each line's last letter, except for the last line, belongs to
the following line. The minimum and maximum are therefore the stated
bottom and top. Since each is comparable to every element of the
stack, every maximal chain contains both, the core is the entire poset
and there are no branches, as $a=b=0$.

Now suppose $m\le K/2$, that $\bp$ and $\bs$ are increasing, that the
minimum of $\bp$ is smaller than the minimum of $\bs$ (and thus that 1
is not in $\bs$).  If $a>0$ then, since the elements $1,2,\ldots,a$
belong to the interfix, and thus are in places greater than $m$, the
first $a$ elements of $\shift_i$ for $i>0$ are greater than
$m+i(K-m)=K+(i-1)(K-m)$, the greatest element of $\shift_{i-1}$, and
therefore do not belong to $\shift_k$ for any $k<i$. Their labels are
at most $(K-m)+i(K-m)=(i+1)(K-m)$, since they precede the $m$-suffix,
so they do not belong to any later line either. Thus each occurs on
just one line, so they form a chain whose only attachment to the rest
of the poset is via the edge to the element in place $a+1$.

When $b>0$, a similar argument shows that the last $b$ elements of
$\shift_i$ for $i<j-1$ do not belong to $\shift_k$ for any $k>i$, nor
do the last $b$ elements of $\shift_i$ for $i>0$ belong to $\shift_k$
for any $k<i$, so regardless of whether the top element is in the
first or last line of the stack, the last $b$ elements on each line
specified in the statement cannot, in $\ppmj$, be below the top
element.

We now show that the bottom element, the first element of $\shift_0$,
is, in $\ppmj$, below the element in place $a+1$, and thus below all
elements in places after the $a$-th, in each $\shift_i$ for $i>0$,
which excludes only the $a$ branch elements at the beginning of each
of these lines.  The letter in place $a+1$ in $\shift_0$ is 1, because
1 is in the place determined by the first letter of $\bp$ (which is
assumed increasing), and this letter must be $a+1$, since that's the
smallest letter not in the interfix, and that letter cannot be in
$\bs$, whose least letter is greater than that of $\bp$.  Thus, the
letter in place $a+1$ in $\shift_i$ is $1+i(K-m)$, which occurs in
place $\bs_1$ of $\shift_{i-1}$.  Since $\bs_1>\bp_1=a+1$, induction
from the first line proves the claim.

Again, similar arguments show that the top element is greater than or
equal to all elements in places up to $K-b$ in each $\shift_i$.

Consequently, the elements other than those described as branches in
the statement are exactly the core.  Since each of the sets of
elements described as branches has been shown to form a chain,
connected by a single edge to the core, it does indeed constitute a
branch.

It remains to consider longest chains, which are necessarily
maximal. A maximal chain $\clc$ not starting at the bottom element
starts in a branch, consisting of the first $a$ elements of some
$\shift_i$ for $i>0$, and must contain all of those $a$ elements,
since the last one of these is the only one that is connected to the
core of the poset, at a vertex $v$, which is the letter in place $a+1$
in $\shift_i$.  But that letter, as shown above, also belongs to
$\shift_{i-1}$ and there lies to the right of place $a+1$.  Thus,
replacing the initial $a$ elements of the chain $\clc$ by all the
letters in $\shift_{i-1}$ up to and including the one in place $a+1$,
while retaining $v$ and the rest of $\clc$, we have lengthened the
chain $\clc$ by at least one element.  So a chain not starting at the
bottom element cannot have maximal length.  Analogous arguments show
the top element must belong to any chain of maximal length.
\end{proof}

The following are consequences of the proof of
Theorem~\ref{thm-branch-is-chain}.
\begin{porism}\label{porism-branches} Let $a$ and $b$ be as defined in
the statement of Theorem~\ref{thm-branch-is-chain}.
  \begin{itemize}
  \item Branches determined by the sequence $1,2,\ldots,a$ go
downwards from the core, that is, their elements are below the vertex
connecting them to the core, and they all have the same length $a$.
  \item Branches determined by the sequence $K-b+1,K-b+2,\ldots,K$ go
upwards from the core, that is, their elements are above the vertex
connecting them to the core, and they all have the same length $b$.
    \item If the prefix and suffix of $\pi$ between them contain both
$1$ and $K$, in particular if $m\geq K/2$, then $\ppmj$ has no
branches.
  \end{itemize}
\end{porism}

\begin{figure}[htbp]
\centering
\begin{tikzpicture}[scale=0.060, line width=1pt]
  \def\rvertex#1,#2,#3,{\draw [red, fill=white] (#1,#2) circle [radius=5] node [black] {$#3$};}
  \def\bvertex#1,#2,#3,{\draw [blue, fill=white] (#1,#2) circle [radius=5] node [black] {$#3$};}
\begin{scope}[shift={(0,80)}]
\draw node {%
$
\begin{array}{ccccccccccccc}
\shift_0=\;\;\;\;\blu3&\blu1&\blu4&\blu8&\blu2&\blu5&\blu9&\rd6&\rd7\\  \shift_1=\;\;\rd{10}&\blu8&\blu{11}&\blu{15}&\blu9&\blu{12}&\blu{16}&\rd{13}&\rd{14}\\   \shift_2=\;\;\rd{17}&\blu{15}&\blu{18}&\blu{22}&\blu{16}&\blu{19}&\blu{23}&\blu{20}&\blu{21}
\end{array}
$
};
\end{scope}
\begin{scope}[shift={(20,0)}, scale=0.85]
\draw [blue] (35,30)--(41,10); 
  \draw [red] (40,10)--(45,-10); 
  \draw [red] (65,80)--(85,0); 
  \draw [blue] (75,40)--(80,20); 
  \draw [red] (105,90)--(115,50); 
  \draw [blue] (115,50)--(120,30); 
  \draw [blue] (145,100)--(155,60); 

  \draw [blue] (0,0)--(5,-20); 
  \bvertex5,-20,3,
  \rvertex45,-10,10,
  \rvertex85,0,17,
  
  \draw [blue] (0,0)--(120,30); 
   
  \bvertex0,0,1,
  \bvertex20,5,4,
  \bvertex40,10,8,
  \bvertex60,15,11,
  \bvertex80,20,15,
  \bvertex100,25,18,
  \bvertex120,30,22,
   
  \draw [blue] (35,30)--(155,60);
  \bvertex35,30,2,
  \bvertex55,35,5,
  \bvertex75,40,9,
  \bvertex95,45,12,
  \bvertex115,50,16,
  \bvertex135,55,19,
  \bvertex155,60,23,
    
  \rvertex70,60,6,
  \rvertex65,80,7,
    
  \rvertex110,70,13,
  \rvertex105,90,14,
    
  \bvertex150,80,20,
  \bvertex145,100,21,
   
\end{scope}
\end{tikzpicture}
\caption{\label{fig-251368947} The poset for three occurrences of the pattern 25\rdd136\rdd{89}47 (whose inverse is 314825967) with overlap 2 and its defining stack.  The core is blue, branches red, and \rdd1 and \rdd{89} are the respective maximal sequences in the interfix of $\pi$ that determine branch lengths.}
\end{figure}
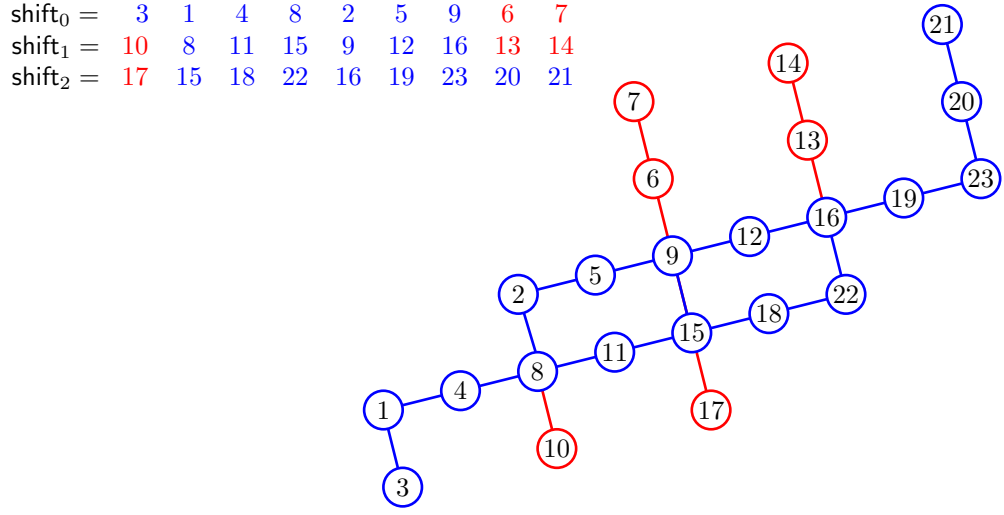

An \emph{upper bound} of two elements $a,b$ in a poset $\clp$ is an
element $x$ such that $a,b\le x$.  A \emph{least upper bound} of $a$
and $b$ is an upper bound $x$ such that $x\le y$ for any upper bound
$y$ of $a$ and $b$.  A \emph{greatest lower bound} is defined
analogously, with '$\le$' replaced by '$\ge$'.  An order theoretic
\emph{lattice} is a poset where each pair of elements has a least
upper bound and a greatest lower bound.

\begin{theorem}\label{thm-core-lattice} The core of a pattern poset
$\ppmj$ is a lattice.  In particular, if the $m$-suffix and $m$-prefix
of $\pi$ between them contain both 1 and $K$ then $\ppmj$ is a
lattice.  Also, the poset obtained from $\ppmj$ by adding an element
greater than all elements of $\ppmj$ and an element smaller than all
elements of $\ppmj$ is a lattice.
\end{theorem}
\begin{proof} Since the core has a bottom and a top, every pair of its
elements has an upper and a lower bound. We show, more generally, that
whenever two elements of $\ppmj$ have common upper bounds, they have a
least upper bound.

  Suppose $a,b$ have common upper bounds but no least upper bound.
They must therefore have two distinct common upper bounds $x$ and $y$
that are minimal with respect to being common upper bounds, and they
are necessarily incomparable. Choose four saturated chains, from $a$
to $x$, $a$ to $y$, $b$ to $x$ and $b$ to $y$.

  Every element other than $x,y$ on either chain from $a$ is
incomparable to every element other than $x,y$ on either chain from
$b$. Namely, if $u$ is on a chain from $a$ and $v$ on a chain from
$b$, with $u\le v$, then $v$ is a common upper bound for $a,b$
strictly below $x$ or $y$, contradicting minimality. The case $v\le u$
is analogous.

Consecutive elements of each chain appear together on some line of the
stack, since their comparison is a cover relation, as the chains are
saturated.  By Lemma~\ref{lem-successive-lines}, the lines containing
elements of either chain from $a$, with $x,y$ omitted, form an
interval of lines, for each chain meets consecutive lines, and the two
chains meet at $a$. The same holds for the chains from $b$. These two
intervals are disjoint, by the incomparability established
above. Assume, without loss of generality, that the interval for the
chains from $a$ precedes that for the chains from $b$.

Let $\shift_r$ be the last line of the former interval.  Each of $x,y$
occurs on a line in each interval, since it occurs together with its
predecessor on each of the two chains ending there. By
Lemma~\ref{lem-successive-lines}, both $x$ and $y$ therefore occur on
$\shift_r$. This makes them comparable, a contradiction. Hence $a,b$
have a least upper bound. Since the dual of a pattern poset is a
pattern poset this also applies to greatest lower bounds.

If the prefix and suffix between them contain both $1$ and $K$ then
the core is the whole poset, so $\ppmj$ is a lattice.

After adjoining $\widehat1$ and $\widehat0$, existing least upper and
greatest lower bounds remain unchanged. A pair with no common upper
bound now has least upper bound $\widehat1$, and a pair with no common
lower bound has greatest lower bound $\widehat0$, since these are
common bounds, and the only possible such. The resulting poset is
therefore a lattice.
\end{proof}

Note that Theorem~\ref{thm-core-lattice} in particular implies that in
a pattern poset two distinct elements cannot both cover (or be covered
by) another two distinct elements, since that would violate unique
least upper (and greatest lower) bounds.

\subsection{Chains and tree posets}

We now give necessary and sufficient conditions for a pattern packing
poset to be a tree, or a chain. Figures~\ref{fig-162534} and
\ref{fig-35124867} supply examples.  (Of course, the poset for one
occurrence of a pattern is always a chain.)

\begin{figure}[h]
  \vskip7ex
  \hspace*{2em}
\begin{tikzpicture} \draw node { $
\begin{array}{cccccccccc} \shift_0&=&1&3&5&6&4&2\\
\shift_1&=&5&7&9&10&8&6
\end{array} $ };
\end{tikzpicture}\\%
\vskip-7ex
\hspace*{11em}
\begin{tikzpicture}[scale=0.45, line width=1pt]
\def\vertex#1,#2,#3,{\draw [blue, fill=white] (#1,#2) circle
[radius=0.7] node [black] {$#3$};}
  \begin{scope}[shift={(0,0)}] \draw [blue] (0,0)--(18,9);
\vertex0,0,1, \vertex2,1,3, \vertex4,2,5, \vertex6,3,7, \vertex8,4,9,
\vertex10,5,10, \vertex12,6,8, \vertex14,7,6, \vertex16,8,4,
\vertex18,9,2,
\end{scope}
\end{tikzpicture}
\caption{The Hasse diagram of the pattern poset $\ppm(2)$, for
$\pi=162534$ (whose inverse is 135642) and $m=2$, with its defining
stack.}
\label{fig-162534}
\end{figure}

\begin{figure}[H] \vskip7ex \hspace*{2em}
\begin{tikzpicture} \draw node { $
\begin{array}{cccccccccc} \shift_0&=&3&4&1&5&2&7&8&6\\
\shift_1&=&9&10&7&11&8&13&14&12\\ \shift_2&=&15&16&13&17&14&19&20&18
\end{array} $ };
\end{tikzpicture}\\%
\vskip-12ex
\hspace*{2em}
\begin{tikzpicture}[scale=0.37, line width=1pt]
\def\vertex#1,#2,#3,{\draw [blue, fill=white] (#1,#2) circle
[radius=0.7] node [black] {$#3$};} \draw [blue] (0,0)--(30,10); \draw
[blue] (0,0)--(0,-6); \draw [blue] (9,3)--(9,-3); \draw [blue]
(18,6)--(18,0); \draw [blue] (15,5)--(15,8); \draw [blue]
(24,8)--(24,11); \draw [blue] (30,10)--(30,13); \vertex0,0,1,
\vertex3,1,5, \vertex6,2,2, \vertex9,3,7, \vertex12,4,11,
\vertex15,5,8, \vertex18,6,13, \vertex21,7,17, \vertex24,8,14,
\vertex27,9,19, \vertex30,10,20, \vertex15,8,6, \vertex24,11,12,
\vertex30,13,18, \vertex0,-3,4, \vertex0,-6,3, \vertex9,0,10,
\vertex9,-3,9, \vertex18,3,16, \vertex18,0,15,
\end{tikzpicture}
\caption{\label{fig-35124867} The Hasse diagram of the pattern poset
$\ppm(3)$, for $\pi=35124867$ (whose inverse is 34152786) and $m=2$,
with its defining stack.}
\end{figure}
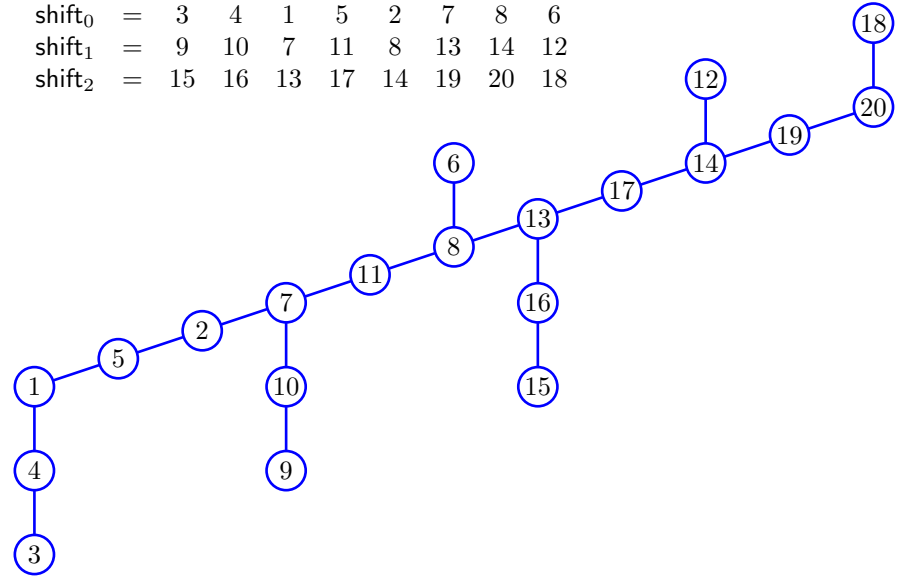

The following theorem characterizes the trees among pattern posets.
The idea of the proof is that having a cycle in the Hasse diagram,
seen as a graph, implies that some four elements, $a<x,y<b$, of that
cycle must satisfy the conditions in the theorem, and conversely.  We
first need a lemma.

A \emph{simple cycle} in a graph is a sequence of vertices, each
connected by edges to its predecessor and successor, and such that the
only repeated vertex in the sequence is the first and last one.  If a
graph has a cycle then it necessarily has a simple cycle. A chain $C$
is \emph{saturated} if no element of the poset lies strictly between
two consecutive elements of~$C$.  A chain $x_1<x_2<\cdots<x_n$ being
saturated is equivalent to each of its inequalities being a cover
relation.  Note that if $a<b$ is a cover relation in a pattern poset
then $a$ must immediately precede $b$ on some line of the stack.  The
converse is not true in general, because we might have another element
$x$ between $a$ and $b$ on some other line of the stack.

In the remainder of the paper, the \emph{interior} of a vector or
sequence of any sort is that thing itself take away its first and last
element.

\begin{lemma}\label{lemma-cycle-in-core} If a pattern poset $\clp$ is
not a tree then it has a simple cycle that consists of the two
elements $a$ and $b$ and two saturated chains $C$ and $D$ that are
disjoint apart from both having $a$ as minimum and $b$ as maximum, and
such that each element in the interior of $C$ is incomparable to every
element in the interior of $D$.
\end{lemma}
\begin{proof} If a pattern poset is not a tree then it has a simple
cycle, seen as a graph, and such a cycle must belong entirely to the
core of $\clp$.  Namely, a cycle that wanders onto a branch can only
be closed by leaving that branch again, via its only edge connected to
the core, and so would not be simple.

Let $a$ be the smallest element of the core of $\clp$ that is covered
by two different elements $x_1$ and $y_1$. (If $a$ is not the bottom
element of $\clp$ then it is the maximum element in the chain of
elements at the bottom of $\clp$ all of whose other elements are
covered by only one element.)  Since the core of $\clp$ is a lattice,
by Theorem~\ref{thm-core-lattice}, $x_1$ and $y_1$ have a least upper
bound $b$.  Then $\clp$ has two saturated chains
$$
 C: a<x_1<x_2<\cdots<x_r<b ~~~\mathrm{and}~~~ D:
a<y_1<y_2<\cdots<y_s<b,
$$
that are disjoint apart from $a$ and $b$, since if $x_i=y_\ell$ for
some $i$ and $\ell$ then $x_i$ would be an upper bound for $x_1$ and
$y_1$, so~$b$ would not be a least upper bound for $x_1$ and $y_1$.
Moreover, if $x_i<y_\ell$ for some $i$ and $\ell$ then $y_\ell$ would
be an upper bound for $x_1$ and $y_1$, so~$b$ would not be a least
upper bound, and analogously if $y_\ell<x_i$.
\end{proof}

We now present a theorem that completely characterizes pattern posets
that are trees, in terms of the stack.  Note that if the overlap $m$
is at most half the length $K$ of the pattern here then, by
Lemma~\ref{lemma-only-two-lines}, any elements $a,b$ that appear on
two lines of the stack appear only on two adjacent lines, which makes
the proof almost trivial.

\begin{theorem}\label{thm-posets-trees} A pattern poset $\ppmj$ has a
cycle if and only if there are two lines $\shift_s,\shift_t$, with
$s<t$, and elements $a,b,x,y$ satisfying the following conditions:
\begin{itemize}
\item No line contains both $x$ and $y$.
\item On $\shift_s$, $x$ lies between $a$ and $b$.
\item On $\shift_t$, $y$ lies between $a$ and $b$.
\item No element lies between $a$ and $b$ on two distinct lines among
$\shift_s,\shift_{s+1},\ldots,\shift_t$.
\end{itemize}
\end{theorem}

\begin{proof} Suppose first that there are elements and lines as
described.  We show that $x$ and $y$ are incomparable.

Any comparison between $x$ and $y$ can be established using only the
lines from $\shift_s$ to $\shift_t$. Indeed, any part of a comparison
passing through earlier lines has its entry and exit letters on
$\shift_s$, and can therefore be replaced by their comparison on that
line. The same argument applies to lines after $\shift_t$.

If $x$ and $y$ were comparable, a sequence of comparisons on these
lines would consequently establish $x<y$, or $y<x$.  Since no line
contains both, this sequence must pass from one line to another at
some intermediate element $z$. But $z$ then lies strictly between $a$
and $b$ on both lines, contrary to the last condition. Thus $x$ and
$y$ are incomparable. Since both lie between $a$ and $b$, there are
two distinct saturated chains from $a$ to $b$, whose union contains a
cycle. Hence $\ppmj$ is not a tree.

Conversely, suppose $\ppmj$ is not a tree. By
Lemma~\ref{lemma-cycle-in-core}, it has an interval that is not a
chain. Choose such an interval $[a,b]$ with as few elements as
possible. Every proper subinterval is then a chain.

Distinct maximal chains in $[a,b]$ can meet only at $a$ and $b$, and
their interior elements are pairwise incomparable.  Otherwise, a
meeting or comparison between them would give two distinct maximal
chains in a proper subinterval.  Since $a$ is covered by at most two
elements, $[a,b]$ therefore consists of exactly two saturated chains
$C$ and $D$, disjoint apart from their endpoints.

The lines containing interior elements of $C$ are consecutive: each
element occurs on consecutive lines, and consecutive elements of $C$
occur together on some line because their comparison is a cover
relation. The same holds for $D$.  These two sets of lines are
disjoint, since no line can contain incomparable elements. We may
assume that the lines containing interior elements of $C$ precede
those containing interior elements of $D$.

Let $\shift_s$ be the last line containing an interior element $x$ of
$C$, and let $\shift_t$ be the first line containing an interior
element $y$ of $D$. Each of $a$ and $b$ occurs on a line with an
interior element of each chain, by the cover relations at the ends of
the chains. Consequently, both $a$ and $b$ occur on every line from
$\shift_s$ to $\shift_t$.  Thus $x$ and $y$ lie between $a$ and $b$ on
their respective lines, and no line contains both.

Finally, every element strictly between $a$ and $b$ belongs to $C$ or
$D$. Among the lines from $\shift_s$ to $\shift_t$, an interior
element of $C$ can occur only on $\shift_s$, and an interior element
of $D$ only on $\shift_t$. No such element can therefore occur on two
of these lines. This proves the last condition and completes the
proof.
\end{proof}

An example illustrating why in the above theorem it is not enough to
examine just the two lines $\shift_s$ and $\shift_t$ is $\pi=13526748$
with overlap 5, whose 3-stack is
$$
\begin{array}{ccccccccccccc} \shift_0=\;\;&1&4&2&7&3&5&6&8\\
\shift_1=\;\;&4&7&5&10&6&8&9&11\\ \shift_2=\;\;&7&10&8&13&9&11&12&14
\end{array}
$$
With $a=7,b=8,x=3,y=10$ we see that (i) no line contains both 3 and
10, (ii) in $\shift_0$ 3 lies between 7 and 8, (iii) in $\shift_2$ 10
lies between 7 and 8, and (iv) on these two lines there is no shared
element between 7 and 8.  But, 5 lies between 7 and 8 in both
$\shift_0$ and $\shift_1$, so $3\prec5\prec10$ in $\pp{5}(3)$ and this
poset has no cycle, it is in fact a chain.

In~\cite[Section 2.2]{Dwyer-Elizalde-2018}, Dwyer and Elizalde
describe the posets we are discussing here in the case of overlap 1,
for patterns they call \emph{non-overlapping}, which are those whose
occurrences can overlap in at most one position.
\begin{corollary}\label{coro-olap-1-tree} For overlap 1 and any
$K$-pattern $\pi$, the poset $\ppmj$ is a tree.  It is a chain (for
any $j>1$) if and only if the prefix of $\pi$ is 1 and the suffix is
$K$ or vice versa.
\end{corollary}

\begin{proof} For overlap 1 and $i\ne \ell$, $\shift_i$ and
$\shift_\ell$ can have at most one element in common. So, there are no
pairs of repeated elements $a,b$ to form a cycle.  By
Proposition~\ref{prop-min-max-elts}, $\ppmj$ has a unique minimal
element and a unique maximal element if and only if the prefix and
suffix contain both 1 and $K$, and a tree is a chain if and only if it
has a unique minimal and a unique maximal element.
\end{proof}

Recall from Corollary~\ref{coro-unique-min-max} that if $\ppmj$ is a
tree and $m\ge K/2$ then $\ppmj$ is necessarily a chain.  We now give
necessary and sufficient conditions for this when $m<K/2$. This
characterization also holds for $m>K/2$, but assuming $m\le K/2$
simplifies the proof.

\begin{proposition}\label{prop-tree-is-chain} Let $\pi$ be a
$K$-pattern.  If $m\le K/2$ and $\ppmj$ is a tree then $\ppmj$ is a
chain if and only if it satisfies the following conditions:
\begin{itemize}
    \item If an element $x$ of $\ppmj$ is the leftmost repeated
element on two different lines of the stack then at least one of the
two sets of elements to the left of $x$ is empty.
    \item If an element $x$ of $\ppmj$ is the rightmost repeated
element on two different lines of the stack then at least one of the
two sets of elements to the right of $x$ is empty.
\end{itemize}
\end{proposition}
\begin{proof} The case $j=1$ is immediate. Let $j>1$. A tree poset is
a chain if and only if it has a unique minimal element and a unique
maximal element.
    
    If $\ppmj$ has a unique minimal element and a unique maximal
element then, by Proposition~\ref{prop-min-max-elts}, the $m$-prefix
and $m$-suffix of $\pi$ between them contain both 1 and $K$.  If 1 is
contained in the $m$-prefix then the first element of $\shift_0$ is at
most $m$, and consequently the first element

of $\shift_i$ is at most $m+i(K-m)$.  But the largest element of
$\shift_{i-1}$, for $i>0$, is $K+(i-1)(K-m)=m+i(K-m)$, so the first
element of $\shift_i$ also appears in $\shift_{i-1}$, and must
consequently be the leftmost repeated element in both, by
Lemma~\ref{lemma-same-order}.  Also, since $m\le K/2$, no element
appears on more than two lines of the stack, by
Lemma~\ref{lemma-only-two-lines}.  Thus the set of elements to its
left on one of the two lines is empty.  The cases when 1 belongs to
the suffix, and when $K$ belongs to the prefix or suffix, are
analogous.

    Conversely, suppose $\ppmj$ does not have a unique minimal
element.  Then, by Proposition~\ref{prop-min-max-elts}, 1 belongs
neither to the $m$-prefix nor the $m$-suffix of $\pi$.  Since 1 does
not belong to the $m$-prefix, the first element of $\shift_{0}$ is at
least $m+1$ and, in general, the first element of $\shift_{i}$ is at
least $m+1+i(K-m)$, which is greater than the greatest element of
$\shift_{i-1}$, which is $K+(i-1)(K-m)$. Thus, the leftmost element in
$\shift_{i}$ is not repeated in $\shift_{i-1}$.  Also, since 1 does
not belong to the $m$-suffix, the first element of $\shift_0$ is at
most $K-m$ and, in general, the first element of $\shift_{i-1}$ is at
most $K-m+(i-1)(K-m)=i(K-m)$, which is smaller than the smallest
element of $\shift_i$, which is $1+i(K-m)$.  Thus, the leftmost
element in $\shift_{i-1}$ is not repeated in $\shift_{i}$.

Therefore, the leftmost element $\ell$ that appears in both $\shift_i$
and $\shift_{i-1}$ is not in the first place in either, and thus
$\ell$ has an element to its left on both lines, and two such elements
can not be identical since that would constitute a different leftmost
repeated element. Both sets of elements preceding it are therefore
nonempty, contrary to the first condition. The case when $\ppmj$ does
not have a unique maximal element is analogous.
\end{proof} We now wish to express the above results, when $m\le K/2$,
in terms of a $K$-pattern $\pi$ rather than its stack.  To simplify
the argument we assume, without loss of generality (see
Remark~\ref{remark-symm-posets}), that the elements of the prefix and
suffix, respectively, of $\pi$ are in increasing order.  We then
define the \emph{gap vector} of the $m$-prefix
$\bp=\bp_1\bp_2\ldots\bp_m$ as
$$
\mathbf{dp}= (\bdp_1,\bdp_2,\ldots,\bdp_{m+1})=
(\bp_1-1,\bp_2-\bp_1-1,\ldots,\bp_m-\bp_{m-1}-1,K-\bp_m).
$$
The gap vector $\mathbf{ds}$ for the suffix
$\bs=\bs_1\bs_2\ldots\bs_m$ is defined in the analogous way, with
$\bs_i$ replacing $\bp_i$.  Note that $\bdp$ and $\bds$ both have
length $m+1$.

\begin{theorem}\label{thm-diff-vectors}
  Let $\pi$ be a $K$-pattern, $m\le K/2$ and $j\ge2$.  The poset $\ppmj$ is a
  tree if and only if the interiors of the gap vectors of the $m$-prefix and
  $m$-suffix of $\pi$ are empty or have dot product 0.  Also, $\ppmj$ is a chain
  if and only if the dot product of the gap vectors is 0.
\end{theorem}
\begin{proof} Note that in the definition of the gap vectors $\bdp$
and $\bds$, the $m$-prefix and the $m$-suffix are assumed to be
increasing.
  
The $\ell$-th coordinate $\bdp_\ell$ of the gap vector of the
$m$-prefix $\bp$ of $\pi$, when $1<\ell\le m$, counts the elements of
$\pi$ that are strictly between $\bp_{\ell-1}$ and $\bp_\ell$ in size,
and likewise for the $m$-suffix~$\bs$.  The $m$ elements that appear
both in $\shift_r$ and $\shift_{r+1}$ in the stack for $\ppmj$ appear
in places $s_1,s_2,\ldots,s_m$ in $\shift_r$ and in places
$p_1,p_2,\ldots,p_m$ in $\shift_{r+1}$.  Thus, between the element $a$
in place $s_i$ and the element $b$ in place $s_{i+1}$ in $\shift_r$,
we have $\bds_{i+1}$ elements, and between the element in place $p_i$,
which is necessarily $a$, and the element in place $p_{i+1}$,
necessarily $b$, in $\shift_{r+1}$, we have $\bdp_{i+1}$ elements.
Also, between $a$ and $b$ on these two lines of the stack there are no
repeated elements.

Thus, the elements $a$ and $b$ that appear in both $\shift_r$ and
$\shift_{r+1}$ (which have no repeated element between them) have some
elements between them on both lines if and only if
\mbox{$\bdp_{i+1}\cdot\bds_{i+1}\ne0$}.  This applies to all
coordinates of the interiors of $\bdp$ and $\bds$, and therefore, by
Theorem~\ref{thm-posets-trees}, $\ppmj$ is a tree if and only if the
interiors of $\bdp$ and $\bds$ have dot product 0 or are empty, as a
repeated element can only appear on two successive lines of the stack
when $m\le K/2$, by Lemma~\ref{lemma-only-two-lines}.

For $\ppmj$ to be a chain it must, in addition to the dot product of the
interiors of $\bdp$ and $\bds$ being 0, have a unique minimal element and a
unique maximal element.  By Proposition~\ref{prop-min-max-elts}, that is
equivalent to $\bp$ and $\bs$ containing, between them, both 1 and $K$. That, in
turn, is equivalent to $(\bdp_1=0 \text{ or } \bds_1=0)$ and
$(\bdp_{m+1}=0 \text{ or } \bds_{m+1}=0)$, which is equivalent to
$\bdp_1\cdot\bds_1=\bdp_{m+1}\cdot\bds_{m+1}=0$, in which case the dot product
of the entire gap vectors $\bdp$ and $\bds$ is 0.
\end{proof}

As an example, if $\pi=124635$ and $m=2$ (see Figure~\ref{fig-124635})
then $\mathbf{dp}=(0,0,4)$ and $\mathbf{ds}=(2,1,1)$. The vector
$\mathbf{dp}$ thus gives the number of elements in $\pi$ that lie
strictly between successive entries of $\bp$ in size, except that its
first entry is the number of elements in $\pi$ strictly smaller than
the first (and smallest) element of $\mathbf{p}$ and the last entry is
the number of elements strictly greater than the last (and largest)
element of $\bp$, and likewise for $\mathbf{ds}$.  Looking at the
first two lines of the stack for $\pi=124635$ in
Figure~\ref{fig-124635}, we see that $\mathbf{ds}=(2,1,1)$ corresponds
to the fact that on the first line of the stack there are two elements
before the first repeated element 5, one element between 5 and the
other repeated element, the 6, and one element after 6 on the first
line.  Similarly, $\mathbf{dp}=(0,0,4)$ corresponds to the fact that
on the second line there is no element preceding 5, no element between
the 5 and the 6 and four elements after the~6.  The dot product of the
interiors is $1\cdot0=0$, and the poset in question is a tree.  It is
not a chain, since the last coordinates of both gap vectors are
positive. If we were to switch 5 and 6 in $\pi$, however, the
resulting poset would become a chain, and $\bds=(2,2,0)$.
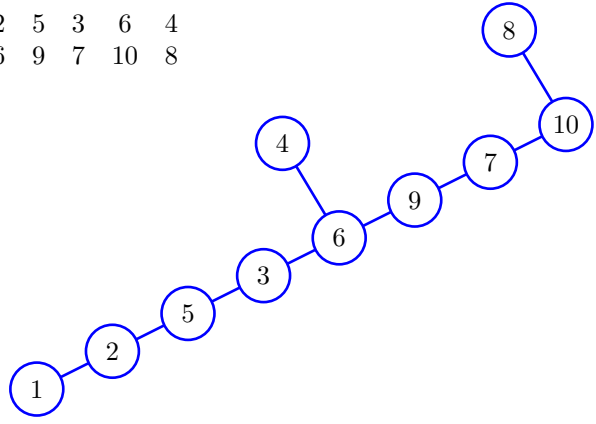
\begin{figure}[h] \vskip7ex \hspace*{2em}
\begin{tikzpicture} \draw node { $
\begin{array}{cccccccccc} \shift_0&=&1&2&5&3&6&4\\
\shift_1&=&5&6&9&7&10&8
\end{array} $ };
\end{tikzpicture}\\%
\vskip-10ex \hspace*{10em}
\begin{tikzpicture}[scale=0.5, line width=1pt]
\def\vertex#1,#2,#3,{\draw [blue, fill=white] (#1,#2) circle
[radius=0.7] node [black] {$#3$};}
  \begin{scope}[shift={(0,0)}] \draw [blue] (0,0)--(14,7); \draw
[blue] (8,4)--(6.5, 6.5); \draw [blue] (14,7)--(12.5, 9.5);
\vertex0,0,1, \vertex2,1,2, \vertex4,2,5, \vertex6,3,3, \vertex8,4,6,
\vertex6.5,6.5,4, \vertex10,5,9, \vertex12,6,7, \vertex14,7,10,
\vertex12.5,9.5,8,
\end{scope}
\end{tikzpicture}
\caption{The Hasse diagram of the pattern poset $\ppm(2)$, for
$\pi=124635$}
\label{fig-124635}
\end{figure}


If $1$ belongs to the prefix (or $K$ to the suffix) of a $K$-pattern
whose poset is a tree then the enumeration of linear extensions has a
product formula:

\begin{proposition}\label{prop-gen-prod-formula-trees} Let $\pi$ be a
$K$-pattern with 1 in its prefix, whose poset is a tree. Let
$v_1<\cdots<v_t$ be the vertices of the core that branches are
attached to, let $\ell$ be the length of the branches, and let $N_i$
be the number of vertices strictly above $v_i$ and not on the branch
attached to $v_i$. Then
$$
 |\cls_\pi^m(j)| = \prod_{i=1}^t\binom{N_i+\ell}{\ell}.
$$  
\end{proposition}
\begin{proof} It follows from Porism~\ref{porism-branches}, since 1 is
in the prefix, that all branches go upwards from the core, which in a
tree is necessarily a chain.  Now, everything outside the lowest
branch is either weakly below its attachment to the core or above that
vertex, so its elements must in a linear extension come after the
former, but can be placed arbitrarily with respect to the latter.
This gives the first factor, and induction gives the formula.
\end{proof}

For a consecutive prefix the numbers $N_i$ have a particularly simple
form, giving the following product.  Recall that if $m\ge K/2$ then a
corresponding tree poset is a chain, so nothing is lost by the
restriction $m\le K/2$ here:

\begin{proposition}\label{prop-prod-formula-trees} Let $\pi$ be a
pattern with prefix $12\cdots m$ and increasing suffix ending in $s$,
where $m\le K/2$.  Let $d=K-m$ and $\ell=K-s$. Then
$$
 |\cls_\pi^m(j)| = \prod_{i=0}^{j-1}\binom{id+\ell}{\ell}.
$$
\end{proposition}
\begin{proof} Since 1 is in the $m$-prefix of $\pi$, all branches of
$\ppmj$ go upwards, and each has length $\ell$, by
Porism~\ref{porism-branches}.  Because the branches rise at regular
intervals from the core (which in a tree is a chain), to count linear
extensions we may assume that the $j$-th occurrence adds a branch
below the previous branches rather than above them.  The elements of
that branch are incomparable to all $(j-1)d$ elements above the node
$N$ of the core that the branch is attached to, and in a linear
extension the $\ell$ elements of the branch can be placed anywhere
with respect to those elements, but must lie above all elements below
and including $N$.  This gives the factor $\binom{(j-1)d+\ell}{\ell}$
and a straightforward inductive argument completes the proof.
\end{proof}

The special case of Proposition~\ref{prop-prod-formula-trees} when
$m=1$ is Proposition~26 in~\cite{Dotsenko-Khoroshkin-2013}, see also
Figure 12 in~\cite{Elizalde-Noy-2012}.

In particular, the proposition applies to every overlap-$1$ pattern
whose first letter is $1$.

\begin{definition}\label{def-sfact} For integers $n\ge1$ and $k\ge1$,
the \emph{$k$-shift factorial of $n$} is $\displaystyle
n\sfact{k}:=\prod_{0\le i<n/k}(n-ik)$.
\end{definition} Thus, $n\sfact{k}$ is the product of all positive
integers in the sequence $n,(n-k),(n-2k),\cdots$.

For overlap 1 we get the following special case:
\begin{corollary}\label{coro-1-k-1} Let $\pi$ be a $K$-pattern with
$\pi_1=1$ and $\pi_K=K-1$.  The number of permutations that
$(1,j)$-pack $\pi$ is $n\sfact{K-1}$, where $n=K+(j-2)(K-1)$.
\end{corollary}

An example of Proposition~\ref{prop-prod-formula-trees} is the pattern
125634 with overlap 2, where $d=4$, $\ell=2$, whose poset is shown in
Figure~\ref{fig-125634-163425-167359248}.  The same figure shows the
poset for the pattern 163425 with overlap 2.  An argument similar to
the one in Proposition~\ref{prop-prod-formula-trees}, also analyzing
the effect on the poset of adding a new occurrence, shows that all
$K$-patterns $\pi$ with $m$-prefix $123\ldots(m-1)K$ and any
increasing $m$-suffix have enumeration of the same form as in
Proposition~\ref{prop-prod-formula-trees}.  Namely, a straightforward
analysis of the stack of the poset for such a pattern shows that each
new occurrence adds $K-m$ elements to the right edge of the Hasse
diagram, when drawn as in Figure~\ref{fig-125634-163425-167359248}.

As with the argument for the tree, we can instead add a new first
occurrence. Its $\ell$ letters between $K-1$ and $K$ form a chain
incomparable to the $(j-1)d$ old elements strictly between these
vertices. All remaining new letters are comparable to every element.
Thus the new factor is $\binom{(j-1)d+\ell}{\ell}$.

\begin{proposition}\label{prop-prod-formula-cones} Let $m\ge2$ and let
$\pi$ be a $K$-pattern with $m$-prefix $123\ldots(m-1)K$ and
increasing $m$-suffix ending in $ab$, and let $\ell=b-a-1$.  Then,
letting $d=K-m$, we have
$$
|\cls_\pi^m(j)| = \prod_{i=0}^{j-1} \binom{id + \ell}{\ell}.
$$
\end{proposition}

\begin{figure}[H] \centering
\begin{tikzpicture}[scale=0.42, every node/.style={font=\small},
nodecircle/.style={draw=blue, line width=1.2pt, circle, inner
sep=1.5pt, minimum size=4.5mm, fill=white, text=black,
font=\footnotesize}, posedge/.style={blue, line width=1.2pt, shorten
>=1pt, shorten <=1pt}]
\begin{scope}[xshift=-180mm, yshift=60mm] \node[nodecircle] (n5) at
(3,1.5) {5}; \node[nodecircle] (n9) at (6,3) {9}; \node[nodecircle]
(n13) at (9,4.5) {13}; \node[nodecircle] (n11) at (9,6.0) {11};
\node[nodecircle] (n12) at (9,7.5) {12}; \node[nodecircle] (n14) at
(9,9.0) {14}; \node[nodecircle] (n10) at (6,10.5) {10};
\node[nodecircle] (n6) at (3,12.0) {6}; \node[nodecircle] (n1) at
($(n5)+(0,-3)$) {1}; \node[nodecircle] (n2) at ($(n6)+(0, 3)$) {2};
\node[nodecircle] (n3) at ($(n5)!0.33!(n6)$) {3}; \node[nodecircle]
(n4) at ($(n5)!0.66!(n6)$) {4}; \node[nodecircle] (n7) at
($(n9)!0.33!(n10)$) {7}; \node[nodecircle] (n8) at ($(n9)!0.66!(n10)$)
{8}; \draw[posedge] (n1)--(n5); \draw[posedge] (n5)--(n3);
\draw[posedge] (n3)--(n4); \draw[posedge] (n4)--(n6); \draw[posedge]
(n5)--(n9); \draw[posedge] (n9)--(n7); \draw[posedge] (n7)--(n8);
\draw[posedge] (n8)--(n10); \draw[posedge] (n9)--(n13); \draw[posedge]
(n13)--(n11); \draw[posedge] (n11)--(n12); \draw[posedge]
(n12)--(n14); \draw[posedge] (n2)--(n6); \draw[posedge] (n6)--(n10);
\draw[posedge] (n10)--(n14);
\end{scope}
\begin{scope}[xshift=-60mm, yshift=60mm] \def\dx{1} \def\dy{1.5}
\def\vsep{2.25} \node[nodecircle] (m1) at (0*\dx, 0*\dy) {1};
\node[nodecircle] (m2) at (1*\dx, 1*\dy) {2}; \node[nodecircle] (m5)
at (2*\dx, 2*\dy) {5}; \node[nodecircle] (m6) at (3*\dx, 3*\dy) {6};
\node[nodecircle] (m9) at (4*\dx, 4*\dy) {9}; \node[nodecircle] (m10)
at (5*\dx, 5*\dy) {10}; \node[nodecircle] (m13) at (6*\dx, 6*\dy)
{13}; \node[nodecircle] (m14) at (7*\dx, 7*\dy) {14};
\node[nodecircle] (m3) at ($(m6) + (0, \vsep)$) {3}; \node[nodecircle]
(m4) at ($(m6) + (0, 2*\vsep)$) {4}; \node[nodecircle] (m7) at ($(m10)
+ (0, \vsep)$) {7}; \node[nodecircle] (m8) at ($(m10) + (0, 2*\vsep)$)
{8}; \node[nodecircle] (m11) at ($(m14) + (0, \vsep)$) {11};
\node[nodecircle] (m12) at ($(m14) + (0, 2*\vsep)$) {12};
\draw[posedge] (m1) -- (m2); \draw[posedge] (m2) -- (m5);
\draw[posedge] (m5) -- (m6); \draw[posedge] (m6) -- (m9);
\draw[posedge] (m9) -- (m10); \draw[posedge] (m10) -- (m13);
\draw[posedge] (m13) -- (m14); \draw[posedge] (m6) -- (m3);
\draw[posedge] (m3) -- (m4); \draw[posedge] (m10) -- (m7);
\draw[posedge] (m7) -- (m8); \draw[posedge] (m14) -- (m11);
\draw[posedge] (m11) -- (m12);
\end{scope}
\begin{scope}[xshift=100mm, xscale=-1, scale=1] \node[nodecircle]
(N20) at (11,8) {20}; \node[nodecircle] (N7) at (3,1.9) {7};
\node[nodecircle] (N1) at ($(N7)+(-2,-0.8)$) {1}; \node[nodecircle]
(N19) at (11,5.1) {19}; \node[nodecircle] (N16) at ($(N19)!0.5!(N20)$)
{16}; \node[nodecircle] (N13) at (7,3.5) {13}; \node[nodecircle] (N10)
at (7,6.75) {10}; \node[nodecircle] (N20) at (11,8) {20};
\node[nodecircle] (N17) at (9,9.0) {17}; \node[nodecircle] (N14) at
(7,10) {14}; \node[nodecircle] (N11) at (5,11) {11}; \node[nodecircle]
(N8) at (3,12) {8}; \node[nodecircle] (N5) at ($(N8)+(-1.2, 1.4)$)
{5}; \node[nodecircle] (N2) at ($(N5)+(-1.2, 1.4)$) {2};
\node[nodecircle] (N3) at ($(N2)+(-1.2, 1.4)$) {3}; \node[nodecircle]
(N9) at ($(N3)+(-1.2, 1.4)$) {9}; \node[nodecircle] (N15) at
($(N9)+(-1.2, 1.4)$) {15}; \node[nodecircle] (N21) at ($(N15)+(-1.2,
1.4)$) {21}; \node[nodecircle] (N6) at ($(N9)+(1.4, 1.2)$) {6};
\node[nodecircle] (N12) at ($(N15)+(1.4, 1.2)$) {12};
\node[nodecircle] (N18) at ($(N21)+(1.4, 1.2)$) {18};
\node[nodecircle] (N4) at ($(N7)!0.5!(N8)$) {4}; \draw[posedge]
(N5)--(N2); \draw[posedge] (N2)--(N3); \draw[posedge] (N3)--(N9);
\draw[posedge] (N9)--(N15); \draw[posedge] (N15)--(N21);
\draw[posedge] (N9)--(N6); \draw[posedge] (N15)--(N12); \draw[posedge]
(N21)--(N18); \draw[posedge] (N1)--(N7); \draw[posedge] (N7)--(N4);
\draw[posedge] (N4)--(N8); \draw[posedge] (N7)--(N13); \draw[posedge]
(N13)--(N10); \draw[posedge] (N10)--(N14); \draw[posedge]
(N13)--(N19); \draw[posedge] (N19)--(N16); \draw[posedge]
(N16)--(N20); \draw[posedge] (N20)--(N17); \draw[posedge] (N5)--(N8);
\draw[posedge] (N8)--(N11); \draw[posedge] (N14)--(N11);
\draw[posedge] (N14)--(N17);
\end{scope} \node[anchor=north west] (A) at ($(current bounding
box.north west)+(-5mm,32mm)$) { \setlength{\arraycolsep}{1pt} $
\begin{array}{cccccccc} \shift_0&=&1&5&3&4&6&2\\
\shift_1&=&5&9&7&8&10&6\\ \shift_2&=&9&13&11&12&14&10
\end{array} $ }; \node (B) at ($(A.east)+(50mm,0)$) {
\setlength{\arraycolsep}{1pt} $
\begin{array}{cccccccc} \shift_0&=&1&2&5&6&3&4\\
\shift_1&=&5&6&9&10&7&8\\ \shift_2&=&9&10&13&14&11&12
\end{array} $ }; \node (C) at ($(B.east)+(75mm,0)$) {
\setlength{\arraycolsep}{1.7pt} $
\begin{array}{ccccccccccc} \shift_0&=&1&7&4&8&5&2&3&9&6\\
\shift_1&=&7&13&10&14&11&8&9&15&12\\
\shift_2&=&13&19&16&20&17&14&15&21&18
\end{array} $ };
\end{tikzpicture}
\caption{\label{fig-125634-163425-167359248} Three non-isomorphic
pattern posets---for (left to right) 163425 and 125634 with overlap~2,
and 167359248 with overlap~3---all with the same enumeration for each
number of occurrences.  An inductive argument for each of these,
analyzing the effect of adding an occurrence, gives the common product
$\prod_{i=0}^{j-1}\binom{4i+2}{2}$, see
Propositions~\ref{prop-gen-prod-formula-trees}
and~\ref{prop-prod-formula-trees}.}
\end{figure}
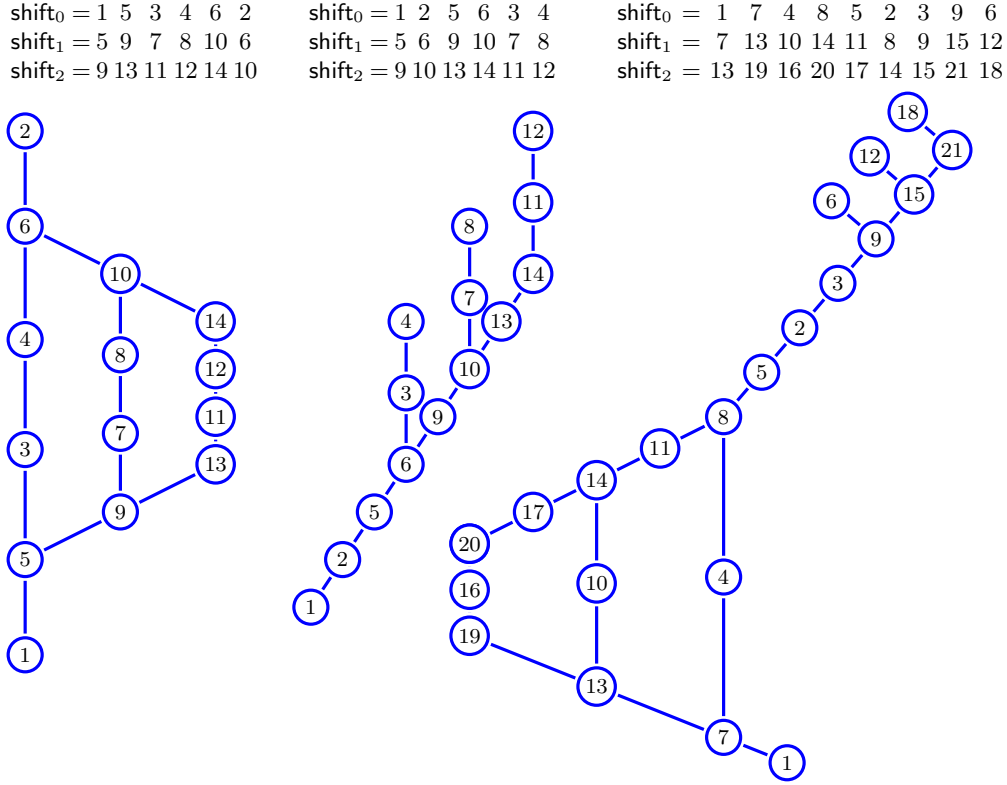

\subsection{Shifted grid posets}

\enlargethispage{2ex}

There is a standard partial order on the integer lattice $\ZZ^2$ where
$(x,y)$ covers both $(x-1,y)$ and $(x,y-1)$ for all $x$ and $y$.  Its
Hasse diagram can be obtained by rotating this lattice $45$ degrees
counterclockwise.  We say that the Hasse diagram $\clh$ of a poset
$\clp$ can be embedded in $\ZZ^2$ if $\clh$ can be obtained from a
connected subset of the Hasse diagram of $\ZZ^2$ by adding vertices to
some of its edges.  An example is given in Figure~\ref{fig-146378259}.

We now define a particular class of posets that can be embedded in
$\ZZ^2$ and coincides with a certain class of pattern posets.

\begin{definition}\label{def-shifted-grid} Fix $m\le K/2$, $J\ge1$
  and nonnegative integers $h_1,\ldots,h_m,v_1,\ldots,v_{m-1}$.  The
  \emph{shifted grid poset of width $J$ with parameters
    $(h_1,\ldots,h_m:v_1,\ldots,v_{m-1})$} has as Hasse diagram the horizontal
  lattice segments in the integer lattice $\ZZ^2$ from $(r,r)$ to $(r+J,r)$,
  $1\le r\le m$, and the vertical lattice edges between them, directed
  rightwards and upwards, with $h_r$ vertices inserted in every horizontal unit
  edge at height $r$, and $v_r$ vertices in every vertical unit edge from height
  $r$ to $r+1$.
\end{definition}

An example of a \sgp\ and a pattern of the type mentioned in the
following theorem is given in Figure~\ref{fig-146378259}. The
corresponding parameters are $h_1=h_2=0,h_3=2$ and $v_1=1,v_2=0$.

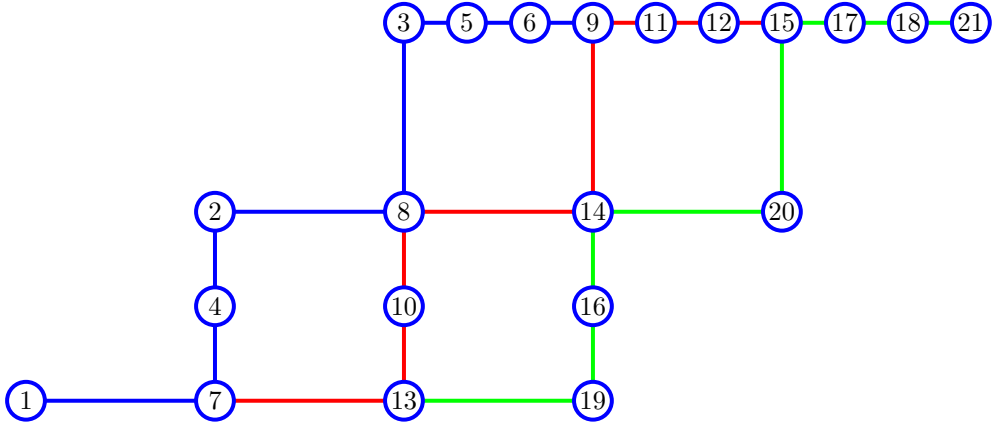
\begin{figure}[H] \setlength{\abovecaptionskip}{-29ex} \hspace*{0.5em}

\begin{tikzpicture}[scale=2.5, line width=1.5pt]
\def\vertex#1,#2,#3,{\draw [blue, fill=white] (#1,#2) circle
  [radius=0.1] node [black] {$#3$};}
\begin{scope}[xshift=-9.3em, yshift=-3ex]
  \draw [blue] (0,0)--(1,0);
  \draw [red] (1,0)--(2,0);
  \draw [green] (2,0)--(3,0);
  \draw [blue]   (1,1)--(2,1);
  \draw [red] (2,1)--(3,1);
  \draw [green] (3,1)--(4,1);

  \draw [blue] (2,2)--(3,2);
  \draw [red] (3,2)--(4,2);
  \draw [green]  (4,2)--(5,2);
  \draw [blue] (1,0)--(1,1);
  \draw [red] (2,0)--(2,1);

  \draw [blue] (2,1)--(2,2);
  \draw [green] (3,0)--(3,1);
  \draw [red]   (3,1)--(3,2);
  \draw [green] (4,1)--(4,2);
  \vertex0,0,1,
  \vertex1,0,7,

  \vertex2,0,13,
  \vertex3,0,19,
  \vertex2,2,3,
  \vertex7/3,2,5,

  \vertex8/3,2,6,
  \vertex3,2,9,
  \vertex10/3,2,11,
  \vertex11/3,2,12,

  \vertex4,2,15,
  \vertex13/3,2,17,
  \vertex14/3,2,18,
  \vertex5,2,21,

  \vertex1,1,2,
  \vertex2,1,8,
  \vertex3,1,14,
  \vertex4,1,20,

  \vertex1,1/2,4,
  \vertex2,1/2,10,
  \vertex3,1/2,16,
\end{scope}
\begin{scope}
  \draw node[xshift=-13em, yshift=-2ex] { $
\begin{array}{ccccccccccc}
  \shift_0&=&1&7&4&2&8&3&5&6&9\\
  \shift_1&=&7&13&10&8&14&9&11&12&15\\
  \shift_2&=&13&19&16&14&20&15&17&18&21\\[75ex]
\end{array} $ };
\end{scope}
\end{tikzpicture}
\caption{\label{fig-146378259} The {\sgp} with parameters
$(0,0,2:1,0)$ for three occurrences of the pattern $\pi=146378259$
(whose inverse is 174283569) with overlap $3$ and its defining
stack. Blue edges correspond to $\shift_0$, red edges to $\shift_1$
and green edges to $\shift_2$.  Shown embedded in the $\ZZ^2$-grid
with 1 at $(1,1)$, 21 at $(6,3)$. Rotating this picture $45$ degrees
counterclockwise gives the Hasse diagram of $\pp{3}(3)$.  }
\end{figure}

\begin{lemma}\label{lemma-homology} Let $\pi$ be a $K$-pattern, $m\le
K/2$ and let $\bdp$ and $\bds$ be the interiors of the gap vectors of
its $m$-prefix $\bp$ and $m$-suffix $\bs$.  The number of simple
cycles in $\ppm(2)$ equals the number of indices $i$ such that
$\bdp_i$ and $\bds_i$ are both positive.
\end{lemma}
\begin{proof} A cycle in $\ppm(2)$ must have elements from both the
first and second lines of its stack, since each line forms a chain in
$\ppm(2)$.  The minimal and maximal elements of such a cycle must
belong to both lines, and thus can only be the letters
$K-m+1,K-m+2,\ldots,K$.

By symmetry we may assume that $\bp$ and $\bs$ are increasing, so
these common letters appear in increasing order on both lines, call
them $c_1<c_2<\cdots<c_m$.  Each common letter separates the preceding
gaps from the following ones, so a simple cycle cannot extend across
an intermediate common letter.

In order for there to be a simple cycle with two of these elements, $c_i$ and
$c_{i+1}$, as minimum and maximum, there must be, on each of the two lines, a
nonempty set of elements strictly between $c_i$ and~$c_{i+1}$.  On the first
line, the places of $c_i$ and $c_{i+1}$ are determined by $\bs_i$ and
$\bs_{i+1}$, which must differ by at least 2 in order for there to be some
elements strictly between $c_i$ and $c_{i+1}$ on the first line, which is
equivalent to $\bds_i>0$.  Conversely, if $\bds_i=0$ then $c_i$ and $c_{i+1}$
are adjacent on the first line, so there are no letters between them that can
form one of the two chains of a cycle in $\ppm(2)$.

On the second line, the places of $c_i$ are determined by $\bp$, and an
analogous argument shows that $\bdp_i$ must be positive in order for there to be
elements strictly between $c_i$ and~$c_{i+1}$ that contribute to a cycle in
$\ppm(2)$.
\end{proof}
\begin{remark} The number of indices $i$ such that both $\bdp_i$ and
$\bds_i$ are positive is the \emph{first Betti number} of the graph of
$\ppm(2)$ seen as a simplicial complex, that is, equals the rank of
its first homology group.
\end{remark}

\begin{lemma}\label{lemma-1-K-sgp}
  Let $1<m\le K/2$. A $K$-pattern $\pi$ whose poset $\ppmj$ is isomorphic, for
  every $j$, to an $m$-level shifted grid poset of width $j$ with fixed
  parameters must have 1 in its prefix and $K$ in its suffix, or vice versa.
\end{lemma}

\begin{proof}
  Note first that a shifted grid poset of width 2 has precisely $m-1$ simple
  cycles, the maximum possible for overlap $m$, the $i$th cycle having bottom
  and top vertices at $(i+1,i)$ and $(i+2,i+1)$, respectively, for $1\le i<m$.
  Thus, by Lemma~\ref{lemma-homology}, the interiors of the gap vectors of both
  $m$-prefix and $m$-suffix of $\pi$ must be positive.
  
  A \sgp\ has a unique minimal and a unique maximal element, so, by
Proposition~\ref{prop-min-max-elts}, a $K$-pattern $\pi$ with such a
poset must have 1 and $K$ in its $m$-prefix or $m$-suffix. Suppose 1
and $K$ both belong to the $m$-prefix (the case for the suffix is
analogous).  Then, by Proposition~\ref{prop-min-max-elts}, the bottom
element of $\ppmj$, for any $j$, is the first letter of $\shift_0$ and
the top letter is the last letter of $\shift_0$.  We show that, for
any $j$, there is a path---a saturated chain---from bottom to top
element consisting precisely of the letters of $\shift_0$, in that
order.

The positivity of the interior of the gap vector of the suffix implies
that no two letters shared with $\shift_1$ are consecutive in
$\shift_0$. Thus every consecutive pair in $\shift_0$ contains a
letter occurring on no other line. Such a pair forms a cover relation,
since any comparison involving that letter must pass through its
predecessor or successor in $\shift_0$.  Hence $\shift_0$ is saturated
for every $j$.

In a \sgp\ the length of the shortest path from bottom to top element
clearly increases with the number of occurrences $j$.

Thus, 1 and $K$ cannot both belong to the prefix or both to the
suffix, so the prefix and suffix must each contain one of 1 and $K$.
\end{proof}

\begin{theorem}\label{thm-shifted-grid}
  Let $1<m\le K/2$ and let $\pi$ be a $K$-pattern whose $m$-prefix
  $\bp=p_1p_2\ldots p_m$ and $m$-suffix $\bs=s_1s_2\ldots s_m$ satisfy
\begin{equation}
  \label{eq-interleave} 1=p_1<s_1<p_2<s_2<\cdots<p_m<s_m=K.
\end{equation} Then $\ppm(j)$ is a \sgp\ of width $j$ with parameters
 $$
 (h_1,h_2,\ldots,h_m : v_1, v_2,\ldots,v_{m-1}),
 $$
 where
 \begin{equation}\label{eq-recover}
   h_i=s_i-p_i-1 \text{ and } v_i=p_{i+1}-s_i-1.
 \end{equation}
 Conversely, given a sequence of {\sgps} $\clp_j$ of width $j$ with the above
 parameters, let $K$ be the sum of the $h_i$ and $v_i$ plus $2m$, and define
 $\bp$ and $\bs$ recursively by~\eqref{eq-recover}, starting with $p_1=1$. Then
 $s_m=K$.  Let $\pi=\bp\bw\bs$, where $\bw$ is any permutation of the letters of
 $[K]$ not in $\bp$ or $\bs$. Then, up to the trivial symmetries in
 Section~\ref{sec-symmetries}, $\pi$ is the only $K$-pattern with $\ppm(j)$
 isomorphic to $\clp_j$ for every $j$.
\end{theorem}

\begin{proof} The case $j=1$ is trivial, so assume $j\ge2$.  The first
occurrence of $\pi$ corresponds to a chain in $\ppm(j)$ whose order is
determined by $\shift_0$.  The elements where the chains determined by
$\shift_0$ and $\shift_1$ meet are $K-m+1,K-m+2,\ldots,K$, as these
are precisely the letters shared by $\shift_0$ and $\shift_1$.

  Because of the interleaving conditions in~\eqref{eq-interleave}
above, the letters $1,2,\ldots,m$ and $K-m+1,K-m+2,\ldots,K$ appear in
alternating order in $\shift_0$ (which is the inverse of $\pi$):
$$
1\prec (K-m+1)\prec 2\prec (K-m+2)\prec \cdots \prec (K-1) \prec
m\prec K,
$$
where $a\prec b$ indicates that $a$ precedes $b$.

Embed the letters $1,2,\ldots,m$ and $K-m+1,K-m+2,\ldots,K$ in the
$\ZZ^2$-grid so that $i$ is placed at $(i,i)$ and $K-m+i$ at
$(i+1,i)$, so that the chain $1, (K-m+1), 2, (K-m+2), \cdots, (K-1),
m, K$ forms a zig-zag path starting at $(1,1)$ and alternatingly turns
up and right after each unit step.

Between the letters $i$ and $K-m+i$ in $\shift_0$ there may be some
letters.  The number of such letters is $s_i-p_i-1$ and the number of
letters between $K-m+i$ and $i+1$ is $p_{i+1}-s_i-1$.  The set of
these letters between $i$ and $K-m+i$ in $\shift_0$ we place, in the
order determined by $\shift_0$, on the horizontal line segment between
the places of $i$ and $K-m+i$ in the $\ZZ^2$-lattice, and we likewise
place the letters between $K-m+i$ and $i+1$ in $\shift_0$ on the
vertical line segment between $K-m+i$ and $i+1$.

We now construct the analogous zig-zag path in the $\ZZ^2$-grid for
the letters in $\shift_1$, but starting from $(2,1)$.  Note that the
first letter in $\shift_1$ is $K-m+1$, which was placed at $(2,1)$
when we placed the letters of $\shift_0$.  Similarly, the letter in
place $p_2$ of $\shift_1$ is $K-m+2$, which we place at $(3,2)$, as we
did when placing the letters of $\shift_0$, and this holds in general,
so the letters in the zig-zag chain for $\shift_1$ that coincide with
letters in $\shift_0$ are placed at the same points in $\ZZ^2$ as
their counterparts in $\shift_0$.

Now, $\shift_1$ is order isomorphic to $\shift_0$, and the path in
$\ZZ^2$ containing the letters of $\shift_1$ is identical to the path
for $\shift_0$ shifted one unit to the right.  Thus horizontal
segments at the same height contain the same number of interior
elements, and likewise for vertical segments connecting the same two
horizontal rows.  Repeating this construction for the remaining lines
gives the required shifted grid of width $j$.

For the converse it may be helpful to consider the \sgp\ in
Figure~\ref{fig-146378259}, where $(K-m+1,K-m+2,\ldots,K)$, the letters
belonging to both $\shift_0$ and $\shift_1$, are $(7,8,9)$.

Summing the defining identities in~\eqref{eq-recover} gives that the sum of the
$h_i$ and $v_i$ equals $s_m-2m,$ since $p_1=1$. Thus $s_m=K$ by the definition
of $K$.

Let $\tau$ be a pattern whose poset $\ptnj$ is isomorphic to $\clp_j$
for every $j$, and let $\pi$ be the pattern defined in the statement
of the theorem. Note that for $\ptnj$ and $\ppmj$ to have the same
number of elements for all $j$ (indeed for just $j=1,2$), we must have
$m=n$ and $\pi$ and $\tau$ must have the same length $K$.

Let $\bp'=p'_1p'_2,\ldots p'_m$ and $\bs'=s'_1s'_2,\ldots, s'_m$ be
the $m$-prefix and $m$-suffix, respectively, of~$\tau$, both of which
we assume to be increasing, without loss of generality.  We will show
that $\bp'=\bp$ and $\bs'=\bs$, which will complete the proof. We may
assume, by Lemma~\ref{lemma-1-K-sgp} and symmetry, that $p'_1=1$ and
$s'_m=K$ (by definition we also have $\pi=1\ldots K$).  Let $\clh$ be
the embedding of the Hasse diagram of $\ptmj$ in $\ZZ^2$, identical to
the embedding of the isomorphic \sgp\ $\ppmj$.

Define the \emph{tail} of a \sgp\ to be the chain of elements to the
left of the element at $(2,1)$ in $\clh$. (The head is defined
analogously, from the top element.)  The Hasse diagram $\clh$ of a
\sgp\ $\ppm(2)$ consists, apart from tail and head, of $m-1$ cycles,
one for each $a$ with $1\le a\le m-1$, containing the elements $K-m+a$
and $K-m+a+1$ together with two disjoint chains between them.  This is
a consequence of the interiors of the gap vectors of both the
$m$-prefix and $m$-suffix being strictly positive.  That positivity of
gap vectors is also necessary in order for the Hasse diagram of the
poset $\ptm(2)$ of any pattern $\tau$ to consist of $m-1$ such cycles,
by Lemma~\ref{lemma-homology}.

Now consider $\ptm(3)$. The positivity of the interior of the gap
vector of $\bs'$ implies that each pair of adjacent letters on the
first line of the stack for $\tau$ contains a letter occurring on no
other line. Hence the first line is a saturated chain $C$, starting at
the bottom element.  Only the $m$ letters $C$ shares with the second
line can have degree greater than two, since every other letter occurs
only on the first line.  The last letter of $C$ covers two letters in
the topmost cycle of $\ptm(2)$. These remain covers when the third
line is added, since that line does not contain the last letter. It is
covered by just one letter, since the first line ends there and only
the second line continues above it.

In $\ppm(3)$, a vertex covering two elements and being covered by one
vertex lies on the top horizontal row.  A saturated chain from the
bottom to such a vertex must pass through at least one vertex of
degree greater than two on each of the $m$ rows. Our chain $C$ can
contain at most $m$ such vertices. Starting from the bottom, it must
therefore turn upwards at the first such vertex on each row:
continuing to the right would encounter a second one on that row. On
the top row it must stop at the first such vertex. Thus the first line
follows precisely the leftmost zig-zag in the grid.

The $m$ letters of $\shift_0$ that are shared with $\shift_1$
consequently occupy the vertices $(i+1,i)$, for $1\le i\le m$. The
\(i\)th shared letter is in place \(s'_i\) on the first line. Its
corresponding vertex is in place \(s_i\) along the first zig-zag, so
\(s'_i=s_i\) for every \(i\).

Return now to $\ppm(2)$ and number the $m-1$ cycles from bottom to top
in each of the two posets.  An isomorphism preserves this order, so it
matches their \(i\)th cycles. Between the $i$th and $(i+1)$st shared
letters, the two sides of the $i$-th cycle have $s'_{i+1}-s'_i-1$ and
$p'_{i+1}-p'_i-1$ interior elements, respectively. The corresponding
numbers for $\pi$ are $s_{i+1}-s_i-1$ and $p_{i+1}-p_i-1$. Since the
suffix values already agree, equality of the total numbers of interior
elements in the matched cycles gives
$p'_{i+1}-p'_i=p_{i+1}-p_i$. Together with $p'_1=p_1=1$, this gives
$\bp'=\bp$, completing the proof.
\end{proof}

\begin{remark}\label{rem-uniq-sgp-patt} Although there is (up to
trivial symmetries) a unique pattern $\pi$ for each \sgp\ with a given
set of parameters, there are typically more patterns with the same
enumeration because their posets are 'essentially' isomorphic to the
\sgp\ for $\pi$. As an example, the pattern $\pi=1364257$ with overlap
3 has the same enumeration as the pattern 147962358 with overlap 4.
The poset $\ppmj$ for the former is a shifted grid poset and the poset
for the latter is obtained from $\ppmj$ by attaching a chain of length
$j+1$ to the bottom element of $\ppmj$, which doesn't affect the
number of linear extensions, because patterns corresponding to shifted
grid posets have 1 in prefix and $K$ in suffix, and their posets
therefore have no branches, by Porism~\ref{porism-branches}, so the
elements of this attached chain are comparable to all the elements of
the poset and must therefore be a prefix of any linear extension.
\end{remark}

There is an example of a shifted grid poset, for the pattern 513624
with overlap 3 in~\cite[Fig. 1]{Dwyer-Elizalde-2018}. This pattern is
equivalent (by sorting prefix and suffix) to 135246, and the poset has
all parameters 0.
\begin{corollary}\label{coro-2-grid-identity} Let
$(h_1,h_2,\ldots,h_m:v_1,v_2,\ldots,v_{m-1})$ be the parameters of a
\sgp\ whose corresponding pattern $\pi$ has length $K$.  Then, setting
$v_m=0$, we have
\begin{equation}\label{eq-2-grid-identity} \sum_{i=1}^m(h_i+v_i)=K-2m.
\end{equation}
\end{corollary}
\begin{proof}
  Summing the identities in~\eqref{eq-recover}, and using $p_1=1$, $s_m=K$ and
  $v_m=0$, gives
  \[
    \sum_{i=1}^m(h_i+v_i) =s_m-p_1-(2m-1)=K-2m. \qedhere
  \]
\end{proof}
We single out a specific subclass of shifted grid posets, given by two
levels---henceforth termed the \emph{shifted two-level posets} and therefore
parametrized by a triple $(h,v,h')=(h_1,v_1,h_2)$ of non-negative
integers. These posets are straightforward to characterize, as follows.

\begin{corollary}\label{coro-2-grid-patterns} Fix $h_1,v_1,h_2\ge0$
and let $\pi$ be a $K$-pattern whose 2-packing poset for every $j\ge1$
is isomorphic to the two-level shifted grid poset of width $j$ with
parameters $(h_1,h_2:v_1)$. Then, up to trivial symmetries, $\pi_1=1$,
$\pi_2=h_1+v_1+3$, $\pi_{K-1}=h_1+2$, and $\pi_K=K$.
\end{corollary}

\begin{proof} By Theorem~\ref{thm-shifted-grid}, up to trivial
symmetries, a pattern of this kind must have $\pi_1=1$ and
$\pi_K=K$. Thus we have
$$
h_1=\pi_{K-1}-2,\qquad v_1=\pi_2-\pi_{K-1}-1.
$$
Therefore $\pi_{K-1}=h_1+2$ and $\pi_2=h_1+v_1+3$.
\end{proof}

A special case of the patterns in
Corollary~\ref{coro-2-grid-patterns}, when $K\ge5$, are those with
prefix $13$ and suffix $2K$, for which $h=v=0,h'=K-4$.  The following
proposition is essentially derived via generating functions
in~\cite[Theorem 4.2]{Elizalde-Noy-2012}, except that there the
interfix is taken to be increasing and the pattern has
$(K+1)(K+2)\ldots(2K-5)$ appended to it, and the overlap is $K-3$.
The posets in question differ only in that the one for the longer
pattern has $K-5$ more elements in its head, and thus they trivially
have the same enumeration of linear extensions.

\begin{proposition}\label{prop-Fuss-Catalan} The number of
permutations that $2$-pack $j$ occurrences of the $K$-pattern with
prefix $13$ and suffix $2K$, where $K\ge4$, is the generalized Catalan
number $\frac{1}{(K-3)j+1}\binom{(K-2)j}{j}$.
\end{proposition}

When $K=4$ the pattern is 1324, whose poset is the smallest nontrivial
\sgp, and whose enumeration is the Catalan number $C_j$.  That result
was obtained already by Liese and Remmel in~\cite{Liese-Remmel-2010};
see also \cite[Section 4.1]{Elizalde-Noy-2012}.  Interestingly, the
case $K=4$ of Proposition~\ref{prop-Fuss-Catalan}, with overlap 2, has
the same enumeration as the pattern
\mbox{$\pi^{(a)}=1(a+1)2(a+2)\ldots a(2a)$} for any $a\ge2$, with
overlap $K-2$, where $K=2a$.  That is straightforward to prove,
because the poset for $\pi^{(a)}$ is obtained from that for 1324 by
adding $a-2$ elements to the tail and to the head of the poset for
1324.

  In our forthcoming paper~\cite{BKS-patt-prob} we give an explicit
formula for the general case of two-level grid posets.

\section{Pattern cover with varying overlap size}\label{sec-vary-olap}

So far, the special cases considered involve pattern packing with a
fixed overlap $m$.  For some classes of patterns, we observe a further
decoupling that allows us to more directly `upgrade' the results
obtained for regular overlaps to arbitrary ones. In particular, let
$\pi$ be a $K$-pattern with auto-correlation at least 2. Tracking only
the occurrences of $\pi$ arising from overlaps of sizes $1$ and $2$,
we can express the vector of successive overlaps in terms of runs of
1s and 2s as
$$
\underline{1}_{r_1},\underline{2}_{q_1},\ldots,\underline{1}_{r_a},\underline{2}_{q_a}
$$
where $a\ge1$, $r_1,q_a\ge0$, and $q_i,r_{i+1}\ge1$ for $1\le i<a$ and
$\underline{x}_n$ is the word $x,\ldots,x$ of length $n$ (or the empty
word when $n=0$).

\begin{proposition}\label{prop-vary-olap} Let $\pi$ be a $K$-pattern
  of the form $1\sigma K$ with $K\ge3$. Let
  $\underline{1}_{r_1},\underline{2}_{q_1},\ldots,\underline{1}_{r_a},\underline{2}_{q_a}$
  specify all successive prescribed overlaps, disregarding potential additional
  occurrences.  Then the number of such permutations is given by the product
  $$
  N_{q_1}\cd N_{q_2}\cd\cdots\cd N_{q_a},
  $$
  where $N_q=|\cls_\pi^2(q+1)|$ for any $q\ge0$.
\end{proposition}
\begin{proof} Given a permutation $\sigma$ of this form we claim that
the letter $\ell$ of the leftmost overlap of size 1 in $\sigma$ is
greater than all letters preceding $\ell$.  Now, $\ell$ is clearly
greater than all the letters in the occurrence of $\pi$ of which
$\ell$ is the last letter.  Therefore, if this occurrence overlaps in
two letters with a preceding occurrence, $\ell$ is larger than both
letters in that overlap.  In particular, it is larger than the last
letter of that overlap, which in turn is greater than all other
letters in that preceding occurrence.  An inductive argument now shows
that~$\ell$ must be greater than all the letters preceding it in
$\sigma$.  An analogous argument shows that $\ell$ is smaller than all
letters following it in $\sigma$.

This determines which subset of letters of $\sigma$ belongs to each of
the segments between successive 1-overlaps in $\sigma$.  A segment
between two successive 1-overlaps is thus a permutation covered by
$\pi$ with overlaps 2, whose number is $N_i$ where $i$ is the number
of 2-overlaps in that segment.  Since the ordering of letters in such
a segment is independent of the ordering in any other such the proof
follows.
\end{proof}

Therefore, for $\pi=1324$, for example, which can have overlaps of
sizes 1 and 2 only, we obtain the full analogue of the enumeration by
descent set where we count permutations with occurrences of $\pi$ in
arbitrary prescribed positions, using
Proposition~\ref{prop-Fuss-Catalan}:

\begin{corollary}\label{coro-vary-olap} The number of permutations
$\sigma$ with occurrences of the consecutive pattern 1324 prescribed
by the overlaps
$\underline{1}_{r_1},\underline{2}_{q_1},\ldots,\underline{1}_{r_a},\underline{2}_{q_a}$
is
$$
C_{q_1+1}\cd C_{q_2+1}\cd\cdots\cd C_{q_a+1}.
$$ 
\end{corollary}

For non-overlapping occurrences, the picture is even simpler; see
Remark~\ref{remark-separated-components}.

\section{Open problems, questions and remarks}\label{sec-open}

\newcounter{nextnum} \newcommand{\nextnum}{ \refstepcounter{nextnum}
  \par\textbf{\thenextnum.}~ }

\nextnum Lemma~\ref{lem-dual} and Theorems~\ref{thm-cover-2},
\ref{thm-branch-is-chain} and \ref{thm-core-lattice} provide
restrictions that apply to all pattern posets.  Is it possible to give
a complete characterization?

\nextnum For which pattern-overlap combinations can we generalize
Proposition~\ref{prop-vary-olap} about varying overlaps?

\nextnum Obviously we can, in a trivial sense, compute the enumeration
for a pattern from its gap vectors, since they determine prefix and
suffix (see the paragraph preceding Theorem~\ref{thm-diff-vectors}).
For the patterns in Proposition~\ref{prop-prod-formula-trees} the
product formula is easily computed directly from the gap vectors and
the same is true of the patterns in
Proposition~\ref{prop-prod-formula-cones}.  Are there other classes of
patterns whose enumeration can be derived in reasonably simple ways
from their gap vectors?

\nextnum We have the following conjecture on the maximum number of
permutations $(m,j)$-packing a $K$-pattern as a function of $(K,m,j)$:

\begin{conjecture}\label{conj-max-numocc} For overlap $m\le K/2$, the
maximum is obtained when prefix and suffix are interleaved and
minimal, that is, for patterns with prefix $135\ldots(2m-1)$ and
suffix $246\ldots(2m)$.
\end{conjecture} It is straightforward to show that for $j$
occurrences a poset of this kind is obtained from the \sgp\ for $j$
occurrences of the pattern $13\ldots(2m-1)24\ldots(2m)$ by adding
branches of length $K-2m$ going up from the maximum element of each
occurrence.

For overlap $m=1$, the conjectured maximizing pattern has prefix 1 and
suffix 2, which, incidentally, was shown
in~\cite{elizalde-most-avoided-2013} to be the asymptotically most
\emph{avoided} consecutive pattern among those whose occurrences can
overlap in at most one place.

For $m>K/2$, the maximizing posets seem to have a structure quite
similar to the grid structure of \sgps, and in some cases they are
isomorphic to \sgps\ apart from having longer head/tail.

\nextnum\label{item-sep} A \emph{separator} is an element $z$
comparable to every element of the poset and whose removal splits the
poset into two nonempty subposets, neither of which is a chain.  The
number of linear extensions of such a poset is the product of the
numbers for each subposet.  Can we characterize patterns whose posets
$\ppmj$ have a separator for $j>m$?  It seems that there is always a
separator (namely $K$) for overlap 2 if the 2-prefix is $1a$ and the
2-suffix is $bc$, where $2<a<b<c-1<K-2$. For $135268479$ with overlap
$3$ and $j=4$ there are three separators; for $479258136$ with overlap
$6$ and $j=7$ there are five.

\nextnum The Wilf class in Proposition \ref{prop-prod-formula-trees},
with enumeration $(2j-1)\sfact{2}$: $1,3,15, 105,\ldots$ seems to
contain, for any given overlap $m>2$, a pattern of any length $K>2m$.
Can we find a uniformly described set of representatives, one for each
length?

Similar statements seem to apply to other sequences of shifted
factorials of the form $(dj-c)\sfact{d}$ and in fact many other
sequences.  The same can be said of the class whose enumeration is the
Catalan numbers.  Here is a handful of pattern-overlap combinations in
that class, whose posets only differ by the lengths of their heads and
tails:

{\verb< 1324-2<}\\
{\verb< 14253-3<}\\
{\verb< 136245-3<}\\
{\verb< 142536-4<}\\
{\verb< 1376245-3<}\\
{\verb< 369247815-5<}\\
{\verb< 162738495-7<}

What kind of operations on patterns might explain this?

\nextnum\label{item-num-seqs} Our computations for $K$ up to 12
suggest\footnote{These results are from running programs to compute the numbers
  of permutations covered by each pattern for each possible overlap between 1
  and $K-2$, and up to at least $7$ occurrences.  In these data, any two
  pattern-overlap combinations whose enumerations differ for seven occurrences
  differ already for five occurrences. Overlap $K-1$ is ignored, since that
  implies that a $K$-pattern $\pi$ is increasing or decreasing, with exactly one
  permutation of each length $n\ge K$ covered by $\pi$.  We have not tried to
  verify rigorously that our programming is correct, but comparisons with other
  methods of computing this have yielded no discrepancies.}
that there are 31,111 combinations of $K$-patterns and overlaps that
are inequivalent modulo the symmetries in
Section~\ref{sec-symmetries}, and 10,787 distinct initial enumerating
sequences for these. Of these combinations, 2,684 have the trivial
sequence $1,1,1,\ldots$ because the poset in question is a chain. Of
the remaining ones 452 have initial terms matching $(2^n)_{n\ge0}$,
385 matching $(3^n)_{n\ge0}$, 374 matching the Catalan numbers and 331
matching the sequence $\binom{3n}{n}/(2n+1)$ for~$n\ge0$. Other
sequences appear for fewer than 300 pattern-overlap combinations each,
and there are over 1,000 such combinations that have a unique
enumeration.

Here $n=j-1$ for the power sequences, and $n=j$ for the Catalan and
Fuss--Catalan sequences. Agreement of finitely many terms of course does not
establish Wilf equivalence, but it seems unlikely that any sequences in these
data that agree up to seven occurrences will split later.

Although we have not analyzed this exhaustively, it seems that for
most of these apparent Wilf classes (sets of pattern-overlap
combinations with the same enumeration) the posets in question are
isomorphic, apart from the number and placement of elements that
belong to head or tail, or else cover and are covered by a single
element, and are comparable to all elements of the poset.  Such
elements don't affect the number of linear extensions.  One exception
to this is the triple of patterns 125634 and 163425 with overlap 2 and
167359248 with overlap 3 (see
Figure~\ref{fig-125634-163425-167359248}), and it would be interesting
to find further such nontrivial Wilf classes.

\nextnum As mentioned in the footnote to item~\ref{item-num-seqs}, if
two pattern-overlap combinations have different enumeration for 1-7
occurrences, for patterns of length up to 12, then they have different
enumeration already for five occurrences.  Is it possible that 5 is a
cutoff for any pair of patterns, of arbitrary lengths, or that there
is at least a small such bound that depends only on pattern length?

\pagebreak
\nextnum The Wilf class with enumeration sequence $2^{j-1}$ contains,
among an apparent infinity of pattern-overlap combinations, 132546-2,
1324657-2 and 123584679-4.  All can be explained by straightforward
arguments showing that on each pair of successive lines of the stack
there is a pair of elements that are incomparable to each other, but
otherwise all elements are comparable.  All pattern posets we have
seen with this enumeration consist of cycles where two such
incomparable elements are the only deviation from the poset being a
chain.  Similar posets seem to exist for the sequence $k^n$ for any
positive $k$, and for an infinite collection of pattern-overlap
combinations in each case.

It seems plausible that understanding some of these Wilf equivalences
will lead to greater understanding than just for each class.  For
anybody interested in exploring these and other Wilf equivalences we
are happy to share both our data and the python code we have used to
compute the enumeration.

\nextnum In our enumeration data mentioned in item~\ref{item-num-seqs}
above, just over a third of the sequences have a product formula. Many
of these formulas have factors that are single binomial coefficients,
such as in Propositions~\ref{prop-prod-formula-trees}
and~\ref{prop-prod-formula-cones}. This is not true in general.
The enumeration for 135289467 with overlap 3, for example, has a
product formula according to our data, but the successive factors are
not single binomial coefficients.  However, the factors are polynomial
starting from $j=3$, namely
$121\binom{j-3}{0}+168\binom{j-3}{1}+72\binom{j-3}{2}$, in the
binomial basis.  All the product formulas in our data have such
polynomial factors.

\nextnum A curiosity in the enumeration for 14325 with overlap 2 is that all the
numbers are congruent to 6 mod 10, except when $j=2^k-1$, those are congruent to
1.  This holds at least up to 23 occurrences and we see no obvious reason for
that looking at the prime factorizations:

{\verb>  1 1:>}\\%
{\verb>  2 6: 2 3>}\\%
{\verb>  3 71: 71>}\\%
{\verb>  4 1266: 2 3 211>}\\%
{\verb>  5 30206: 2 11 1373>}\\%
{\verb>  6 902796: 2 2 3 23 3271>}\\%
{\verb>  7 32420011: 13 131 19037>}\\%
{\verb>  8 1359292626: 2 3 3 75516257>}\\%
{\verb>  9 65164480466: 2 1123 29013571>}\\%
{\verb> 10 3515569641156: 2 2 3 521 4729 118907>}\\%
{\verb> 11 210779736073446: 2 3 3907 8991542363>}\\%
{\verb> 12 13903319821066836: 2 2 3 131 523 16910805031>}\\%
{\verb> 13 1000559812125494076: 2 2 3 83379984343791173>}\\%
{\verb> 14 78012524487061315416: 2 2 2 3 773 4205073549324133>}\\%
{\verb> 15 6550837823204594551731: 3 11 19 8406637 1242816504869>}\\%
{\verb> 16 589404446176366002280146: 2 3 1796953 54667024881968347>}

Another apparent regularity in this is that 4 seems to divide all but
those for which $j=2^i+2^p-1$ where $0\le p\le i$, such as 8,9,11,15.

\subsection*{Acknowledgments} We thank Sergi Elizalde for drawing our
attention to \cite{Dwyer-Elizalde-2018, elizalde-most-avoided-2013}.
Although we only cite~\cite{Dotsenko-Khoroshkin-2013} and
~\cite{Elizalde-Noy-2012} in relation to the posets that play a
central role here, we first became aware of this connection when we
saw the reference to the unpublished work of Ran Pan on the {\sc
OEIS}~\cite{OEIS} page for sequence A274644, one of the sequences we
generated early on in our work.

\end{document}